\documentclass[leqno]{article}

\usepackage[frenchb,english]{babel}
\usepackage[utf8]{inputenc}
\usepackage{amsmath}
\usepackage{amssymb}
\usepackage{amsfonts}
\usepackage{enumerate}
\usepackage{vmargin}
\usepackage[all]{xy}
\usepackage{mathrsfs}
\usepackage{mathtools}
\usepackage{lmodern}
\usepackage{slashed}
\usepackage[colorlinks=true,linkcolor=blue,pagebackref=true]{hyperref}%
\setmarginsrb{3cm}{3cm}{3.5cm}{3cm}{0cm}{0cm}{1.5cm}{3cm}
\usepackage{comment}

\usepackage{accents}

\newlength{\dhatheight}

\newcommand{\C}{\mathrm{C}}
\newcommand{\D}{\mathrm{D}}
\newcommand{\E}{\ensuremath{\mathbb{E}}}

\newcommand{\B}{\mathrm{B}} 

\let\L\relax 
\newcommand{\L}{\mathrm{L}}
\let\O\relax 
\newcommand{\J}{\mathcal{J}} 
\newcommand{\O}{\ensuremath{\mathbb{O}}}
\let\P\relax 
\newcommand{\P}{\mathsf{P}} 

\newcommand{\scr}{\mathscr}

\newcommand{\M}{\mathrm{M}}

\let\cal\relax
\newcommand{\cal}{\mathcal}

\newcommand{\R}{\ensuremath{\mathbb{R}}}
\newcommand{\T}{\ensuremath{\mathbb{T}}}

\newcommand{\W}{\mathrm{W}}
\newcommand{\sgn}{\mathrm{sgn}}
 
\newcommand{\Id}{\mathrm{Id}}
\newcommand{\PL}{\mathrm{PL}}
\newcommand{\Asym}{\mathrm{Asym}} 
\newcommand{\la}{\langle}
\newcommand{\ra}{\rangle}

\renewcommand{\leq}{\ensuremath{\leqslant}}
\renewcommand{\geq}{\ensuremath{\geqslant}}
\newcommand{\qed}{\hfill \vrule height6pt  width6pt depth0pt}
\newcommand{\bnorm}[1]{ \big\| #1  \big\|}

\newcommand{\norm}[1]{\left\Vert#1\right\Vert}

\newcommand{\co}{\colon}

\newcommand{\ot}{\otimes}
\newcommand{\ovl}{\overline}
\newcommand{\otvn}{\ovl\ot}

\newcommand{\dec}{\mathrm{dec}}

\newcommand{\cb}{\mathrm{cb}}
\newcommand{\TRO}{\mathrm{TRO}}

\let\i\relax 
\newcommand{\i}{\mathrm{i}}
\let\k\relax 
\newcommand{\k}{\mathrm{k}} 

\newcommand{\ov}{\overset}

\newcommand{\JW}{\mathrm{JW}}
\newcommand{\JC}{\mathrm{JC}}
\newcommand{\JB}{\mathrm{JB}}

\newcommand{\JBW}{\mathrm{JBW}}

\newcommand{\w}{\mathrm{w}} 

\renewcommand{\d}{\mathop{}\mathopen{}\mathrm{d}} 
\newcommand{\e}{\mathrm{e}} 
\renewcommand{\d}{\mathop{}\mathopen{}\mathrm{d}}

\DeclareMathOperator{\Span}{span} 
\DeclareMathOperator{\tr}{Tr} 
\DeclareMathOperator{\Tr}{Tr} 
\let\ker\relax 
\DeclareMathOperator{\ker}{Ker} 
\DeclareMathOperator{\Ran}{Ran} 
\DeclareMathOperator{\dom}{dom} 
\let\Re\relax 
\DeclareMathOperator{\Re}{Re} 
\let\Im\relax 
\DeclareMathOperator{\Im}{Im} 

\newtheorem{thm}{Theorem}[section]
\newtheorem{defi}[thm]{Definition}
\newtheorem{prop}[thm]{Proposition}

\newtheorem{cor}[thm]{Corollary}
\newtheorem{lemma}[thm]{Lemma}

\newtheorem{remark}[thm]{Remark}
\newtheorem{example}[thm]{Example}

\newenvironment{proof}[1][]{\noindent {\it Proof #1} : }{\hbox{~}\qed
\smallskip
}

\usepackage{tocloft}
\numberwithin{equation}{section}
\usepackage[nottoc,notlot,notlof]{tocbibind}

\let\OLDthebibliography\thebibliography
\renewcommand\thebibliography[1]{
  \OLDthebibliography{#1}
  \setlength{\parskip}{0pt}
  \setlength{\itemsep}{0pt plus 0.3ex}
}

\newcommand\reallywidehat[1]{\arraycolsep=0pt\relax%
\begin{array}{c}
\stretchto{
  \scaleto{
    \scalerel*[\widthof{\ensuremath{#1}}]{\kern-.5pt\bigwedge\kern-.5pt}
    {\rule[-\textheight/2]{1ex}{\textheight}} 
  }{\textheight} %
}{0.5ex}\\           
#1\\                 
\rule{-1ex}{0ex}
\end{array}
}

\begin{document}
\selectlanguage{english}
\title{\bfseries{A modular theory of contractive projections on Banach spaces with applications to noncommutative $\L^p$-spaces}}
\date{}
\author{\bfseries{C\'edric Arhancet}}%
\maketitle

\begin{abstract}
We develop a differential-geometric approach to contractive projections on complex Banach spaces based on Levi forms of plurisubharmonic functions. We show that if a projection $P$ contracts a $\mathrm{C}^2$ plurisubharmonic function $\Phi$, then, at any point of its range, the Levi form of $\Phi$ defines a tangent Hilbertian geometry for which $P$ becomes an orthogonal projection. When the range is finite-dimensional and the relevant Levi forms are nondegenerate, comparing the Levi geometries associated with different points leads naturally to positive relative operators and associated one-parameter groups. We apply this framework to tracial noncommutative $\mathrm{L}^4$-spaces. Exploiting the cubic form of the duality mapping, we characterize contractive projections by a Jordan triple orthogonality property and prove that a projection on $\mathrm{L}^4(\mathcal{M})$ is $2$-contractive if and only if it is completely contractive. When the ambient von Neumann algebra is finite-dimensional, we identify the relative operator associated with the left and right tangent forms on the range with a restriction of an inverse square root of a modular operator on a linking von Neumann algebra. This entails that the range of any $2$-contractive projection is completely isometric to a rectangular noncommutative $\mathrm{L}^4$-space associated with a $\mathrm{W}^*$-ternary ring of operators. As a consequence, every such projection is contractively decomposable. Finally, for every $1 < p < \infty$ with $p \neq 2$, we prove that every finite-rank $2$-contractive projection on a tracial noncommutative $\mathrm{L}^p$-space is completely contractive. For $p > 2$, we obtain a weighted ternary description of its range in terms of a finite-dimensional $\mathrm{W}^*$-$\mathrm{TRO}$.
\end{abstract}


\makeatletter
 \renewcommand{\@makefntext}[1]{#1}
 \makeatother

{
  \hypersetup{linkcolor=blue}
 \tableofcontents
}

\section{Introduction}
\label{sec:Introduction}

The study of contractive projections and contractively complemented subspaces is one of the classical themes of the isometric theory of Banach spaces. Given a Banach space $X$ and a closed subspace $Y\subset X$, the basic problem is to decide whether there exists a projection $P\co X\to X$ such that
\begin{equation}
P^2=P,\qquad \Ran P=Y,\qquad \norm{P}_{X \to X}
\leq 1,
\end{equation}
and, when such a projection exists, to understand both the intrinsic geometry of $Y$ and the structure of $P$. These two questions are closely related. Indeed, on a smooth Banach space, a subspace can be the range of at most one projection of norm one, see \cite{CoS70}. Since noncommutative $\L^p$-spaces are smooth for $1<p<\infty$, the classification of their contractively complemented subspaces is, in this range, essentially equivalent to the classification of their contractive projections.

For classical $\L^p$-spaces, this problem admits a rigid answer. By the work of Ando \cite{And66}, completed in greater generality by Bernau and Lacey \cite{BeL74}, a contractively complemented subspace of a classical $\L^p$-space, for $1<p<\infty$ and $p\neq2$, is itself isometric to an $\L^p$-space. Moreover, the corresponding projections are described in terms of weighted conditional expectations. Thus, in the commutative setting, norm-one complementation essentially preserves the category of $\L^p$-spaces.

The noncommutative situation is substantially richer. The first obstruction already occurs for Schatten spaces. If $\sigma\co S^p\to S^p$ denotes the transpose map, then
\begin{equation}
P=\frac12(\Id+\sigma)
\end{equation}
is a contractive projection whose range is the space of symmetric matrices. This range is not, in general, a Schatten space. More generally, the seminal work of Arazy and Friedman \cite{ArF92} gives a complete description of contractively complemented subspaces of Schatten spaces in terms of the $\L^p$-versions of Cartan factors of types I--IV. In particular, rectangular operator spaces, symmetric and antisymmetric operator spaces, and spin factors all occur naturally. The same work explicitly raises the problem of describing the contractively complemented subspaces of general noncommutative $\L^p$-spaces.

This problem remains one of the basic structural questions concerning noncommutative $\L^p$-spaces. Its difficulty is not merely technical. The examples above indicate that the range of a contractive projection need not carry an associative structure and that Jordan and ternary structures are unavoidable. At the endpoint $p=1$, this phenomenon is reflected by the appearance of preduals of $\JBW^*$-triples, see \cite{FrB85b}. On the other hand, once the canonical operator space structure of a noncommutative $\L^p$-space is taken into account, rectangular structures become especially natural. In the Schatten setting, Le Merdy, Ricard and Roydor \cite{LRR09} showed that completely $1$-complemented subspaces are precisely, up to the appropriate direct sums, spaces of the form $S^p(H,K)$. These are the $\L^p$-counterparts of ternary rings of operators.

This suggests that the problem of contractive complementation has several levels of rigidity. A scalar norm-one condition may detect a Jordan structure, whereas matricial contractivity should remember the non-symmetrized ternary product. This principle is consistent with a number of results obtained in recent years. The range of a $2$-positive contractive projection on an arbitrary noncommutative $\L^p$-space is completely order and completely isometrically isomorphic to another noncommutative $\L^p$-space \cite{ArR24}; in the $\sigma$-finite case the projection can be viewed as a conditional expectation modulo a symmetric two-sided change of density. Positive contractive projections lead more generally to Jordan conditional expectations and to nonassociative $\L^p$-spaces associated with $\JW^*$-algebras, see \cite{Arh24a,Arh26b}. In a different direction, contractively decomposable projections are naturally related to $\W^*$-ternary rings of operators and to rectangular noncommutative $\L^p$-spaces \cite{Arh24b}. Recent results give a more intrinsic description of their ranges in terms of rectangular $\L^p$-spaces \cite{Bof26,Arh26c}.

The general contractive case, however, is much less understood. In particular, there is no general mechanism which starts only from the metric inequality
\begin{equation}
\norm{P(x)}_{\L^p(\cal{M})}
\leq\norm{x}_{\L^p(\cal{M})}
\end{equation}
and recovers the hidden Jordan or ternary structure of the range. The main purpose of the present paper is to introduce such a mechanism. Our approach is based on the differential geometry of the norm, and more precisely on the Levi forms of plurisubharmonic functions on complex Banach spaces.

Differential methods have appeared before in the study of norm-one projections. In particular, second derivatives of norms were successfully used in \cite{LRG07} to study contractive complementation in vector-valued spaces. Here we take a different point of view. Instead of considering only the real Hessian of the norm, we exploit the complex structure and the associated Levi form. Levi forms on complex Banach spaces were investigated in \cite{Lig76}. The key observation is that a contractive projection produces, at any point of its range, a canonical Hilbertian geometry.

More precisely, let $\Omega$ be an open subset of a complex Banach space $X$, let $\Phi\co\Omega\to\R$ be a $\C^2$ plurisubharmonic function and suppose that a projection $P\co X\to X$ satisfies $P(\Omega) \subset \Omega$ and
\begin{equation}
\Phi(Px)
\leq \Phi(x),\quad x \in \Omega.
\end{equation}
If $a\in\Omega\cap\Ran P$, then the Levi form $\scr L_{\Phi,a}$ satisfies
\begin{equation}
\scr L_{\Phi,a}(Pz)
\leq \scr L_{\Phi,a}(z),\quad z \in X.
\end{equation}
After quotienting by the null space of the Levi form and completing, one obtains a Hilbert space on which $P$ becomes an orthogonal projection. Consequently, although the complement $\Id-P$ need not be contractive for the original Banach norm, both $P$ and $\Id-P$ are contractive for every tangent Levi geometry based at a point of $\Ran P$. Equivalently, the involution
\begin{equation}
2P-\Id
\end{equation}
becomes a unitary involution on each of these tangent Hilbert spaces.

Thus a single Banach space projection gives rise to a whole family of Hilbertian geometries indexed by the points of its range. At first order, contractivity is governed by duality mappings and norming functionals, in accordance with the classical characterization of Calvert \cite{Cal75}. The first-order condition identifies the directions annihilated by the corresponding norming functional. Differentiating this condition along the range yields the orthogonality relations for the Levi form. The role of the second-order viewpoint is therefore not to provide an independent contractivity condition, but to package these relations into a Hilbertian geometry. When the Levi form is positive, the decomposition
\begin{equation}
X
=\Ran P \oplus \ker P
\end{equation}
becomes orthogonal for this tangent geometry, making relative positive operators and their spectral calculus available.

Comparing two such Hilbertian geometries leads naturally to relative positive operators. Suppose first that $\Ran P$ is finite-dimensional and that the restrictions to $\Ran P$ of the Levi forms based at two points $a,b \in \Omega \cap \Ran P$ are nondegenerate. If $\langle\cdot,\cdot\rangle_a$ and $\langle\cdot,\cdot\rangle_b$ denote the corresponding inner products on $\Ran P$, there exists a unique positive invertible operator $\Delta_{b|a}\co\Ran P\to\Ran P$ characterized by
\begin{equation}
\la x,y \ra_b
=\la x,\Delta_{b|a}y \ra_a,
\quad x,y \in \Ran P.
\end{equation}
It is selfadjoint for the first Hilbertian structure, and its imaginary powers form a one-parameter group of isometries for both tangent Hilbertian geometries.

A stronger phenomenon appears when the relative operator can be constructed on a common ambient Hilbert completion. More precisely, if the Levi forms at $a$ and $b$ are nondegenerate on $X$ and define equivalent Hilbertian norms, the corresponding ambient relative operator $\widehat{\Delta}_{b|a}$ commutes with the orthogonal extension of $P$. Consequently, the closures of $\Ran P$ and $\ker P$ are reducing subspaces for $\widehat{\Delta}_{b|a}$, and
\begin{equation}
P\widehat{\Delta}_{b|a}^{\i t}
=
\widehat{\Delta}_{b|a}^{\i t}P,
\quad t\in\R.
\end{equation}
Thus the relative one-parameter group dynamically preserves the decomposition determined by the projection. This is the origin of the modular terminology used in this paper. The appearance of positive relative operators and their imaginary powers is reminiscent of modular constructions associated with standard subspaces, see for instance \cite{CoL25}. The important difference is that here the modular object is not given in advance: it is extracted directly from a norm-one projection and from the local differential geometry of the ambient Banach space.

The case of noncommutative $\L^4$-spaces is particularly well suited to this approach. There are two reasons for this. First, the fourth power of the norm has an explicit Levi form. If
\begin{equation}
\Phi_4(x)=\norm{x}_{\L^4(\cal{M})}^4,
\end{equation}
then
\begin{equation}
\label{intro-Levi-L4}
\scr L_{\Phi_4,a}(z)
=
2\bigl(\norm{a^*z}_{\L^2(\cal{M})}^2+\norm{za^*}_{\L^2(\cal{M})}^2\bigr).
\end{equation}
Thus the tangent Hilbert geometry separates naturally into a left part and a right part. This is precisely the two-sided structure which is needed in order to recover rectangular operator spaces.

Second, the duality mapping of $\L^4(\cal M)$ is cubic:
\begin{equation}
\label{intro-duality-L4}
\frak j_4(x)=xx^*x.
\end{equation}
Polarization of this cubic mapping produces the Jordan triple product
\begin{equation}
\label{intro-Jordan-triple}
\{x,y,z\}
=
\frac12(xy^*z+zy^*x).
\end{equation}
The special role of the exponent $4$ in noncommutative $\L^p$-geometry also appears in the theory of $2$-isometries, where selfpolar forms can be used to recover operator-algebraic structure; see \cite{JRS05}. Our situation is different: we start with a projection rather than an isometry, and no positivity or matricial assumption is imposed initially. Nevertheless, the cubic duality mapping allows us to extract a ternary structure directly from scalar contractivity.

Indeed, let $P\co\L^4(\cal M)\to\L^4(\cal M)$ be a contractive projection and set
\begin{equation}
Y=\Ran P,\qquad
Y^\sharp=\Ran P^\dagger,
\end{equation}
where $P^\dagger$ denotes the adjoint projection under the $\L^4$--$\L^{\frac{4}{3}}$ trace duality. We prove that
\begin{equation}
\label{intro-dual-range}
Y^\sharp
=
\frak j_4(Y)
=
\{xx^*x:x\in Y\}.
\end{equation}
In particular, the nonlinear image $\frak j_4(Y)$ is a linear subspace of $\L^{\frac{4}{3}}(\cal M)$. Polarization then yields
\begin{equation}
\label{intro-Jordan-dual-range}
\{Y,Y,Y\}
\subset Y^\sharp.
\end{equation}
Thus scalar contractivity forces Jordan triple relations taking values in the dual range. At this stage, however, the product is not internal to $Y$, so these relations do not yet equip $Y$ with a Jordan triple structure. Recovering an internal ternary product requires the additional deweighting argument developed below.

The first nontrivial matrix level yields a stronger conclusion. One of the main results of this paper is that for projections on $\L^4$ the level $2$ condition already contains all matrix levels. More precisely, we prove that
\begin{equation}
\label{intro-main-equivalence}
P\text{ is $2$-contractive}
\quad\Longleftrightarrow\quad
YY^*Y\subset Y^\sharp
\quad\Longleftrightarrow\quad
P\text{ is completely contractive}.
\end{equation}
Hence, for projections on noncommutative $\L^4$-spaces, complete contractivity can be tested on $2\times2$ matrices. This collapse of an infinite family of matricial norm inequalities to a single matrix level is one of the main rigidity phenomena established here.

By trace-adjoint duality, the same equivalence holds for projections on $\L^{\frac43}(\cal M)$, see Corollary~\ref{cor-dual-two-contractive-L43}.

For Schatten spaces, the equivalence between $2$-contractive and completely contractive complementation follows from \cite[Theorem~1.1]{LRR09}. The contribution of \eqref{intro-main-equivalence} is a direct duality-mapping proof beyond the Schatten setting, together with the intrinsic criterion $YY^*Y\subset Y^\sharp$. Theorem~\ref{thm-Haagerup-even-finite-level} extends this equivalence to Haagerup $\L^4$-spaces associated with arbitrary von Neumann algebras. We refer also to \cite{ArL26}, \cite{Bof25} and \cite{You83} for related results.

The mechanism behind \eqref{intro-main-equivalence} is transparent. Scalar contractivity gives only the symmetrized expression \eqref{intro-Jordan-triple}. By applying the scalar result to the amplification $\Id_{S_2^4}\ot P$ and choosing suitable $2\times2$ block matrices, one can separate the two outer variables and isolate the non-symmetrized product $xy^*z$. Once
\begin{equation}
xy^*z\in Y^\sharp,\quad x,y,z\in Y,
\end{equation}
is known, the same property at every matrix level follows entrywise, and the cubic characterization of contractivity implies complete contractivity. In this sense, scalar contractivity detects Jordan triple geometry, whereas $2$-contractivity detects TRO geometry.

The same polynomial principle gives a finite-level criterion at every even exponent. If $p=2m$ for some integer $m \geq 2$ and
\begin{equation}
\label{intro-even-level}
N_m
\ov{\mathrm{def}}{=}
\left\lceil\frac{2m}{3}\right\rceil,
\end{equation}
we prove that, for an arbitrary von Neumann algebra $\cal M$, an $N_m$-contractive projection on its Haagerup space $\L^{2m}(\cal M)$ is completely contractive.

The proof extracts an arbitrary alternating product of length $2m-1$ as a torus Fourier coefficient of the power duality mapping at the matrix level $N_m$. Trace-adjoint duality gives the analogous result at the conjugate exponent $\frac{2m}{2m-1}$. In particular, level $2$ suffices for $p=6$ and $p=\frac{6}{5}$. We do not claim that the level in \eqref{intro-even-level} is optimal. Indeed, for Schatten spaces, level $2$ is already known to suffice for any $1 <p < \infty$, $p \neq 2$, by \cite[Theorem~1.1]{LRR09}.

The remaining difficulty is that the ternary products still belong a priori to the dual range $Y^\sharp$, not to $Y$. In the finite-dimensional setting, the unsymmetrized ternary relation yields natural left and right multiplication maps from $Y$ into $Y^\sharp$. Choose a basis $x_1,\ldots,x_k$ of $Y$ and define
\begin{equation}
\label{intro-a-b}
a
=\sum_{j=1}^k x_jx_j^*,
\qquad
b
=\sum_{j=1}^k x_j^*x_j.
\end{equation}
We prove the exact two-sided relation
\begin{equation}
\label{intro-density-dual-range}
aY=Y^\sharp=Yb.
\end{equation}
This algebraic relation has a natural tangent-Hilbert interpretation. Indeed, if $\frak H^\ell$ and $\frak H^r$ denote the sums of the left and right tangent forms associated with the points $x_1,\ldots,x_k$, then
\[
\frak H^\ell(x,y)=\Tr(y^*ax),
\qquad
\frak H^r(x,y)=\Tr(y^*xb),
\]
and
\[
\sum_{j=1}^k\scr L_{\Phi_4,x_j}(x,y)
=
2\bigl(\frak H^\ell(x,y)+\frak H^r(x,y)\bigr).
\]
The multiplication maps $y\mapsto ay$ and $y\mapsto yb$ are precisely the Riesz representatives of these two Hilbertian structures in the common dual range $Y^\sharp$. Their relative operator is therefore
\begin{equation}
\label{intro-relative-operator}
\Delta(y)=a^{-1}yb.
\end{equation}
Moreover, we have
\begin{equation}
\label{mod-intro}
\Delta^{\i t}(Y)
=Y,\quad t \in \R.
\end{equation}
In fact, these maps are complete isometries for the inherited operator space structure, and the corresponding left-right unitary action commutes with the projection. The relation \eqref{mod-intro} admits a modular interpretation. Let $e=s_\ell(P)$ and $f=s_r(P)$ and consider the linking von Neumann algebra
\begin{equation}
\cal R(e\cal Mf)
=
\begin{bmatrix}
e\cal Me&e\cal Mf\\
f\cal Me&f\cal Mf
\end{bmatrix}
\end{equation}
together with the density
\begin{equation}
D=
\begin{bmatrix}
a^2&0\\
0&b^2
\end{bmatrix}.
\end{equation}
Then the restriction of the modular automorphism group associated with $D$ to the upper-right corner is
\begin{equation}
\sigma_t^D(y)=a^{2\i t}yb^{-2\i t},
\end{equation}
and the relative group coming from the two Levi geometries satisfies
\begin{equation}
\Delta^{\i t}(y)=\sigma_{-\frac t2}^D(y).
\end{equation}
Thus the ``modular group'' obtained abstractly by comparing tangent Hilbert geometries is exactly the restriction of a genuine von Neumann algebra modular group on the appropriate linking algebra.

This identification leads to a two-sided deweighting of the range. Namely, we set
\begin{equation}
\label{intro-deweighting}
\cal V
=
a^{-\frac14}Yb^{-\frac14}.
\end{equation}
Using the invariance under the relative modular group together with the ternary relation $YY^*Y\subset Y^\sharp$, we prove that $\cal V$ is a finite-dimensional $\W^*$-$\TRO$. Hence
\begin{equation}
Y=a^{\frac14}\cal Vb^{\frac14}.
\end{equation}
The quarter powers in this formula are not an arbitrary change of variables: they are dictated simultaneously by the exponent $p=4$, by the Levi geometry and by the modular structure of the linking algebra.

We further show that this algebraic deweighting is compatible with the operator space norms. More precisely, after equipping the linking algebra of $\cal V$ with the natural faithful state determined by $D$, the preceding identification induces a surjective complete isometry
\begin{equation}
\L^4(\cal{V},\varphi)
\to \Ran P.
\end{equation}
Consequently, when the ambient von Neumann algebra is finite-dimensional, the range of any $2$-contractive projection on its noncommutative $\L^4$-space is completely isometric to a rectangular noncommutative $\L^4$-space associated with a $\W^*$-$\TRO$. Moreover, we obtain an explicit formula for the original projection in terms of a state-preserving conditional expectation on the linking algebra. More precisely, if $\E_{12}$ denotes the upper-right corner of this conditional expectation, then
\begin{equation*}
P(x)
=a^{\frac{1}{4}}\E_{12}\big(a^{-\frac{1}{4}}exf b^{-\frac{1}{4}}\big)b^{\frac{1}{4}},
\quad x \in \L^4(\cal M).
\end{equation*}
This realizes $P$ as the off-diagonal part of a completely positive contractive projection on a matrix amplification and proves directly that $P$ is contractively decomposable.

The preceding results fit naturally into the broader program of understanding which nonassociative structures are selected by norm-one projections on noncommutative $\L^p$-spaces. They also provide a new route to these structures. The existing classifications for Schatten spaces rely on the fine structure theory of Cartan factors, while the results for positive and decomposable projections use, in one form or another, order, lifting or change-of-density arguments. Our approach starts instead from the bare contractive inequality and differentiates it. The Jordan triple relations arise from the first-order geometry encoded by the duality mapping, while the TRO structure is recovered from the matricial ternary relations and the subsequent deweighting. The modular group arises from comparing the left and right tangent Hilbertian structures associated with the range.

This viewpoint suggests a useful general principle. The correct object attached to a contractively complemented subspace may not be visible directly inside the ambient $\L^p$-space. At first it can appear in a dual range or in a tangent Hilbert space. A suitable relative modular operator can then provide the change of density which transports this infinitesimal structure back to the original range. The case $p=4$ treated here gives a concrete realization of this principle and provides evidence that differential and modular methods may be useful for the general problem of classifying contractively complemented subspaces of noncommutative $\L^p$-spaces.

\paragraph{Organization of the paper}
In Section~\ref{sec-preliminaries}, we first recall the necessary background on noncommutative $\L^p$-spaces, ternary rings of operators, Jordan structures and Levi forms on complex Banach spaces. In Section~\ref{sec-Duality}, we study duality mappings associated with gauge functions and their interaction with contractive projections. This provides the first-order part of our approach. In Section~\ref{sec-tangent-Hilbert-geometry}, we introduce the tangent Hilbert geometry associated with a plurisubharmonic function contracted by a projection and establish the corresponding orthogonality, bicontractivity and intertwining properties. In finite dimensions, we compare different tangent Hilbert structures by means of relative positive operators and their imaginary powers.

In Section~\ref{sec-the-case-NC}, we specialize the general theory to noncommutative $\L^4$-spaces. We compute explicitly the Levi form of the fourth power of the norm and relate it to the Peirce decomposition associated with partial isometries. In Section~\ref{sec-L4}, we use the cubic duality mapping to characterize scalar contractivity in terms of the Jordan triple product and to prove the equivalence between $2$-contractivity and complete contractivity. We then return to the tangent Hilbert geometry. When the ambient von Neumann algebra is finite-dimensional, the left and right tangent forms associated with a basis of the range define two Hilbertian structures whose relative operator is identified with the restriction of an inverse square root of a modular operator on a linking von Neumann algebra.

This leads to the two-sided deweighting \eqref{intro-deweighting}, the $\W^*$-$\TRO$ structure of the deweighted range and the complete isometric identification with a rectangular noncommutative $\L^4$-space. In Section~\ref{sec-even-exponents}, we prove the finite matrix-level criterion \eqref{intro-even-level} for all even exponents and its trace-dual counterpart.

Finally, in Section~\ref{sec-weighted-tangent-general-p}, we return to arbitrary exponents $p>2$. We introduce weighted left and right tangent geometries and associated Gram densities for finite-rank $2$-contractive projections, identify the resulting relative operators with modular dynamics, and obtain a weighted $\W^*$-$\TRO$ description of the range. This yields, by duality, complete contractivity of every finite-rank $2$-contractive projection on a tracial noncommutative $\L^p$-space for $1<p<\infty$, $p\neq2$.

\section{Preliminaries}
\label{sec-preliminaries}Unless explicitly stated otherwise, $\cal M$ is a semifinite von Neumann algebra equipped with a normal semifinite faithful trace $\tau$. For $1\leq p<\infty$, we denote by $\L^p(\cal M)$ the space of all $\tau$-measurable operators $x$ affiliated with $\cal M$ such that $\tau(|x|^p)<\infty$, equipped with the norm
\begin{equation}
\label{Def-norm-Lp}
\norm{x}_{\L^p(\mathcal{M})}
\ov{\mathrm{def}}{=} \bigl(\tau(\vert x\vert^p)\bigr)^{\frac{1}{p}}.
\end{equation}
Here $\vert x \vert \ov{\mathrm{def}}{=}(x^*x)^{\frac{1}{2}}$ denotes the modulus of $x$. We also set $\L^{\infty}(\mathcal{M}) \ov{\mathrm{def}}{=} \mathcal{M}$, equipped with its operator norm. The subspace $\L^p(\mathcal{M}) \cap \mathcal{M}$ is dense in $\L^p(\mathcal{M})$. Unless explicitly stated otherwise, we use the sesquilinear trace duality
\[
\la x,y\ra_{\L^p(\cal M),\L^{p^*}(\cal M)}
\ov{\mathrm{def}}{=}\tau(y^*x),
\quad \frac1p+\frac1{p^*}=1.
\]
In the remainder of the paper, we write $\tr$ for $\tau$. See \cite{Ray03}, \cite{PiX03}, \cite{Pis98} for more information on noncommutative $\L^{p}$-spaces and to the books \cite{Bla06}, \cite{KaR97a}, \cite{KaR97b} and \cite{Tak03} for the topic of von Neumann algebras.

For $n\geq1$, we equip $\M_n\otvn\cal M$ with the trace $\Tr_n\ot\tau$ and use the canonical identification $
\L^p(\M_n\otvn\cal M)
=S_n^p[\L^p(\cal M)]$. If $T\co\L^p(\cal M)\to\L^p(\cal M)$ is bounded, its $n$th amplification is the entrywise map
\[
T_n
\ov{\mathrm{def}}{=} \Id_{S_n^p} \ot T,
\quad
T_n([x_{ij}])
=[T(x_{ij})].
\]
We say that $T$ is $n$-contractive if $\norm{T_n}\leq1$, and completely contractive if it is $n$-contractive for any integer $n \geq 1$. We equip $\L^p(\cal{M})$ with its natural operator space structure \cite{Pis98}. With this structure, the canonical identification $
\L^p(\M_n\otvn\cal M)
=
S_n^p[\L^p(\cal M)]$ is completely isometric. Consequently, the completely bounded norm of a bounded map $T$ between noncommutative $\L^p$-spaces is
\[
\norm{T}_{\cb}
=
\sup_{n \geq 1}
\norm{\Id_{S_n^p}\ot T}.
\]
Likewise, a linear map is a complete isometry if and only if all its matrix amplifications are isometries. 

We also recall the notion of decomposability used below. Let $\cal M$ and $\cal N$ be von Neumann algebras and let $T \co \L^p(\cal M) \to \L^p(\cal N)$ be a bounded map. Put
\[
T^\circ(x)\ov{\mathrm{def}}{=} T(x^*)^*,
\quad x\in\L^p(\cal M).
\]
The map $T$ is decomposable if there exist bounded linear maps $v_1,v_2 \co \L^p(\cal M) \to \L^p(\cal N)$ such that the block map $
\begin{bmatrix}
v_1 & T \\
T^\circ & v_2
\end{bmatrix}
\co \L^p(\M_2 \otvn \cal M) \to \L^p(\M_2 \otvn\cal N),
\quad
\begin{bmatrix}a&b\\c&d\end{bmatrix}
\mapsto
\begin{bmatrix}v_1(a)&T(b)\\
T^\circ(c)&v_2(d)\end{bmatrix}$, is completely positive. Its decomposable norm is $
\norm{T}_{\mathrm{dec}}
\ov{\mathrm{def}}{=}
\inf\max\{\norm{v_1},\norm{v_2}\}$, where the infimum is taken over all such pairs $(v_1,v_2)$. We say that $T$ is contractively decomposable if $\norm{T}_{\mathrm{dec}}\leq1$. See \cite[Definition~4.1]{Arh24b}.

\subsection{Jordan structures}

\paragraph{Ternary rings of operators} Rectangular operator spaces are naturally modeled by off-diagonal corners of von Neumann algebras. The intrinsic objects corresponding to such corners are ternary rings of operators, which we briefly recall. A ternary ring of operators (or simply $\TRO$) is a norm closed subspace $V$ of the space $\B(H,K)$, for some complex Hilbert spaces $H$ and $K$, which is closed under the triple product $(x,y,z) \mapsto xy^*z$, see, e.g., \cite[4.4.1 p.~161]{BLM04}. A TRO $V$ is called a $\W^*$-$\TRO$ if it is weak* closed in the dual Banach space $\B(H,K)$. A sub-TRO of a TRO $V$ is a closed subspace $W$ of $V$ satisfying $W W^*W \subset W$. We refer to \cite{BLM04}, \cite{BFT12}, \cite{BuT19}, \cite{BuT13a}, \cite{BuT13b}, \cite{DoR07}, \cite{EsM23}, \cite{Ham92}, \cite{Ham99}, \cite {KaR02}, \cite{NeR03}, \cite{PlR19}, \cite{SaS13}, \cite{SaS17} and \cite{Zet83} for more information on TROs.

\begin{example} \normalfont
\label{typical-TRO}
A basic example of $\W^*$-$\TRO$ is given by $e\cal{M}f$, where $e,f$ are projections of a von Neumann algebra $\cal{M}$.
\end{example}

\paragraph{Operator space structures of TROs}
Since the results of this paper are formulated in terms of complete isometries, the canonical operator space structure of a TRO will play an essential role. Every TRO admits an operator space structure. Indeed, let us assume that $V$ is a TRO contained in $\B(H,K)$. Then for each integer $n \geq 1$, the matrix space $\M_n(V)$ can be identified with a TRO contained in $\M_n(\B(H,K)) \cong \B(H^n,K^n)$. This provides a canonical operator space matrix norm on $V$ such that each matrix space $\M_n(V)$ is a TRO. By \cite[Proposition 2.1 p.~265]{KaR02} or \cite[Proposition 2.1 p.~83]{Ham99}, the TRO-matrix norms are uniquely determined on each TRO and do not depend on the choice of the representing Hilbert spaces.

\paragraph{$\TRO$-homomorphisms} A $\TRO$-homomorphism (or triple morphism) is a linear map $T \co V \to W$ between two $\TRO$s respecting the triple product, i.e.,
$$
T(xy^*z) 
= T(x)T(y)^*T(z), \quad x,y,z \in V.
$$
If, in addition, $T$ is an injection from $V$ onto $W$, we call $T$ a $\TRO$-isomorphism from $V$ onto $W$. By \cite[Proposition 2.4 p.~498]{EOR01}, a TRO-isomorphism between $\W^*$-TROs is necessarily weak* continuous. By  \cite[Proposition 2.1 p.~495]{EOR01}, every TRO-homomorphism is completely contractive, and every injective $\TRO$-homomorphism is completely isometric. Finally, by \cite[Proposition 2.1 p.~83]{Ham99} (see also \cite[Corollary 4.4.6 p.~163]{BLM04}), if $T \co V \to W$ is a linear map between TRO's then $T$ is a surjective complete isometry if and only if $T$ is a surjective 2-isometry if and only if $T$ is a TRO-isomorphism. Thus, for TROs, the complete isometric structure already determines the triple product.

\begin{example} \normalfont
\label{E-finite-dim-TRO}
Every finite-dimensional $\TRO$ $V$ is completely isometric to a finite direct sum of rectangular matrix spaces, i.e., it has the form
$$
V 
= \M_{m_1,n_1} \oplus_\infty \cdots \oplus_\infty \M_{m_k,n_k}, 
$$
see \cite[Corollary A.2 p.~302]{Kan13}.
\end{example}


\paragraph{Linking algebra} The linking algebra allows us to pass back and forth between the rectangular setting of TROs and the usual setting of von Neumann algebras. If $\cal{M}$ is a von Neumann algebra and $e$ is an orthogonal projection in $\cal{M}$, then $e\cal{M} e^\perp$ is a $\W^*$-TRO by Example \ref{typical-TRO}. Conversely, let $V$ be a $\W^*$-TRO contained in $\B(H,K)$. Replacing $H$ and $K$, if necessary, with the essential subspaces $\ovl{V^*K}$ and $\ovl{VH}$, respectively, we may and do assume that the representation is nondegenerate. We can then consider the adjoint space $V^* \ov{\mathrm{def}}{=} \{x^* : x \in V \}$, which is a subspace of the space $\B(K,H)$ and the von Neumann subalgebras 
\begin{equation}
\label{def-MV-NV}
\mathrm{M}(V) \ov{\mathrm{def}}{=} \ovl{\Span VV^*}^{\w^*}
\quad \text{and} \quad 
\mathrm{N}(V) \ov{\mathrm{def}}{=} \ovl{\Span V^*V}^{\w^*}
\end{equation}
of the algebras $\B(K)$ and $\B(H)$. Since $\mathrm{M}(V) \cdot V \subset V$ and $V \cdot \mathrm{N}(V) \subset V$, we can introduce the von Neumann algebra
\begin{equation}
\label{Linking-algebra}
\mathrm{R}(V)
\ov{\mathrm{def}}{=} \begin{bmatrix}
  \mathrm{M}(V)   &  V \\
  V^*   &  \mathrm{N}(V) \\
\end{bmatrix}
\subset \B(K \oplus H),
\end{equation}
which is called the linking von Neumann algebra of $V$, see \cite{KaR02}, \cite[p.~846]{Rua04} and \cite{WCW24}. Then we have a TRO-isomorphism 
\begin{equation}
\label{TRO-linking}
V 
= e \mathrm{R}(V)e^\perp,
\end{equation}
where $e 
\ov{\mathrm{def}}{=} \begin{bmatrix}
    \Id_K & 0  \\
    0 & 0  \\
\end{bmatrix}
$ 
and 
$e^\perp
\ov{\mathrm{def}}{=} \begin{bmatrix}
   0  &  0 \\
   0  &  \Id_H \\
\end{bmatrix}$. So a $\W^*$-TRO $V$ can be identified with the off-diagonal corner at the (1,2) position of its linking von Neumann algebra. By \cite[Lemma 2.1 p.~847]{Rua04}, the central covers $c(e)$ and $c(e^\perp)$ of $e$ and $e^\perp$ in the linking von Neumann algebra $\mathrm{R}(V)$ are equal to $1$.

Conversely, it is worth noting that, according to \cite[Theorem 2.1]{WCW24}, a von Neumann algebra is $*$-isomorphic to the linking von Neumann algebra of a $\W^*$-TRO if and only if it contains no abelian direct summand.

\paragraph{$\JB^*$-triples and $\JBW^*$-triples} 
We recall that an element $a$ in a Banach algebra $A$ is called hermitian \cite[p.~156]{Chu12} if $\norm{\exp \i ta} = 1$ for all $t \in \R$.

A $\JB^*$-triple, e.g. \cite[p.~130]{CGRP14} \cite[Definition 2.5.25 p.~164]{Chu12}, is a complex Banach space $X$ equipped with a continuous triple product $X \times X \times X \to X$, $(x,y,z) \mapsto \{x,y,z\}$ such that

\begin{enumerate}
\item $\{x,y,z\}$ is linear in $x$ and $z$, conjugate linear in $y$, and symmetric in $x$ and $z$,

\item $\{a,b,\{x,y,z\}\} = \{\{a,b,x\},y,z\} - \{x,\{b,a,y\},z\} + \{x,y,\{a,b,z\}\}$ for any $x,y,z,a,b \in X$,

\item the left multiplication $\mathrm{D}(x,x) \co X \to X$, defined by
\begin{equation}
\label{Left-mult-JBW-triple}
\mathrm{D}(x,x)y
\ov{\mathrm{def}}{=} \{x,x,y\},\quad y\in X,
\end{equation}
is Hermitian in the algebra $\B(X)$ of bounded operators on $X$ and has positive spectrum,

\item for any $x \in X$, we have $\norm{\{x,x,x\}}_X = \norm{x}_X^3$.
\end{enumerate}
The last axiom is a Jordan analogue of the Gelfand-Naimark axiom of $\mathrm{C}^*$-algebras. Moreover, if $X$ is a dual Banach space, then it is called a $\JBW^*$-triple \cite[Definition 2.5.30]{Chu12} \cite[p.~528]{CGRP14}. In this case, the predual is unique by \cite[Theorem 5.7.38 p.~233]{CGRP18}. 
We refer to the books \cite{Chu12}, \cite{CGRP14}, \cite{CGRP18}, \cite{Fri05}, \cite{Isi19} and \cite{Upm85} and to the important papers \cite{Hor87a}, \cite{Hor87b} and \cite{HoN88} for more information on $\JB^*$-triples and $\JBW^*$-triples.

\begin{example} \normalfont
\label{Ex-JBW-algebra-triple}
A $\JBW^*$-algebra $\cal{M}$ admits a structure of $\JBW^*$-triple with triple product 
\begin{equation}
\label{triple-prod-JBWstar}
(x,y,z) \mapsto \{x,y,z\} 
\ov{\mathrm{def}}{=} (x \circ y^*) \circ z + (z \circ y^*) \circ x - (x \circ z) \circ y^*.
\end{equation}
See \cite[Lemma 3.1.6 p.~174]{Chu12} and \cite[p.~224]{CGRP18}. 
\end{example}

\paragraph{Triple homomorphisms} A linear map $T \co X \to Y$ between two $\JB^*$-triples $X$ and $Y$ is called a triple homomorphism if it preserves the triple product:
$$
T(\{x,y,z\}) 
= \{T(x), T (y), T(z)\}, \quad x,y, z \in X.
$$
A bijective triple homomorphism is called a triple isomorphism. In $\JB^*$-triples, the metric structure and the algebraic structure determine each other. Indeed, by \cite[Theorem 3.1.7 p.~175]{Chu12} and \cite[Theorem 3.1.20 p.~183]{Chu12} a linear bijection between two $\JB^*$-triples is an isometry if and only if it is a triple-isomorphism. When both triples are $\JBW^*$-triples, such a bijection is weak* continuous with respect to their canonical preduals, see \cite[Corollary 5.7.39 p.~234]{CGRP18}.

\paragraph{$\JC^*$-triples and $\JW^*$-triples} A $\JC^*$-triple (or $\mathrm{J}^*$-algebra) \cite[Definition 2.5.34 p.~169]{Chu12} \cite[p.~440]{CGRP18} \cite[Definition 22.7.1 p.~460]{Isi19} is a norm closed subspace of $\B(H,K)$ which is also closed under the map $x \mapsto xx^*x$. By a polarization argument, any $\JC^*$-triple is closed under the triple product $(x,y,z) \mapsto \{x,y,z\}\ov{\mathrm{def}}{=}\frac{1}{2}(xy^*z+zy^*x)$. A $\JC^*$-triple is a $\JB^*$-triple. A weak* closed $\JC^*$-subtriple of $\B(H,K)$ is called a $\JW^*$-triple. We shall also use this terminology for a $\JBW^*$-triple which is isometrically triple isomorphic to such a concrete $\JW^*$-triple.

\begin{example} \normalfont
A $\W^*$-$\TRO$ is a $\JW^*$-triple for the symmetrized triple product
\[
\{x,y,z\}
\ov{\mathrm{def}}{=}\frac12(xy^*z+zy^*x).
\]
\end{example}

\begin{example} \normalfont
By \cite[Proposition 6.1.41 p.~362]{CGRP18}, the complexification $\mathrm{H}_3(\O)_\C=\mathrm{H}_3(\O_\C)$ of the $\JBW$-factor $\mathrm{H}_3(\O)$ is equipped with a structure of $\JBW^*$-factor. The underlying $\JBW^*$-triple is called a Cartan factor of type VI.
\end{example}

\begin{example} \normalfont
By \cite[p.~140]{Isi19}, the space $\Asym_n \ov{\mathrm{def}}{=} \{x \in \M_n : x^\top=-x \}$ of skew-symmetric complex matrices of $\M_n$ is a $\JBW^*$-triple (called a Cartan factor of type $\mathrm{III}_n$). If $n$ is even, it is even equipped with a structure of reversible $\JW^*$-algebra by \cite[Proposition 25.2.2 p.~513]{Isi19}.
\end{example}

Let $\cal{A}$ be a $\C^*$-algebra. The formula
\begin{equation}
\label{triple product-C-star}
\{x,y,z\}
\ov{\mathrm{def}}{=}
\frac12(xy^*z+zy^*x),
\quad
x,y,z \in \cal{A},
\end{equation}
endows $\cal{A}$ with a structure of $\JB^*$-triple. The same formula defines a $\JB^*$-triple structure on $\B(H,K)$ for any complex Hilbert spaces $H$ and $K$; see \cite[Fact~4.1.41 p.~465]{CGRP14}. Furthermore, each $\JB^*$-algebra is a $\JB^*$-triple with the same norm and triple product 
$$
\{x,y,z\}
= (x\circ y^*) \circ z + (z\circ y^*)\circ x - (x\circ z)\circ y^*
$$ 
(see \cite[Theorem 4.1.45]{CGRP14}). By a $\JB^*$-subtriple of a $\JB^*$-triple $E$ we mean a norm closed subspace of $E$ which is also closed for the triple product.

\paragraph{Tripotents and Peirce projections} An element $u$ in a $\JB^*$-triple $X$ is called a tripotent if it is a fixed point of the triple product, i.e., $\{u,u,u\} = u$, see \cite[Definition 4.2.13 p.~505]{CGRP14} and \cite[Definition 1.2.31 p.~28]{Chu12} for a larger context.

\begin{example}
\normalfont
\label{ex-trip-partial-iso}
If $X$ is a $\JC^*$-triple (in particular if $X$ is a $\JW^*$-triple), by \eqref{triple product-C-star} the equality $\{u,u,u\} = u$ is equivalent to $uu^*u=u$, for any $u \in X$. So the tripotents are precisely the partial isometries of $X$, by \cite[Exercise 15 p.~244]{Con90}.
\end{example}

With a tripotent $u$ and $0 \leq i \leq 2$, we can introduce the Peirce projections $\P_i^u \co X \to X$ with range $X_i(u)$. For any $x \in X$, we have
\begin{equation}
\label{Peirce-1}
\P_2^u(x) 
\ov{\mathrm{def}}{=} \{u, \{u, x, u\},u\}, \quad
\P_1^u \ov{\mathrm{def}}{=} 2\big(\mathrm{D}(u, u) - \P_2^u\big) 
\end{equation}
and 
\begin{equation}
\label{Peirce-2}
\P_0^u+ \P_1^u+ \P_2^u 
= \Id_{X}.
\end{equation}
These maps are contractive linear projections by \cite[Corollary 1.2]{FrB85b}. A crucial property of $\JBW^*$-triples is that for a tripotent $u$ of $X$ the Peirce-2 subspace $X_2(u)$ is a $\JBW^*$-algebra with product $(a,b) \mapsto a \circ b \ov{\mathrm{def}}{=} \{a,u,b\}$, involution $a^* \ov{\mathrm{def}}{=} \{u, a, u\}$ where $a,b \in X_2(u)$, and unit $u$.

\begin{example} \normalfont
\label{Pierce-example}
In the case where $X$ is a $\JBW^*$-algebra and $e \in X$ is a projection, i.e. a selfadjoint idempotent, the Peirce projections are given by the following expressions (essentially \cite[p.~48]{HOS84})
\begin{equation}
\label{Pierce-example-fo}
\P_2^e(x)
=\{e, x^*, e\}, \quad \P_1^e(x)= 2\{e,x^*,1-e\}, \quad \P_0^e(x)=\{1-e,x^*,1-e\},
\end{equation}
where $x \in X$.
\end{example}

A tripotent $u$ in a $\JBW^*$-triple $X$ is called complete if $X_0(u) = \{0\}$ \cite[p.~16]{HKPP20} \cite[p.~517]{CGRP14} and unitary if $X = X_2(u)$. 

\subsection{Banach space geometry}

Following \cite[Definition 5.1.1 p.~426]{Meg98} and \cite[Definition 1.1 p.~41]{Cio90}, we say that a Banach space $X$ is strictly convex (or rotund) if for any $t \in (0,1)$ and any $x,y \in X$ with $\norm{x}_X=\norm{y}_X=1$ and $x \not= y$ we have $\norm{tx+(1-t)y}_X < 1$. According to \cite[Proposition 1.2 p.~41]{Cio90}, this is equivalent to say that if $x \not= y$ with $\norm{x}_X=\norm{y}_X=1$ then $\norm{x+y}_X < 2$. A Banach space $X$ is called smooth \cite[Definition 3.1 p.~21]{Cio90} (see also \cite[pp.~480-481]{Meg98}) if for any $x \in X \setminus \{0\}$ there exists a unique $x^* \in X^*$ such that $\norm{x^*}_{X^*} = 1$ and $\la x,x^* \ra_{X,X^*} = \norm{x}_X$. If $X$ is a \textsl{reflexive} Banach space, it is known \cite[Definition 5.4.7 p.~481]{Meg98} \cite[Corollary 1.4 p.~43]{Cio90} that $X$ is smooth if and only if $X^*$ is strictly convex and $X$ is strictly convex if and only if $X^*$ is smooth.

The differentiability of the norm is a very useful tool in understanding the geometry of the underlying Banach space. Various notions of differentiability of the norm of Banach spaces include Fr\'echet, Gâteaux, and strong subdifferentiability.

\paragraph{Gâteaux differentiability}
Let $\Omega$ be an open subset of a normed vector space $E$, let $f \co \Omega \to \R$ be a real-valued function, let $x \in \Omega$ and let $h$ be a nonzero element of $E$. The function $f_h$ given by $t \mapsto f(x + t h)$ is defined on an open interval containing $0$. If the limit
$$
\lim_{t \to 0} \frac{f(x + t h) - f (x)}{t}
$$
exists, then we denote this derivative by $\nabla_x f (h)$. It is called the G\^ateaux derivative (directional derivative) of $f$ at $x$ in the direction $h$. If $\nabla_x f(h)$ is defined and $\lambda \in \R \setminus \{0\}$, then $\nabla_x f (\lambda h)$ is defined and $\nabla_x f (\lambda h) = \lambda \nabla_x f(h)$. The function $f$ is G\^ateaux differentiable at $x$ if $\nabla_x f(h)$ exists for any $h \in E$ and the map $h \mapsto \nabla_x f(h)$ is a continuous real linear functional on $E$.

If $f$ is Fr\'echet differentiable at $x$, it is also G\^ateaux differentiable there, and 
\begin{equation}
\label{Gateaux-Frechet}
\mathrm{D}_\R f(x).h
=\nabla_x f (h), \quad h \in E.
\end{equation}
For a Banach space $X$ and $x\in X\setminus\{0\}$, the norm is Fr\'echet differentiable at $x$ if and only if the norming functional at $x$ is unique and is strongly exposed by $x$ in the closed unit ball of $X^*$. We refer to the book \cite{DGZ93} for more information.

According to \cite[Theorem 3.5 p.~22]{Cio90} (see also \cite[Theorem 1 p.~22]{Die75}), a Banach space $X$ is smooth if and only if its norm $\norm{\cdot}_X$ is G\^ateaux differentiable on $X \setminus \{0\}$.

Recall the following result \cite{CoS70}.

\begin{prop}[Cohen--Sullivan]
\label{prop-Sullivan}
A subspace $Y$ of a smooth Banach space $X$ can be the range of at most one projection of norm one. 
\end{prop}

The following is \cite[Theorem 3.2.6 p.~297]{Meg98} combined with \cite[5.10 p.~148]{FHHMPZ01}.

\begin{prop}
\label{Prop-Fabian}
Let $X$ be a Banach space and consider a bounded linear operator $P \co X \to X$. Then $P$ is a projection if and only if $P^*\co X^* \to X^*$ is a projection. In this case and if $P$ is contractive, $P(X)^*$ is canonically isometrically isomorphic to $P^*(X^*)$, i.e.,
\begin{equation}
\label{Ran-P-P*}
(\Ran P)^*
\cong \Ran(P^*).
\end{equation}
\end{prop}

\subsection{Plurisubharmonic functions on open subsets of Banach spaces}

If $\Omega$ is an open set of the complex plane $\mathbb{C}$, we recall \cite[p.~246]{Muj86} \cite[Definition 3.1.2 p.~60]{ArG01} that a function $f \co \Omega \to [-\infty,\infty)$ is said to be subharmonic if $f$ is upper semicontinuous and if for each $a \in \Omega$ and each $r > 0$ such that $\ovl{B}(a,r) \subset \Omega$ we have the inequality
$$
f(a)
\leq \frac{1}{2\pi} \int_{0}^{2\pi} f(a+r\e^{\i \theta}) \d \theta
$$
and if $f$ is not identically $-\infty$ on any connected component of $\Omega$.

Let $\Omega$ be an open subset of a \textit{complex} Banach space $X$. Following \cite[Definition 34.3 p.~247]{Muj86}, we say that a function $f \co \Omega \to [-\infty, \infty)$ is plurisubharmonic if $f$ is upper semicontinuous and if for any $a \in \Omega$ and any $b \in X$ such that $
\{a+\lambda b : |\lambda|\leq 1\} \subset \Omega$ we have
$$
f(a)
\leq \frac{1}{2\pi} \int_{0}^{2\pi} f(a+b\e^{\i \theta}) \d \theta.
$$
Since $f$ is Borel measurable and bounded above by \cite[Proposition 34.1 p.~245]{Muj86} on each disc $a + \ovl{B}(0,1)b \subset \Omega$, the integral $\int_{0}^{2\pi} f(a+b\e^{\i \theta}) \d \theta $ is well-defined, though it can be $-\infty$. 

An upper semicontinuous function $f\colon\Omega\to[-\infty,\infty)$ is plurisubharmonic if and only if for any $x,y \in X$, its restriction $z \mapsto f(x+zy)$ is subharmonic or identically $-\infty$ on each connected component of $\{z \in \mathbb{C} : x+z y \in \Omega \}$.

\begin{example} \normalfont
\label{Ex-C-conv-pluri}
Let $\Omega$ be an open convex subset of a complex Banach space $X$. Any continuous convex function $f \co \Omega \to \R$ is plurisubharmonic.
\end{example}

\subsection{Levi forms of functions defined on open subsets of Banach spaces}


Levi forms in the setting of complex Banach spaces have been studied in \cite{Lig76}. 
We briefly recall the relevant terminology. Let $\Omega$ be an open subset of a complex Banach space $X$ and let $\Phi \co \Omega \to \R$ be a real function of class $\mathrm{C}^2$, where Fr\'echet differentiability is understood with respect to the underlying real Banach space $X_{\R}$. Thus, for each $a \in \Omega$, the real Hessian $
\D_{\R}^2\Phi(a) \co X_{\R} \times X_{\R} \to \R$ is a continuous symmetric $\R$-bilinear form. In general it is not complex bilinear. The complex structure nevertheless determines a canonical Hermitian form from this real Hessian, namely the Levi form of $\Phi$ at $a$. 
With the convention that Hermitian forms are complex linear in the first variable, it is given by 
\begin{equation}
\label{Levi-form-diagonal}
\scr{L}_{\Phi,a}(z)
=\frac14\left(\D_{\R}^2\Phi(a)(z,z)+\D_{\R}^2\Phi(a)(\i z,\i z)\right), \quad z \in X,
\end{equation}
explicitly introduced in \cite[p.~1]{Lig76}. In particular, we have
\begin{equation}
\label{Hermitian-Levi}
\scr{L}_{\Phi,a}(z,w)
=\ovl{\scr{L}_{\Phi,a}(w,z)}, \quad z,w \in X.
\end{equation}
Its polar form is given by polarization by the following formula.

\begin{prop}
For any $z,w \in X$, we have
\begin{equation}
\label{Levi-form-general}
\scr{L}_{\Phi,a}(z,w)
=\frac14\big(\D_{\R}^2\Phi(a)(z,w)+\D_{\R}^2\Phi(a)(\i z,\i w)
+\i\D_{\R}^2\Phi(a)(z,\i w)-\i\D_{\R}^2\Phi(a)(\i z,w)\big).
\end{equation}
\end{prop}

\begin{proof}
For any $z,w \in X$, we use the shorthand notation $\phi(z,w) \ov{\mathrm{def}}{=} \D_{\R}^2\Phi(a)(z,w)$. Then $\phi$ is a continuous symmetric $\R$-bilinear form on $X_{\R} \times X_{\R}$. The diagonal form can be written
\[
\scr{L}_{\Phi,a}(z) 
\ov{\eqref{Levi-form-diagonal}}{=} \frac{1}{4} \left( \phi(z,z) + \phi(\i z, \i z) \right)
\]
For any $z,w \in X$, the polarization identity \cite[Theorem 1.1 p.~2]{Wei80} 
gives
\begin{align*}
\MoveEqLeft
\scr{L}_{\Phi,a}(z,w) 
= \frac{1}{4} \sum_{k=0}^{3} \i^k \scr{L}_{\Phi,a}(z + \i^k w) \\       
&=\frac{1}{4} \Big( \underbrace{\scr{L}_{\Phi,a}(z+w) - \scr{L}_{\Phi,a}(z-w)}_{A} + \i \big( \underbrace{\scr{L}_{\Phi,a}(z+\i w) - \scr{L}_{\Phi,a}(z-\i w)}_{B} \big) \Big).
\end{align*}
On the one hand, using the identity $\phi(x+y,x+y) - \phi(x-y,x-y) = 4\phi(x,y)$ we see that
\begin{align*}
A &= \scr{L}_{\Phi,a}(z+w) - \scr{L}_{\Phi,a}(z-w) \\
&= \frac{1}{4} \left( \phi(z+w, z+w) + \phi(\i z + \i w, \i z + \i w) \right) - \frac{1}{4} \left( \phi(z-w, z-w) + \phi(\i z - \i w, \i z - \i w) \right) \\
&= \frac{1}{4} \left( \phi(z+w, z+w) - \phi(z-w, z-w) \right) + \frac{1}{4} \left( \phi(\i z + \i w, \i z + \i w) - \phi(\i z - \i w, \i z - \i w) \right) \\
&= \frac{1}{4} \left( 4\phi(z,w) \right) + \frac{1}{4} \left( 4\phi(\i z, \i w) \right) 
= \phi(z,w) + \phi(\i z, \i w)
\end{align*}
On the other hand, we have
\begin{align*}
B &= \scr{L}_{\Phi,a}(z+\i w) - \scr{L}_{\Phi,a}(z-\i w) \\
&= \frac{1}{4} \left( \phi(z+\i w, z+\i w) + \phi(\i z - w, \i z - w) \right) - \frac{1}{4} \left( \phi(z-\i w, z-\i w) + \phi(\i z + w, \i z + w) \right) \\
&= \frac{1}{4} \left( \phi(z+\i w, z+\i w) - \phi(z-\i w, z-\i w) \right) - \frac{1}{4} \left( \phi(\i z + w, \i z + w) - \phi(\i z - w, \i z - w) \right) \\
&= \frac{1}{4} \left( 4\phi(z, \i w) \right) - \frac{1}{4} \left( 4\phi(\i z, w) \right) 
= \phi(z, \i w) - \phi(\i z, w)
\end{align*}
Combining these results, we deduce that
\[
\scr{L}_{\Phi,a}(z,w) 
= \frac{1}{4} \Big( \phi(z,w) + \phi(\i z, \i w) + \i \left( \phi(z, \i w) - \phi(\i z, w) \right) \Big)
\]
Distributing $\i$ and using the original notation $\D_{\R}^2\Phi(a)$, we obtain the equation \eqref{Levi-form-general}, i.e.,
\[
\scr{L}_{\Phi,a}(z,w) 
= \frac{1}{4} \left( \D_{\R}^2\Phi(a)(z,w) + \D_{\R}^2\Phi(a)(\i z,\i w) + \i \D_{\R}^2\Phi(a)(z,\i w) - \i \D_{\R}^2\Phi(a)(\i z,w) \right)
\]
\end{proof}

It is stated in \cite[Proposition 2.3 p.~67]{Lig76} \cite[Proposition 35.2 p.~257]{Muj86} that a real function $\Phi \co \Omega \to \R$ of class $\mathrm{C}^2$ is plurisubharmonic if and only if for any $a \in \Omega$ and any $z \in X$ we have 
\begin{equation}
\label{Levi-def-pos}
\scr{L}_{\Phi,a}(z,z) 
\geq 0.
\end{equation}


\section{Duality mappings associated with a gauge function}
\label{sec-Duality}

\subsection{Definitions}

Following \cite[Definition 4.1 p.~25]{Cio90} and \cite[Definition 2.13, p.~123]{Pat18}, a gauge function (or a weight function) is a continuous strictly increasing function $\varphi \co \R_+ \to \R_+$ such that $\varphi(0)=0$ and $\lim_{t \to \infty} \varphi(t) = \infty$. Let $X$ be a Banach space and let $\varphi \co \R_+ \to \R_+$ be a gauge function. For each $x \in X \setminus \{0\}$, we define, as in \cite[Definition 2.14, p.~123]{Pat18} \cite[Definition 4.1 p.~25]{Cio90}, the subset
\begin{equation}
\label{Def-JX-gauge}
\J_{X,\varphi}(x)
\ov{\mathrm{def}}{=}
\big\{x^* \in X^* : \la x,x^*\ra_{X,X^*}
=\norm{x}_X\norm{x^*}_{X^*}, 
\norm{x^*}_{X^*}=\varphi(\norm{x}_X) \big\}
\end{equation}
of the dual $X^*$, and we let $\J_{X,\varphi}(0) \ov{\mathrm{def}}{=} \{0\}$. According to \cite[Theorem 2.11, p.~124]{Pat18} or \cite[Remark 4.2 p.~25]{Cio90}, the subset $\J_{X,\varphi}(x)$ is nonempty for any $x \in X$. The multivalued map $\J_{X,\varphi} \co X \to 2^{X^*}$ is called the duality mapping associated with the gauge function $\varphi$. The duality mapping associated with the gauge function $\varphi(t)=t$ is called the normalized duality mapping, which we denote simply by $\J_X$. According to \cite[Corollary 4.5 p.~27]{Cio90}, the Banach space $X$ is smooth if and only if each duality mapping $\J_{X,\varphi}$ is single-valued. In this case, set
\[
F_\varphi(x)
\ov{\mathrm{def}}{=}
\int_0^{\norm{x}_X}\varphi(s)\d s,
\quad x\in X.
\]
The same result gives
\begin{equation}
\label{derivation-duality-mapping}
\nabla_xF_\varphi(y)
=
\Re\la y,\J_{X,\varphi}(x) \ra_{X,X^*},
\quad x\in X\setminus\{0\},\quad y\in X.
\end{equation}
The choice of the gauge only modifies the radial normalization of the duality mapping. More precisely, by \cite[Proposition 4.7(f) p.~28]{Cio90}, if $x \neq 0$, then
\begin{equation}
\label{relation-duality-gauge-normalized}
\J_{X,\varphi}(x)
=
\frac{\varphi(\norm{x}_X)}{\norm{x}_X} \J_X(x).
\end{equation}
In particular, all duality mappings associated with gauge functions determine the same norming functionals up to a positive scalar factor.


\begin{example} \normalfont
Suppose that $1 < p < \infty$. We will mainly use the gauge
\begin{equation}
\label{power-gauge}
\varphi_p(t)
\ov{\mathrm{def}}{=}t^{p-1}, \quad t \geq 0
\end{equation}
and we write
\begin{equation}
\label{Def-JpX}
\frak{j}_{p,X}
\ov{\mathrm{def}}{=}
\J_{X,\varphi_p}.
\end{equation}
Then we have
\begin{equation}
\label{relation-Jp-normalized}
\frak{j}_{p,X}(x)
\ov{\eqref{relation-duality-gauge-normalized}}{=}
\norm{x}_X^{p-2} \J_X(x),
\quad x \neq 0.
\end{equation}
If $X$ is smooth, then
\begin{equation}
\label{derivation-duality-mapping-2}
\nabla_x\left(\frac1p\norm{\cdot}_X^p\right)(y)
=
\Re\la y,\frak{j}_{p,X}(x)\ra_{X,X^*},
\quad
x\in X\setminus\{0\}, y\in X.
\end{equation}
In this case, for any $x\in X$, the element $\frak{j}_{p,X}(x)$ is characterized by the equalities
\begin{equation}
\label{properties-Jp}
\norm{\frak{j}_{p,X}(x)}_{X^*}
\ov{\eqref{Def-JX-gauge}}{=} \norm{x}_X^{p-1}
\quad\text{and}\quad
\big\la x,\frak{j}_{p,X}(x) \big\ra_{X,X^*}
\ov{\eqref{Def-JX-gauge}}{=} \norm{x}_X^p.
\end{equation}
\end{example}


Now, suppose that the Banach space $X$ is reflexive, smooth and strictly convex. Then \cite[Proposition 4.7(e) p.~28]{Cio90} implies that the duality mapping associated with a gauge $\varphi$ is bijective with inverse given by the duality mapping on $X^*$ associated with the inverse gauge $\varphi^{-1}$. Since
\[
\varphi_p^{-1}(t)=t^{\frac{1}{p-1}}=t^{p^*-1}=\varphi_{p^*}(t),
\quad \frac1p+\frac1{p^*}=1,
\]
we obtain
\begin{equation}
\label{inverse-power-duality-mapping}
\frak{j}_{p^*,X^*}\frak{j}_{p,X}=\Id_X
\quad\text{and}\quad
\frak{j}_{p,X}\frak{j}_{p^*,X^*}=\Id_{X^*}.
\end{equation}

\subsection{Contractive projections and duality mappings}

The following lemma is a generalization of \cite[Lemma 1]{And66} and is fundamental for us.

\begin{lemma}
\label{Lemma-smooth}
Let $X$ be a smooth strictly convex Banach space. Suppose that $1 < p < \infty$. Let $P \co X \to X$ be a contractive projection. For any $x \in X$, we have
\begin{equation}
\label{power-duality-map-range-projection}
x \in \Ran P
\quad \Longleftrightarrow \quad
\frak{j}_{p,X}(x) \in \Ran P^*.
\end{equation}
\end{lemma}

\begin{proof}
Suppose first that $x \in \Ran P$. Since $P(x)=x$, using the contractivity of $P^* \co X^* \to X^*$, we obtain 
\begin{align*}
\MoveEqLeft
\norm{x}_X^p         
\ov{\eqref{properties-Jp}}{=} \big\la x,\frak{j}_{p,X}(x) \big\ra_{X,X^*} 
=\big\la P(x),\frak{j}_{p,X}(x) \big\ra_{X,X^*}  \\
&=\big\la x,P^*(\frak{j}_{p,X}(x)) \big\ra_{X,X^*}
\leq \norm{x}_X \norm{P^*(\frak{j}_{p,X}(x))}_{X^*}
\leq \norm{x}_X\norm{\frak{j}_{p,X}(x)}_{X^*}
\ov{\eqref{properties-Jp}}{=} \norm{x}_X^{p}.
\end{align*}
Consequently, if $x \neq 0$, we obtain $\big\la x,P^*(\frak{j}_{p,X}(x)) \big\ra_{X,X^*}
=
\norm{x}_X^p$ and
\[
\norm{P^*(\frak{j}_{p,X}(x))}_{X^*}
=
\norm{x}_X^{p-1}.
\]
Thus $P^*(\frak{j}_{p,X}(x))$ belongs to the duality mapping of $x$ associated with the gauge $\varphi_p$. Since this duality mapping is single-valued, we obtain $
P^*(\frak{j}_{p,X}(x))
=\frak{j}_{p,X}(x)$ and therefore $\frak{j}_{p,X}(x) \in \Ran P^*$. The case $x = 0$ is obvious.

Conversely, suppose that $\frak{j}_{p,X}(x)$ belongs to the subspace $\Ran P^*$. Then we have $
P^*(\frak{j}_{p,X}(x))
=\frak{j}_{p,X}(x)$ and hence
\begin{align*}
\MoveEqLeft
\norm{x}_X^p
\ov{\eqref{properties-Jp}}{=}
\la x,\frak{j}_{p,X}(x) \ra_{X,X^*} 
=\big\la x,P^*(\frak{j}_{p,X}(x)) \big\ra_{X,X^*}
=\big\la P(x),\frak{j}_{p,X}(x) \big\ra_{X,X^*} \\
&\leq \norm{P(x)}_X \norm{\frak{j}_{p,X}(x)}_{X^*}
\ov{\eqref{properties-Jp}}{=} \norm{P(x)}_X \norm{x}_X^{p-1}
\leq \norm{x}_X^p.         
\end{align*}
If $x \neq 0$, it follows that $\norm{P(x)}_X=\norm{x}_X$ and
\begin{equation}
\label{inter-82}
\big\la P(x),\frak{j}_{p,X}(x) \big\ra_{X,X^*}
=
\norm{x}_X^p.
\end{equation}
By \eqref{properties-Jp} and \eqref{inter-82}, the functional $
x^*\ov{\mathrm{def}}{=}\frac{\frak{j}_{p,X}(x)}{\norm{x}_X^{p-1}}$ satisfies $\norm{x^*}_{X^*}=1$ and
\[
\bigg\la \frac{x}{\norm{x}_X},x^*\bigg\ra_{X,X^*}
=
\bigg\la \frac{P(x)}{\norm{x}_X},x^*\bigg\ra_{X,X^*}
=1.
\]
Consequently, we have
\[
1
=
\bigg\la \frac{x+P(x)}{2\norm{x}_X},x^*\bigg\ra_{X,X^*}
\leq
\norm{\frac{x+P(x)}{2\norm{x}_X}}_X
\leq
\frac{\norm{x}_X+\norm{P(x)}_X}{2\norm{x}_X}
=1.
\]
Hence
\[
\norm{\frac{x+P(x)}{2\norm{x}_X}}_X=1.
\]
Since both $x/\norm{x}_X$ and $P(x)/\norm{x}_X$ belong to the unit sphere and $X$ is strictly convex, we obtain $P(x)=x$. We conclude that $x \in \Ran P$.
\end{proof}

\begin{remark} \normalfont
In Lemma~\ref{Lemma-smooth}, the implication from $\Ran P$ to $\Ran P^*$ only uses smoothness, whereas the converse implication only uses strict convexity.
\end{remark}

The preceding property actually characterizes smooth strictly convex Banach spaces for any fixed power gauge. 

\begin{prop}
\label{Prop-characterization-smooth-strict-convex}
Let $X$ be a Banach space, let $1<p<\infty$, and let $j_p \co X \to X^*$ satisfy 
\begin{equation}
\label{power-selec}
\norm{j_{p}(x)}_{X^*}
= \norm{x}_X^{p-1}
\quad\text{and}\quad
\big\la x,j_{p}(x) \big\ra_{X,X^*}
= \norm{x}_X^p, \quad x \in X.
\end{equation}
The following assertions are equivalent.
\begin{enumerate}
\item The Banach space $X$ is smooth and strictly convex, and $j_p=\frak{j}_{p,X}$.
\item For any contractive projection $P \co X \to X$ and any $x \in X$, we have
\begin{equation}
\label{characterization-ranges-power-duality-map}
x\in\Ran P
\quad\Longleftrightarrow\quad
j_p(x)\in\Ran P^*.
\end{equation}
\end{enumerate}
\end{prop}

\begin{proof}
1. $\Rightarrow 2.$: Suppose first that $X$ is smooth and strictly convex. The equivalence \eqref{characterization-ranges-power-duality-map} follows from Lemma~\ref{Lemma-smooth}.

2. $\Rightarrow 1.$: Conversely, assume \eqref{characterization-ranges-power-duality-map}. We first prove that $X$ is smooth. Let $u \in X$ satisfy $\norm{u}_X=1$, and let $\phi \in X^*$ be any norming functional for $u$, that is,
\begin{equation}
\label{inter-99}
\norm{\phi}_{X^*}=1
\quad\text{and}\quad
\phi(u)=1.
\end{equation}
Consider the rank-one operator $P_{\phi,u} \co X \to X$ defined by $P_{\phi,u}(x) \ov{\mathrm{def}}{=} \phi(x)u$. It is a contractive projection with $\Ran P_{\phi,u} = \mathbb{C}u$. Since $u \in \Ran P_{\phi,u}$, condition \eqref{characterization-ranges-power-duality-map} implies that $
j_p(u) \in \Ran P_{\phi,u}^*$. For any $g \in X^*$, we have
\[
P_{\phi,u}^*(g)
=g(u)\phi,
\]
and therefore $\Ran P_{\phi,u}^*=\mathbb{C}\phi$. Hence $j_p(u)=\lambda \phi$ for some scalar $\lambda$. Since $\norm{u}_X=1$, \eqref{power-selec} gives
\[
1=j_p(u)(u)
=\lambda \phi(u)
\ov{\eqref{inter-99}}{=} \lambda.
\]
Thus $j_p(u)=\phi$. Since $\phi$ was an arbitrary norming functional for $u$, the norming functional of $u$ is unique. Hence $X$ is smooth. By the single-valuedness of the duality mapping associated with $\varphi_p$, we also obtain $
j_p
=\frak{j}_{p,X}$. It remains to prove strict convexity. Suppose, by contradiction, that $X$ is not strictly convex. Then, by \cite[Proposition 1.2 p.~41]{Cio90}, there exist distinct elements $u,v \in X$ such that
\[
\norm{u}_X=\norm{v}_X=1
\quad\text{and}\quad
\norm{\frac{u+v}{2}}_X=1.
\]
Choose $\phi \in X^*$ with $\norm{\phi}_{X^*}=1$ and $
\phi\left(\frac{u+v}{2}\right)
=1$. Since $|\phi(u)|\leq1$ and $|\phi(v)|\leq1$, the equality $\phi(u)+\phi(v)=2$ implies
\[
\phi(u)=\phi(v)=1.
\]
By the smoothness proved previously and since $\norm{u}_X=\norm{v}_X=1$, we have
\[
j_p(u)
=\phi
=j_p(v).
\]
Consider again the contractive rank-one projection $P_{\phi,u}(x)=\phi(x)u$. Since $j_p(v)=\phi$ belongs to $\Ran P_{\phi,u}^*$, condition \eqref{characterization-ranges-power-duality-map} yields
\[
v \in \Ran P_{\phi,u}
=\mathbb{C}u.
\]
Hence $v=\lambda u$ for some scalar $\lambda$ with $|\lambda|=1$. Since $\phi(u)=\phi(v)=1$, we obtain $\lambda=1$, and therefore $u=v$, a contradiction. Thus $X$ is strictly convex.
\end{proof}

There is also a classical characterization of contractively complemented subspaces in terms of duality mappings. If $X$ is reflexive, smooth and strictly convex, Calvert showed that a closed subspace $Y$ of $X$ is the range of a contractive projection if and only if the image of $Y$ under the normalized duality mapping is a linear subspace of $X^*$, see \cite{Cal75}. This characterization is unchanged if one uses the power duality mapping $\frak{j}_{p,X}$. Indeed, by \cite[Proposition 4.7 (d),(f) pp.~27--28]{Cio90}, the normalized duality mapping and $\frak{j}_{p,X}$ have the same radial directions, and since $Y$ is a subspace, we have
\begin{equation}
\label{same-image-power-normalized}
\frak{j}_{p,X}(Y)=\J_X(Y).
\end{equation}
Consequently, we have the equivalence
\[
Y\text{ is contractively complemented}
\quad\Longleftrightarrow\quad
\frak{j}_{p,X}(Y)\text{ is a linear subspace of }X^*,
\]
under the assumptions of Calvert's theorem. Moreover, we have the following observation.

\begin{prop}
Let $X$ be a reflexive smooth strictly convex Banach space. Suppose that $1 < p < \infty$. Let $P \co X \to X$ be a contractive projection. If $Y=\Ran P$ then
\begin{equation}
\label{power-duality-map-range-equality}
\frak{j}_{p,X}(Y)=\Ran P^*.
\end{equation}
\end{prop}

\begin{proof}
Lemma~\ref{Lemma-smooth} gives the inclusion $\frak{j}_{p,X}(Y)\subset\Ran P^*$. Conversely, let $z \in \Ran P^*$ and let $p^*$ be the conjugate exponent of $p$. Since $X$ is reflexive, the space $X^*$ is smooth and strictly convex. Applying Lemma~\ref{Lemma-smooth} to the contractive projection $P^* \co X^* \to X^*$, we obtain
\[
\frak{j}_{p^*,X^*}(z) \in \Ran P^{**}.
\]
Under the canonical identification $X^{**}=X$, we have $P^{**}=P$ and therefore $\Ran P^{**}=Y$. Hence $\frak{j}_{p^*,X^*}(z)\in Y$. Consequently, we have
\[
\frak{j}_{p,X}\frak{j}_{p^*,X^*}(z)
\ov{\eqref{inverse-power-duality-mapping}}{=} z,
\]
and therefore $z \in \frak{j}_{p,X}(Y)$. This proves \eqref{power-duality-map-range-equality}.
\end{proof}

\subsection{Case of noncommutative $\L^p$-spaces}

If the Banach space  $X$ is a noncommutative $\L^p$-space, the power duality mapping admits a particularly simple algebraic expression. This is the noncommutative generalization of the classical formula $\frak{j}_{p,\L^p(\Omega)}(f) = |f|^{p-1}\sgn(f)$ for any $f \in \L^p(\Omega)$, stated in \cite[Proposition 4.9 p.~72]{Cio90}, where $\Omega$ is a measure space.

\begin{prop}
\label{prop-J-lp}
Suppose that $1<p<\infty$. Let $\cal{M}$ be a von Neumann algebra equipped with a normal semifinite faithful trace. For this proposition, we identify $\L^{p^*}(\cal{M})$ with the dual of $\L^p(\cal{M})$ through the bilinear duality bracket
\begin{equation}
\label{bilinear-duality-bracket}
\la x,y \ra_{\L^p(\cal{M}),\L^{p^*}(\cal{M})}
\ov{\mathrm{def}}{=} \tr(xy).
\end{equation}
If $x \in \L^p(\cal{M})$ has polar decomposition $x=u|x|$, then
\begin{equation}
\label{comput-J}
\frak{j}_{p,\L^p(\cal{M})}(x)
= |x|^{p-1}u^*.
\end{equation}
\end{prop}

\begin{proof}
Recall that $\L^p(\cal{M})$ is strictly convex and smooth by \cite[Corollary 5.2]{PiX03}. Hence $\frak{j}_{p,\L^p(\cal{M})}$ is single-valued. Since $u^*u=s(|x|)$, we have
\[
\bnorm{|x|^{p-1}u^*}_{\L^{p^*}(\cal{M})}
=
\bnorm{|x|^{p-1}}_{\L^{p^*}(\cal{M})}
=
\norm{x}_{\L^p(\cal{M})}^{p-1}.
\]
Furthermore, using the tracial property,
\[
\big\la x,|x|^{p-1}u^*\big\ra_{\L^p(\cal{M}),\L^{p^*}(\cal{M})}
\ov{\eqref{bilinear-duality-bracket}}{=}
\tr\big(u|x|^pu^*\big)
=
\tr\big(|x|^pu^*u\big)
=
\tr(|x|^p)
=
\norm{x}_{\L^p(\cal{M})}^p.
\]
Thus $|x|^{p-1}u^*$ satisfies \eqref{properties-Jp}, and therefore it is equal to $\frak{j}_{p,\L^p(\cal{M})}(x)$.
\end{proof}

In particular, if $p=4$, we obtain $\frak{j}_{4,\L^4(\cal{M})}(x)=|x|^{3}u^*=x^*xx^*$.

\begin{remark} \normalfont
\label{Remark-sesquilinear-duality}
For $1<p<\infty$, let $p^* \ov{\mathrm{def}}{=} \frac{p}{p-1}$. In the remainder of the paper, we identify $\L^{p^*}(\cal M)$ with the dual of $\L^p(\cal M)$ through the sesquilinear pairing
\begin{equation}
\label{sesquilinear-pairing}
\la x,y\ra
=
\Tr(y^*x),
\quad
x\in\L^p(\cal M), y\in\L^{p^*}(\cal M),
\end{equation}
rather than through the bilinear pairing \eqref{bilinear-duality-bracket}. More precisely, the map
\[
\iota_p\co \L^{p^*}(\cal M)\to(\L^p(\cal M))^*,
\qquad
[\iota_p(y)](x)=\Tr(y^*x),
\]
is a conjugate linear isometric bijection. Under this identification, if $x=u|x|$ is the polar decomposition of $x$, the power duality mapping associated with the gauge $t\mapsto t^{p-1}$ is represented by
\begin{equation*}
\frak{j}_{p,\L^p(\cal M)}(x)=u|x|^{p-1},
\end{equation*}
in the sense that
\[
\J_{\L^p(\cal M),\varphi_p}
=
\iota_p\circ\frak{j}_{p,\L^p(\cal M)}.
\]
In particular, $\frak{j}_{2,\L^2(\cal M)}=\Id_{\L^2(\cal M)}$ and, when $p=4$, we have
\begin{equation}
\label{duality-map-L4}
\frak{j}_{4,\L^4(\cal M)}(x)
=
u|x|^3
=
xx^*x.
\end{equation}
Moreover, \eqref{inverse-power-duality-mapping} becomes
\begin{equation}
\label{inverse-power-duality-mapping-NC}
\frak{j}_{\frac43,\L^{\frac43}(\cal M)}
\frak{j}_{4,\L^4(\cal M)}
=
\Id_{\L^4(\cal M)},
\qquad
\frak{j}_{4,\L^4(\cal M)}
\frak{j}_{\frac43,\L^{\frac43}(\cal M)}
=
\Id_{\L^{\frac43}(\cal M)}.
\end{equation}
The formula \eqref{duality-map-L4} is especially adapted to our purposes, since polarization of the cubic map $x\mapsto xx^*x$ produces the Jordan triple product $(x,y,z)\mapsto\frac12(xy^*z+zy^*x)$. Moreover, this convention preserves rectangular supports. Indeed, if $e$ and $f$ are projections in $\cal M$ and $x\in e\L^p(\cal M)f$, then $\frak{j}_{p,\L^p(\cal M)}(x)$ belongs to $e\L^{p^*}(\cal M)f$, whereas the bilinear realization of the duality mapping belongs to the opposite corner $f\L^{p^*}(\cal M)e$.
\end{remark}

\section{Tangent Hilbert geometry}
\label{sec-tangent-Hilbert-geometry}
\subsection{First order geometry}

We first explain the geometric idea underlying our approach. The basic principle is purely Banach geometric. Let $X$ be a smooth complex Banach space, considered below as a real Banach space whenever differentiation of real-valued functions is involved. Let $P\co X\to X$ be a contractive projection and fix $a\in\Ran P$ with $a\neq0$. Suppose that $1<p<\infty$ and consider the convex function
\begin{equation}
\label{def-Phi-p-tangent}
\Psi_p(x)\ov{\mathrm{def}}{=}\frac1p\norm{x}_X^p,
\quad x\in X.
\end{equation}
Since $P(a)=a$, for any $z \in X$ and any $t \in \R$, we have
\begin{equation}
\label{basic-tangent-contractivity}
\Psi_p(a+tP(z))
=
\Psi_p(P(a+tz))
\leq
\Psi_p(a+tz).
\end{equation}
Thus the geometry of the $p$-power of the norm near a point $a$ of the range is more curved in the ambient direction $z$ than in its projected direction $P(z)$. This elementary observation naturally leads to differential methods. Such methods have already proved useful in the study of contractive projections. In particular, second derivatives of norms were used in \cite[Section 3]{LRG07} to detect structural properties of their ranges. Since $X$ is smooth, the function $\Psi_p$ is Gâteaux differentiable at every nonzero point. By the identification of the power duality mapping with the Gâteaux derivative of $\frac1p\norm{\cdot}_X^p$, we have
\begin{equation}
\label{first-variation-Phi-p}
\nabla_a\Psi_p(z)
=
\Re\big\la z,\frak{j}_{p,X}(a)\big\ra_{X,X^*},
\quad z\in X.
\end{equation}
Here $z\mapsto\nabla_a\Psi_p(z)$ is a continuous real linear functional on the underlying real Banach space $X_{\R}$.
To exploit \eqref{basic-tangent-contractivity}, fix $z\in X$ and define
\[
F_z(t)
\ov{\mathrm{def}}{=}
\Psi_p(a+tz)-\Psi_p(a+tP(z)),
\quad t \in \R.
\]
By \eqref{basic-tangent-contractivity}, $F_z(t) \geq 0$ for any $t \in \R$, whereas $F_z(0)=0$. Hence $t=0$ is a minimum of $F_z$. Since $F_z$ is differentiable at $0$, we obtain $F_z'(0)=0$. Therefore
\begin{align}
\label{first-tangent-orthogonality}
0
=\nabla_a\Psi_p(z)-\nabla_a\Psi_p(P(z))
=\nabla_a\Psi_p(z-P(z)).
\end{align}
Using \eqref{first-variation-Phi-p}, this says
\[
\Re\big\la z-P(z),\frak{j}_{p,X}(a)\big\ra_{X,X^*}
=0.
\]
Applying the same identity to $\i z$ instead of $z$ gives
\[
\Im\big\la z-P(z),\frak{j}_{p,X}(a)\big\ra_{X,X^*}
=0,
\]
and hence
\begin{equation}
\label{complex-first-tangent-orthogonality}
\big\la z-P(z),\frak{j}_{p,X}(a)\big\ra_{X,X^*}=0.
\end{equation}
Thus the directions belonging to $\ker P$ are tangent, at $a$, to the level hypersurface of $\Psi_p$, or equivalently to the sphere through $a$. In particular, the first-order geometry is precisely the geometry encoded by the duality mapping.

\subsection{Second order geometry}

Suppose now that $\Psi_p$ is twice Fr\'echet differentiable at $a$ as a real-valued function on $X_{\R}$. Since $F_z$ has a minimum at $0$, we have
\begin{equation}
\label{second-derivative-minimum}
F_z''(0)\geq0.
\end{equation}
Indeed, we have
\[
F_z''(0)
=
\lim_{t\to0}
\frac{F_z(t)+F_z(-t)-2F_z(0)}{t^2},
\]
and the numerator is positive. On the other hand, by the chain rule,
\[
F_z''(0)
=
\D_{\R}^2\Psi_p(a)[z,z]
-
\D_{\R}^2\Psi_p(a)[P(z),P(z)].
\]
Consequently,
\begin{equation}
\label{abstract-second-tangent-contraction}
\D_{\R}^2\Psi_p(a)[P(z),P(z)]
\leq
\D_{\R}^2\Psi_p(a)[z,z],
\quad z \in X.
\end{equation}
Since $\Psi_p$ is convex, the real symmetric bilinear form $\D_{\R}^2\Psi_p(a)$ is positive semidefinite. Indeed, for any $z\in X$, the function $ t\mapsto \Psi_p(a+tz)$ is convex on $\R$, and therefore
\[
\D_{\R}^2\Psi_p(a)[z,z]
=
\left.\frac{\d^2}{\d t^2}\right|_{t=0}\Psi_p(a+tz)
\geq 0.
\]
Consequently, the quadratic form $z\mapsto\D_{\R}^2\Psi_p(a)[z,z]$ defines a possibly degenerate local Hilbertian metric. If
\[
\cal N_a
\ov{\mathrm{def}}{=}
\{z\in X:\D_{\R}^2\Psi_p(a)[z,z]=0\},
\]
then the Cauchy--Schwarz inequality for positive semidefinite bilinear forms shows the equality $\D_{\R}^2\Psi_p(a)[z,w]=0$ for any $z\in\cal N_a$ and every $w\in X$. Hence $\D_{\R}^2\Psi_p(a)$ induces a positive definite real inner product on $X/\cal N_a$, whose completion is a real Hilbert space. The Levi form introduced below may be viewed as the Hermitian counterpart of this real second-order geometry.

The following result says that the Levi form measures exactly the second-order deviation of $\Phi$ from the mean-value property along complex lines.

\begin{prop}
Let $\Omega$ be an open subset of a complex Banach space $X$ and let $\Phi\co\Omega\to\R$ be of class $\mathrm C^2$. For any $z \in X$ and any $a \in \Omega$, we have
\begin{equation}
\label{Levi-average-limit}
\scr{L}_{\Phi,a}(z)
=\lim_{r \to 0}\frac1{r^2}
\left(
\frac1{2\pi}\int_0^{2\pi}\Phi(a+r\e^{\i t}z) \d t-\Phi(a)
\right).
\end{equation}
\end{prop}

\begin{proof}
For any $a \in \Omega$ and any $z \in X$, Taylor's formula \cite[(8.14.3) p.~190]{Die69} on the underlying real Banach space gives, as $r \to 0$,
\begin{equation}
\label{Taylor-1}
\Phi(a+r\e^{\i t}z)
=\Phi(a)+r\D_{\R}\Phi(a)(\e^{\i t}z)
+\frac{r^2}{2}\D_{\R}^2\Phi(a)(\e^{\i t}z,\e^{\i t}z)+o(r^2),
\end{equation}
where the remainder is uniform in $t \in [0,2\pi]$. Since $\e^{\i t}z = (\cos t)z+(\sin t)\i z$, we have
\[
\frac1{2\pi} \int_0^{2\pi}\D_{\R} \Phi(a)(\e^{\i t}z) \d t
=\frac1{2\pi} \int_0^{2\pi} \big[(\cos t)\D_{\R} \Phi(a)(z) + (\sin t)\D_{\R}\Phi(a)(\i z)\big] \d t
=0
\]
and
\begin{align*}
\MoveEqLeft
\frac1{2\pi} \int_0^{2\pi}
\D_{\R}^2\Phi(a)(\e^{\i t}z,\e^{\i t}z) \d t 
=\frac1{2\pi} \int_0^{2\pi}
(\cos t)^2\D_{\R}^2  \Phi(a)(z,z) + (\cos t \sin t)\D_{\R}^2\Phi(a)(z,\i z) \\
&+(\sin t \cos t)\D_{\R}^2\Phi(a)(\i z,z)+(\sin t)^2\D_{\R}^2\Phi(a)(\i z,\i z) \d t \\
&=\frac12\left(\D_{\R}^2\Phi(a)(z,z)+\D_{\R}^2\Phi(a)(\i z,\i z)\right) 
\ov{\eqref{Levi-form-diagonal}}{=} 2\scr{L}_{\Phi,a}(z).
\end{align*}
Consequently, we have
\begin{equation}
\label{circular-average-Levi}
\frac1{2\pi} \int_0^{2\pi} \Phi(a+r\e^{\i t}z) \d t
\ov{\eqref{Taylor-1}}{=} \Phi(a)+r^2\scr{L}_{\Phi,a}(z)+o(r^2).
\end{equation}
We conclude that
\begin{equation*}
\scr{L}_{\Phi,a}(z)
\ov{\eqref{circular-average-Levi}}{=} \lim_{r \to 0} \frac1{r^2}
\left(
\frac1{2\pi} \int_0^{2\pi} \Phi(a+r\e^{\i t}z) \d t - \Phi(a)
\right).
\end{equation*}
\end{proof}

There is also a useful positivity interpretation. Recall that a $\mathrm{C}^2$ real-valued function $\Phi \co \Omega \to \R$ on an open subset of a complex Banach space is plurisubharmonic precisely when its restriction to every complex affine line is subharmonic, or equivalently when $
\scr{L}_{\Phi,a}(z)
\ov{\eqref{Levi-def-pos}}{\geq} 0$ 
for any $a \in \Omega$ and any $z \in X$. Since $\Phi_p=\norm{\cdot}_X^p$ is convex for $p \geq 1$, its restriction to every complex affine line is a convex, hence subharmonic, function on $\mathbb{C}$. Consequently, wherever $\Phi_p$ is of class $C^2$,
\begin{equation}
\label{Levi-p-positive}
\scr{L}_{\Phi_p,a}(z)
\geq0.
\end{equation}

\subsection{Hilbertian geometry associated to a projection contracting a plurisubharmonic function}

A contractive projection is automatically contractive for the infinitesimal Hermitian geometry determined by the Levi form at every point of its range.

\begin{prop}
\label{prop-Levi-contractive-projection}
Let $\Omega$ be an open subset of a complex Banach space $X$. Consider a plurisubharmonic function $\Phi \co \Omega \to \R$ of class $\mathrm{C}^2$. Let $P \co X \to X$ be a bounded projection such that $P(\Omega) \subset \Omega$  satisfying
\begin{equation}
\label{Phi-contractivity}
\Phi(P(x))
\leq \Phi(x),\quad x \in \Omega.
\end{equation}
Let $a \in \Omega \cap \Ran P$. Then
\begin{equation}
\label{Levi-projection-contraction}
\scr{L}_{\Phi,a}(P(z))
\leq \scr{L}_{\Phi,a}(z),
\quad z \in X.
\end{equation}
Let
\begin{equation}
\label{Levi-null-space}
\cal{N}_{\Phi,a}
\ov{\mathrm{def}}{=}
\{z \in X : \scr{L}_{\Phi,a}(z)=0 \}.
\end{equation}
Then $P$ leaves $\cal N_{\Phi,a}$ invariant and induces an orthogonal projection on the Hilbert space obtained by completing $X/\cal N_{\Phi,a}$ for the inner product induced by $\scr{L}_{\Phi,a}(\cdot,\cdot)$. Moreover, we have
\begin{equation}
\label{Levi-orthogonality}
\scr{L}_{\Phi,a}(P(z),y)
=\scr{L}_{\Phi,a}(z,y),
\quad z\in X,\quad y \in \Ran P.
\end{equation}
\end{prop}

\begin{proof}
Let $z \in X$. Note that $P(a)=a$. For all sufficiently small $r > 0$ and any $t \in \R$, using the contractivity \eqref{Phi-contractivity}, we see that 
\[
\Phi(a+r\e^{\i t}P(z))
=\Phi(P(a)+r\e^{\i t}P(z))
=\Phi(P(a+r\e^{\i t}z))
\ov{\eqref{Phi-contractivity}}{\leq} \Phi(a+r\e^{\i t}z).
\]
Averaging over $t\in[0,2\pi]$ and using \eqref{circular-average-Levi}, we obtain
\[
\Phi(a)+r^2\scr{L}_{\Phi,a}(Pz)+o(r^2)
\leq
\Phi(a)+r^2\scr{L}_{\Phi,a}(z)+o(r^2).
\]
After division by $r^2$ and passage to the limit $r\to0^+$, this gives \eqref{Levi-projection-contraction}.

Since $\Phi$ is plurisubharmonic, its Levi form is positive semidefinite. Therefore the Cauchy--Schwarz inequality for positive sesquilinear forms shows that
\[
\scr{L}_{\Phi,a}(z)=0
\quad\Rightarrow\quad
\scr{L}_{\Phi,a}(z,w)=0,\quad w\in X.
\]
Consequently, $\cal N_{\Phi,a}$ is the radical of $\scr{L}_{\Phi,a}(\cdot,\cdot)$ and the latter induces an inner product on $X/\cal N_{\Phi,a}$. If $z \in \cal N_{\Phi,a}$, then \eqref{Levi-projection-contraction} implies
\[
0 \leq \scr{L}_{\Phi,a}(Pz)
\ov{\eqref{Levi-projection-contraction}}{\leq} \scr{L}_{\Phi,a}(z)
\ov{\eqref{Levi-null-space}}{=} 0,
\]
and therefore $P(z)$ belongs to $\cal N_{\Phi,a}$. Hence $P$ induces a projection $\widetilde{P} \co X/\cal N_{\Phi,a} \to X/\cal N_{\Phi,a}$, $[z] \mapsto [P(z)]$ on the quotient space $X/\cal N_{\Phi,a}$. Moreover, we have
\[
\scr{L}_{\Phi,a}(\widetilde{P}[z])
=
\scr{L}_{\Phi,a}(Pz)
\ov{\eqref{Levi-projection-contraction}}{\leq}
\scr{L}_{\Phi,a}(z).
\]
So  the map $\widetilde{P}$ is contractive for the induced Hilbertian norm. It therefore extends to a contractive projection on the Hilbert space completion of $X/\cal N_{\Phi,a}$. Recall that every contractive projection on a Hilbert space is orthogonal by \cite[Proposition 3.3 p.~37]{Con90}. Since $[y]$ belongs to the range of the induced projection for any $y \in \Ran P$, we obtain
\[
\scr{L}_{\Phi,a}(z-Pz,y)
=0,
\]
which is \eqref{Levi-orthogonality}.
\end{proof}

The Levi form also defines a conjugate linear map into the dual space.
\begin{defi}
Let $\Omega$ be an open subset of a complex Banach space $X$, let $\Phi \co \Omega \to \R$ be of class $\mathrm{C}^2$ and let $a\in\Omega$. We introduce the conjugate linear operator $\Lambda_{\Phi,a} \co X \to X^*$ defined by
\begin{equation}
\label{Levi-dual-map}
\big\la z,\Lambda_{\Phi,a}(y) \big\ra_{X,X^*}
\ov{\mathrm{def}}{=}
\scr{L}_{\Phi,a}(z,y),
\quad z,y \in X.
\end{equation}
We call $\Lambda_{\Phi,a}$ the Levi dual map of $\Phi$ at $a$.
\end{defi}

\begin{prop}
Under the assumptions of Proposition~\ref{prop-Levi-contractive-projection}, let $a\in\Omega\cap\Ran P$. Then
\begin{equation}
\label{Levi-adjoint-fixed}
P^*\Lambda_{\Phi,a}(y)
=\Lambda_{\Phi,a}(y),
\quad y \in \Ran P.
\end{equation}
\end{prop}

\begin{proof}
For any $z \in X$ and $y \in \Ran P$, we have
\begin{align*}
\MoveEqLeft
\big \la z,P^*\Lambda_{\Phi,a}(y) \big\ra_{X,X^*}
=\big \la P(z),\Lambda_{\Phi,a}(y) \big\ra_{X,X^*}
\ov{\eqref{Levi-dual-map}}{=} \scr{L}_{\Phi,a}(Pz,y) \\
&\ov{\eqref{Levi-orthogonality}}{=} \scr{L}_{\Phi,a}(z,y)
\ov{\eqref{Levi-dual-map}}{=} \big\la z,\Lambda_{\Phi,a}(y) \big\ra_{X,X^*}.         
\end{align*}
We conclude that $P^*\Lambda_{\Phi,a}(y)=\Lambda_{\Phi,a}(y)$.
\end{proof}

The next result may be regarded as an infinitesimal bicontractivity principle. A contractive projection $P$ need not have a contractive complement $\Id-P$ for the original Banach norm. Nevertheless, at every point $a\in \Omega \cap \Ran P$, the projection and its complement are simultaneously contractive for the Hilbertian geometry determined by the Levi form. Equivalently, the Banach space symmetry $2P-\Id$, which need not be an isometry of $X$, becomes an involutive isometry after passing to the Levi Hilbert space at $a$.

\begin{cor}
\label{cor-Levi-Pythagoras-bicontractivity}
Let $\Omega$ be an open subset of a complex Banach space $X$. Consider a plurisubharmonic function $\Phi \co \Omega \to \R$ of class $\mathrm{C}^2$. Let $P \co X \to X$ be a bounded projection such that $P(\Omega) \subset \Omega$ satisfying $\Phi(P(x)) \leq \Phi(x)$ for any $x \in \Omega$. Let $a \in \Omega \cap \Ran P$. Then, for any $z,w \in X$, we have
\begin{equation}
\label{Levi-Pythagoras}
\scr{L}_{\Phi,a}(z,w)
=
\scr{L}_{\Phi,a}(Pz,Pw)
+
\scr{L}_{\Phi,a}((\Id-P)z,(\Id-P)w).
\end{equation}
In particular, we have
\begin{equation}
\label{Levi-Pythagoras-diagonal}
\scr{L}_{\Phi,a}(z)
=
\scr{L}_{\Phi,a}(Pz)
+
\scr{L}_{\Phi,a}((\Id-P)z),
\quad z\in X.
\end{equation}
Consequently, both $P$ and $\Id-P$ are contractive for the Levi seminorm at $a$:
\begin{equation}
\label{Levi-bicontractivity}
\scr{L}_{\Phi,a}(Pz)
\leq
\scr{L}_{\Phi,a}(z),
\qquad
\scr{L}_{\Phi,a}((\Id-P)z)
\leq
\scr{L}_{\Phi,a}(z).
\end{equation}
Moreover, the symmetry
\begin{equation}
\label{Levi-symmetry}
S\ov{\mathrm{def}}{=}2P-\Id
\end{equation}
is an isometry for the Levi form:
\begin{equation}
\label{Levi-symmetry-isometry}
\scr{L}_{\Phi,a}(Sz,Sw)
=
\scr{L}_{\Phi,a}(z,w),
\quad z,w \in X.
\end{equation}
Hence $S$ leaves $\cal{N}_{\Phi,a}$ invariant and induces an involutive unitary operator on the Hilbert space $\cal{H}_{\Phi,a}$ obtained by completing $X/\cal{N}_{\Phi,a}$ for the inner product induced by $\scr{L}_{\Phi,a}(\cdot,\cdot)$. In particular, the operators induced by $P$ and $\Id-P$ on this Hilbert space are complementary orthogonal projections.
\end{cor}

\begin{proof}
Let $z,w \in X$. Since $P(w) \in \Ran P$, applying \eqref{Levi-orthogonality} with $(\Id-P)z$ instead of $z$, we obtain
\begin{equation}
\label{inter-525}
\scr{L}_{\Phi,a}((\Id-P)z,Pw)
\ov{\eqref{Levi-orthogonality}}{=}
\scr{L}_{\Phi,a}(P(\Id-P)z,Pw)
=0,
\end{equation}
where the last equality follows from $P(\Id-P)=0$. Since the Levi form is Hermitian, \eqref{inter-525} also gives
\[
\scr{L}_{\Phi,a}(Pz,(\Id-P)w)
=\overline{\scr{L}_{\Phi,a}((\Id-P)w,Pz)}
\ov{\eqref{Levi-orthogonality}}{=}\overline{\scr{L}_{\Phi,a}(P(\Id-P)w,Pz)}
=0.
\]
Using the decompositions $z=Pz+(\Id-P)z$ and $w=Pw+(\Id-P)w$, sesquilinearity gives
\begin{align*}
\MoveEqLeft
\scr{L}_{\Phi,a}(z,w)
=\scr{L}_{\Phi,a}\big(Pz+(\Id-P)z,Pw+(\Id-P)w\big) \\
&=\scr{L}_{\Phi,a}\big(Pz,Pw\big) +\scr{L}_{\Phi,a}\big(Pz,(\Id-P)w\big) 
+\scr{L}_{\Phi,a}\big((\Id-P)z,Pw\big) \\
&+\scr{L}_{\Phi,a}\big((\Id-P)z,(\Id-P)w\big)  
\ov{\eqref{inter-525}}{=} \scr{L}_{\Phi,a}(Pz,Pw) + \scr{L}_{\Phi,a}((\Id-P)z,(\Id-P)w),
\end{align*}
which proves \eqref{Levi-Pythagoras}. Taking $w=z$ gives \eqref{Levi-Pythagoras-diagonal}. Since $\Phi$ is plurisubharmonic, its Levi form is positive semidefinite, and hence both terms on the right-hand side of \eqref{Levi-Pythagoras-diagonal} are nonnegative. This proves \eqref{Levi-bicontractivity}.

Now, we prove \eqref{Levi-symmetry-isometry}, i.e., $\scr{L}_{\Phi,a}(Sz,Sw) = \scr{L}_{\Phi,a}(z,w)$ for any $z,w \in X$. Since
\begin{equation}
\label{def-S-inter-46}
Sz
\ov{\eqref{Levi-symmetry}}{=}  Pz-(\Id-P)z
\quad\text{and}\quad
Sw
\ov{\eqref{Levi-symmetry}}{=} Pw-(\Id-P)w,
\end{equation}
we have
\begin{align*}
\MoveEqLeft
\scr{L}_{\Phi,a}(Sz,Sw)
\ov{\eqref{def-S-inter-46}}{=} \scr{L}_{\Phi,a}\big(Pz-(\Id-P)z,Pw-(\Id-P)w\big) =\scr{L}_{\Phi,a}\big(Pz,Pw\big)\\
&-\scr{L}_{\Phi,a}\big((\Id-P)z,Pw\big)-\scr{L}_{\Phi,a}\big(Pz,(\Id-P)w\big)+\scr{L}_{\Phi,a}\big((\Id-P)z,(\Id-P)w\big)\\
&=\scr{L}_{\Phi,a}(Pz,Pw)
+
\scr{L}_{\Phi,a}((\Id-P)z,(\Id-P)w) 
=
\scr{L}_{\Phi,a}(z,w),         
\end{align*}
since the mixed terms vanish by the preceding orthogonality relations \eqref{inter-525}. Moreover, $S^2=\Id$. Thus $S$ preserves the null space $\cal N_{\Phi,a}$ and induces an involutive isometry on the associated Hilbert completion. Since this induced operator is an involutive isometry on a Hilbert space, it is a selfadjoint unitary. The identities
\[
P=\frac12(\Id+S)
\quad\text{and}\quad
\Id-P=\frac12(\Id-S)
\]
then show that the operators induced by $P$ and $\Id-P$ are complementary orthogonal projections.
\end{proof}

\begin{cor}
\label{cor-Levi-intertwining}
Under the assumptions of Proposition~\ref{prop-Levi-contractive-projection}, let $a\in\Omega\cap\Ran P$. Then the Levi dual map $\Lambda_{\Phi,a} \co X \to X^*$ defined by \eqref{Levi-dual-map} satisfies the intertwining relation
\begin{equation}
\label{Levi-intertwining}
P^*\Lambda_{\Phi,a}
=
\Lambda_{\Phi,a}P.
\end{equation}
Equivalently, $P$ is selfadjoint with respect to the Levi form:
\begin{equation}
\label{Levi-selfadjointness}
\scr{L}_{\Phi,a}(Pz,y)
=
\scr{L}_{\Phi,a}(z,Py),
\quad z,y \in X.
\end{equation}
In particular, if the Levi form $\scr{L}_{\Phi,a}(\cdot,\cdot)$ is nondegenerate, then
\begin{equation}
\label{Levi-kernel-orthogonal-range}
\ker P
=
(\Ran P)^{\perp_{\Phi,a}},
\end{equation}
where
\[
(\Ran P)^{\perp_{\Phi,a}}
\ov{\mathrm{def}}{=}
\{ z \in X : \scr{L}_{\Phi,a}(z,y)=0\text{ for any }y \in \Ran P\}.
\]
Consequently, in this case $P$ is uniquely determined by its range as the orthogonal projection onto $\Ran P$ for the Hilbertian geometry induced by the Levi form at $a$.
\end{cor}

\begin{proof}
Let $z,y \in X$. We first observe that
\[
\scr{L}_{\Phi,a}(Pz,(\Id-P)y)=0.
\]
Indeed, by the Hermitian property of the Levi form and \eqref{Levi-orthogonality},
\[
\scr{L}_{\Phi,a}(Pz,(\Id-P)y)
= \ovl{\scr{L}_{\Phi,a}((\Id-P)y,Pz)}
\ov{\eqref{Levi-orthogonality}}{=}
\ovl{\scr{L}_{\Phi,a}(P(\Id-P)y,Pz)}
=0.
\]
Consequently, since $y = (\Id-P)(y)+Py$, we have
\[
\scr{L}_{\Phi,a}(Pz,y)
=\scr{L}_{\Phi,a}\big(Pz,(\Id-P)(y)+Py\big)
=
\scr{L}_{\Phi,a}(Pz,Py).
\]
Since $P(y)$ belongs to the subspace $\Ran P$, another application of \eqref{Levi-orthogonality} yields
\[
\scr{L}_{\Phi,a}(Pz,Py)
=
\scr{L}_{\Phi,a}(z,Py).
\]
This proves \eqref{Levi-selfadjointness}. Now, for any $z,y\in X$, using \eqref{Levi-dual-map}, we obtain
\begin{align*}
\MoveEqLeft
\big\la z,P^*\Lambda_{\Phi,a}(y)\big\ra_{X,X^*}
=
\big\la Pz,\Lambda_{\Phi,a}(y)\big\ra_{X,X^*}
\ov{\eqref{Levi-dual-map}}{=}
\scr{L}_{\Phi,a}(Pz,y) \\
&=
\scr{L}_{\Phi,a}(z,Py)
\ov{\eqref{Levi-dual-map}}{=}
\big\la z,\Lambda_{\Phi,a}(Py) \big\ra_{X,X^*}.         
\end{align*}
Hence
\[
P^*\Lambda_{\Phi,a}(y)
=\Lambda_{\Phi,a}(Py),
\]
which is \eqref{Levi-intertwining}. Suppose finally that the Levi form is nondegenerate. If $z \in \ker P$ and $y\in \Ran P$, then \eqref{Levi-selfadjointness} gives
\[
\scr{L}_{\Phi,a}(z,y)
=
\scr{L}_{\Phi,a}(z,Py)
\ov{\eqref{Levi-selfadjointness}}{=}
\scr{L}_{\Phi,a}(Pz,y)
=
0.
\]
Thus
\[
\ker P
\subset (\Ran P)^{\perp_{\Phi,a}}.
\]
Conversely, let $z \in (\Ran P)^{\perp_{\Phi,a}}$. Since $P(z)$ belongs to $\Ran P$, we have
\[
0
=
\scr{L}_{\Phi,a}(z,Pz)
=\scr{L}_{\Phi,a}(z,P^2z)
=\scr{L}_{\Phi,a}(Pz,Pz).
\]
The nondegeneracy of the Levi form implies $Pz=0$. Hence $z \in \ker P$, proving \eqref{Levi-kernel-orthogonal-range}.
\end{proof}

\begin{remark} \normalfont
\label{rem-Levi-intertwining-meaning}
Relation \eqref{Levi-intertwining} is stronger than the fixed-point relation \eqref{Levi-adjoint-fixed}. It says that the Levi dual map intertwines the projection $P$ with its Banach adjoint $P^*$. Thus, although $P$ need not be selfadjoint for any ambient Hilbert space structure, it becomes selfadjoint for every Levi geometry based at a point $a\in\Omega\cap\Ran P$. When the Levi form at $a$ is nondegenerate, the decomposition
\[
X
= \Ran P \oplus \ker P
\]
is therefore orthogonal for this Hilbertian geometry.
\end{remark}

\begin{remark} \normalfont
\label{rem-Levi-orthogonality-first-order}
The Levi orthogonality relation \eqref{Levi-orthogonality} can also be obtained by differentiating the first-order condition along the range. Indeed, let $\Phi\co\Omega\to\R$ be of class $\C^2$, let $P$ be a bounded projection such that $P(\Omega)\subset\Omega$ and $\Phi\circ P\leq\Phi$, and fix $a\in \Omega \cap \Ran P$. For any $a'\in\Omega\cap\Ran P$ and any $k\in\ker P$, the function
\[
t\mapsto \Phi(a'+tk)-\Phi(a')
\]
has a minimum at $0$. Hence
\[
\D_{\R}\Phi(a')[k]=0.
\]
Differentiating this identity at $a'=a$ in a direction $y\in\Ran P$ gives
\[
\D_{\R}^2\Phi(a)[y,k]=0.
\]
Since both $\Ran P$ and $\ker P$ are complex subspaces, applying the same argument to $\i y$ and $\i k$ and using \eqref{Levi-form-general} yields
\[
\scr L_{\Phi,a}(k,y)=0,
\quad
k \in \ker P, y \in \Ran P.
\]
Thus the orthogonality itself is already encoded in the variation of the first-order condition. Plurisubharmonicity supplies the positivity which turns the Levi form into a possibly degenerate Hilbertian geometry; this Hilbertian realization is what allows us to introduce relative positive operators and their imaginary powers.
\end{remark}

\subsection{Relative metric operators and modular groups}
\label{sec-relative-modular}

\begin{prop}
\label{prop-relative-Riesz-operator}
Let $Y$ be a complex vector space equipped with two positive definite Hermitian forms $\langle\cdot,\cdot\rangle_0$ and $\langle\cdot,\cdot\rangle_1$, linear in the first variable. Let $Y'$ denote the algebraic dual of $Y$. For $j\in\{0,1\}$, consider the conjugate linear Riesz map $\cal R_j\co Y\to Y'$ defined by
\begin{equation}
\label{relative-metric-Riesz-maps}
\cal R_j(y)(x)
\ov{\mathrm{def}}{=}
\langle x,y\rangle_j,
\quad x,y\in Y.
\end{equation}
Suppose that
\begin{equation}
\label{common-Riesz-range}
\cal R_0(Y)=\cal R_1(Y).
\end{equation}
Then the formula
\begin{equation}
\label{relative-metric-operator-definition}
\Delta
\ov{\mathrm{def}}{=}
\cal R_0^{-1}\cal R_1
\end{equation}
defines a linear bijection $\Delta\co Y\to Y$. It is the unique linear operator satisfying
\begin{equation}
\label{relative-metric-comparison}
\langle x,y\rangle_1
=
\langle x,\Delta y\rangle_0,
\quad
x,y \in Y.
\end{equation}
Moreover, $\Delta$ is positive and symmetric with respect to $\langle\cdot,\cdot\rangle_0$, in the sense that
\begin{equation}
\label{relative-metric-selfadjoint-positive}
\langle\Delta x,y\rangle_0
=
\langle x,\Delta y\rangle_0
\quad\text{and}\quad
\langle x,\Delta x\rangle_0
=
\langle x,x\rangle_1>0
\end{equation}
for any $x,y\in Y$ with $x\neq0$.
If $Y$ is finite-dimensional, condition \eqref{common-Riesz-range} is automatic, $\Delta$ has spectrum contained in $(0,\infty)$ and
\begin{equation}
\label{relative-metric-square-root}
\langle x,y\rangle_1
=
\big\la \Delta^{\frac12}x,\Delta^{\frac12}y\big\ra_0,
\quad
x,y \in Y.
\end{equation}
\end{prop}

\begin{proof}
For $j\in\{0,1\}$, the map $\cal R_j$ is conjugate linear. Indeed, for $\lambda\in\mathbb C$ and $x,y\in Y$, we have
\[
\cal R_j(\lambda y)(x)
=
\langle x,\lambda y\rangle_j
=
\overline{\lambda}\langle x,y\rangle_j
=
\overline{\lambda}\cal R_j(y)(x).
\]
Moreover, $\cal R_j$ is injective. Indeed, if $\cal R_j(y)=0$, then
\[
\langle y,y\rangle_j
=
\cal R_j(y)(y)
=
0,
\]
and the positive definiteness of $\langle\cdot,\cdot\rangle_j$ gives $y=0$.
By \eqref{common-Riesz-range}, the map $\cal R_0^{-1}\cal R_1$ is well-defined on $Y$. Since both $\cal R_0^{-1}$ and $\cal R_1$ are conjugate linear, their composition $\Delta$ is linear. It is bijective, with inverse $
\Delta^{-1}
=
\cal R_1^{-1}\cal R_0$. For $x,y\in Y$, the definition of $\Delta$ gives
\[
\cal R_0(\Delta y)
=
\cal R_1(y).
\]
Evaluating both sides at $x$, we obtain
\[
\langle x,\Delta y\rangle_0
=
\cal R_0(\Delta y)(x)
=
\cal R_1(y)(x)
=
\langle x,y\rangle_1,
\]
which proves \eqref{relative-metric-comparison}.
Suppose that $T\co Y\to Y$ is another linear operator satisfying
\[
\langle x,y\rangle_1
=
\langle x,Ty\rangle_0,
\quad x,y\in Y.
\]
Then $\cal R_0(Ty)=\cal R_1(y)=\cal R_0(\Delta y)$. Since $\cal R_0$ is injective, we obtain $Ty=\Delta y$ for any $y \in Y$. This proves uniqueness.
Using \eqref{relative-metric-comparison} and the Hermitian property of the two forms, we have
\begin{align*}
\langle\Delta x,y\rangle_0
&=
\overline{\langle y,\Delta x\rangle_0}
=
\overline{\langle y,x\rangle_1}
=
\langle x,y\rangle_1
=
\langle x,\Delta y\rangle_0.
\end{align*}
Thus $\Delta$ is symmetric with respect to $\langle\cdot,\cdot\rangle_0$. Furthermore, we have
\[
\langle x,\Delta x\rangle_0
\ov{\eqref{relative-metric-comparison}}{=}
\langle x,x\rangle_1>0
\]
for any nonzero $x\in Y$. Hence $\Delta$ is positive. If $Y$ is finite-dimensional, the injective conjugate linear maps $\cal R_0,\cal R_1 \co Y \to Y'$ are automatically surjective. Thus \eqref{common-Riesz-range} is automatic. The spectral theorem applied to the positive selfadjoint operator $\Delta$ gives $\sigma(\Delta)\subset(0,\infty)$ and a positive square root $\Delta^{1/2}$. Hence
\[
\big\langle\Delta^{\frac12}x,\Delta^{\frac12}y\big\rangle_0
=
\langle x,\Delta y\rangle_0
\ov{\eqref{relative-metric-comparison}}{=}
\langle x,y\rangle_1,
\]
which proves \eqref{relative-metric-square-root}.
\end{proof}

\begin{remark} \normalfont
\label{rem-finite}
In finite dimension, $\Delta$ is a positive invertible selfadjoint operator.
\end{remark}

\begin{remark} \normalfont
\label{rem-relative-metric-Riesz-range}
The equality of the Riesz ranges in \eqref{common-Riesz-range}, rather than the surjectivity of the Riesz maps onto the whole algebraic dual, is the essential assumption. This formulation is useful when both Hermitian forms are represented in a common dual subspace, as happens below for the Levi forms and for the left and right tangent forms taking their Riesz representatives in $\Ran P^\dagger$.
\end{remark}

The family of Levi forms associated with points of the range can be compared by positive operators. This gives an intrinsic relative modular structure. Such an operator is obtained by raising one index of the second metric with the first one.

\begin{prop}
\label{prop-relative-operator-Hilbert-forms}
Let $Y$ be a finite-dimensional complex vector space equipped with two positive definite Hermitian forms $h_0$ and $h_1$, which we assume to be linear in the first variable. There exists a unique linear operator $\Delta_{1|0}\co Y\to Y$ such that
\begin{equation}
\label{relative-Hilbert-forms}
h_1(x,y)
=
h_0(x,\Delta_{1|0}y),
\quad
x,y\in Y.
\end{equation}
The operator $\Delta_{1|0}$ is invertible, positive and selfadjoint with respect to $h_0$. More precisely,
\begin{equation}
\label{relative-operator-selfadjoint}
h_0(\Delta_{1|0}x,y)
=
h_0(x,\Delta_{1|0}y),
\quad
x,y\in Y,
\end{equation}
and
\begin{equation}
\label{relative-operator-positive}
h_0(x,\Delta_{1|0}x)
=
h_1(x,x)>0,
\quad
x\in Y\setminus\{0\}.
\end{equation}
Consequently,
\begin{equation}
\label{relative-square-root-Hilbert-forms}
h_1(x,y)
=
h_0\left(\Delta_{1|0}^{\frac12}x,\Delta_{1|0}^{\frac12}y\right),
\quad
x,y \in Y.
\end{equation}
If $G_0$ and $G_1$ denote the Gram matrices of $h_0$ and $h_1$ in the same basis, so that
\[
h_j(x,y)=[y]^*G_j[x],
\quad
j \in \{0,1\},
\]
then
\begin{equation}
\label{relative-operator-Gram-matrices}
[\Delta_{1|0}]
=
G_0^{-1}G_1.
\end{equation}
We call $\Delta_{1|0}$ the \emph{relative operator} of $h_1$ with respect to $h_0$.
\end{prop}

\begin{proof}
For $j \in \{0,1\}$, consider the Riesz map $\cal R_j\co Y\to Y^*$ defined by
$$
\cal R_j(y)(x)
\ov{\mathrm{def}}{=}
h_j(x,y).
$$
Since $h_j$ is positive definite, $\cal R_j$ is injective. Since $Y$ is finite-dimensional, $\dim Y=\dim Y^*$, and therefore $\cal R_j$ is bijective. Notice that $\cal R_j$ is conjugate linear. Hence the composition
\begin{equation}
\label{abstract-relative-Riesz-operator}
\Delta_{1|0}
\ov{\mathrm{def}}{=}
\cal R_0^{-1}\cal R_1
\end{equation}
is linear. By definition,
\[
\cal R_0(\Delta_{1|0}y)=\cal R_1(y),
\]
and therefore, for any $x,y \in Y$, we have
\[
h_0(x,\Delta_{1|0}y)
=
\cal R_0(\Delta_{1|0}y)(x)
=
\cal R_1(y)(x)
=
h_1(x,y).
\]
This proves \eqref{relative-Hilbert-forms}. It also proves uniqueness, since the nondegeneracy of $h_0$ implies that an operator satisfying \eqref{relative-Hilbert-forms} is uniquely determined.

We prove selfadjointness. Using \eqref{relative-Hilbert-forms} and the Hermitian property of $h_0$ and $h_1$, we have
\begin{align*}
h_0(\Delta_{1|0}x,y)
=
\overline{h_0(y,\Delta_{1|0}x)}
=
\overline{h_1(y,x)}
=
h_1(x,y)
=
h_0(x,\Delta_{1|0}y).
\end{align*}
Thus $\Delta_{1|0}$ is selfadjoint with respect to $h_0$. Moreover, for any nonzero $x \in Y$, we have
\[
h_0(x,\Delta_{1|0}x)
\ov{\eqref{relative-Hilbert-forms}}{=}
h_1(x,x)>0
\]
Hence $\Delta_{1|0}$ is positive and, in particular, injective. Since $Y$ is finite-dimensional, it is invertible.

The positive square root $\Delta_{1|0}^{\frac12}$ is therefore well-defined. Using its selfadjointness with respect to $h_0$, we obtain
\[
h_0\left(\Delta_{1|0}^{\frac12}x,\Delta_{1|0}^{\frac12}y\right)
=
h_0(x,\Delta_{1|0}y)
\ov{\eqref{relative-Hilbert-forms}}{=}
h_1(x,y),
\]
which proves \eqref{relative-square-root-Hilbert-forms}.

Finally, fix a basis of $Y$ and let $D$ be the matrix of $\Delta_{1|0}$. By assumption,
\[
h_0(x,\Delta_{1|0}y)
=
[D y]^*G_0[x]
=
[y]^*D^*G_0[x].
\]
Comparing this expression with
\[
h_1(x,y)
=
[y]^*G_1[x],
\]
we obtain $
D^*G_0=G_1$. Taking adjoints and using $G_0^*=G_0$ and $G_1^*=G_1$ gives $G_0D=G_1$. Hence
\[
D=G_0^{-1}G_1,
\]
which proves \eqref{relative-operator-Gram-matrices}.
\end{proof}

\begin{remark} \normalfont
\label{rem-relative-operator-positivity}
The positivity and selfadjointness of $\Delta_{1|0}$ in Proposition~\ref{prop-relative-operator-Hilbert-forms} are relative to the Hilbert structure $h_0$. In particular, the matrix $G_0^{-1}G_1$ need not be selfadjoint for the standard Euclidean inner product, although it is selfadjoint and positive for the inner product determined by $G_0$.
\end{remark}

\begin{prop}
\label{prop-closed-relative-metric-operator}
Let $\cal{H}$ be a complex Hilbert space equipped with a Hermitian inner product $\langle\cdot,\cdot\rangle_0$, linear in the first variable. Let $\cal D$ be a dense subspace of $\cal H$ equipped with a positive definite Hermitian form $\langle\cdot,\cdot\rangle_1$. Suppose that this form is closed with respect to $\langle\cdot,\cdot\rangle_0$, that is, $\cal D$ is complete for the norm
\[
\norm{x}_{\cal D}
\ov{\mathrm{def}}{=}
\left(\norm{x}_{\cal H}^2+\langle x,x\rangle_1\right)^{\frac12}.
\]
Then there exists a unique positive selfadjoint operator $\Delta$ on $\cal H$ such that
\begin{equation}
\label{relative-metric-form-domain}
\dom\Delta^{\frac12}=\cal D
\end{equation}
and
\begin{equation}
\label{relative-metric-form}
\langle x,y\rangle_1
=
\left\langle \Delta^{\frac12}x,\Delta^{\frac12}y\right\rangle_0,
\quad
x,y \in \cal D.
\end{equation}
In particular, for $x\in\cal D$ and $y\in\dom\Delta$, we have
\begin{equation}
\label{relative-metric-operator-form}
\langle x,y\rangle_1
=
\langle x,\Delta y\rangle_0.
\end{equation}
Moreover, $\ker\Delta=\{0\}$. We call $\Delta$ the \emph{relative metric operator} of $\langle\cdot,\cdot\rangle_1$ with respect to $\langle\cdot,\cdot\rangle_0$.

If $\cal D=\cal H$ and there exists $C>0$ such that
\begin{equation}
\label{relative-form-upper-bound}
\langle x,x\rangle_1
\leq
C\langle x,x\rangle_0,
\quad x \in \cal H,
\end{equation}
then $\Delta$ is bounded and \eqref{relative-metric-operator-form} holds for all $x,y\in\cal H$. If, in addition, there exists $c>0$ such that
\begin{equation}
\label{relative-form-lower-bound}
c\langle x,x\rangle_0
\leq
\langle x,x\rangle_1,
\quad x \in \cal H,
\end{equation}
then $\Delta$ is boundedly invertible and
\begin{equation}
\label{relative-operator-bounds}
c\Id\leq\Delta\leq C\Id.
\end{equation}
Consequently, if the two Hermitian forms define equivalent Hilbert norms on the same vector space, then $\Delta$ is a bounded positive invertible operator characterized by
\begin{equation}
\label{relative-operator-equivalent-Hilbert-norms}
\langle x,y\rangle_1
=
\langle x,\Delta y\rangle_0,
\quad x,y\in\cal H.
\end{equation}
\end{prop}

\begin{proof}
The existence of $\Delta$ follows from the representation theorem for densely defined closed positive sesquilinear forms. Indeed, the form
\[
q(x,y)
\ov{\mathrm{def}}{=}
\langle x,y\rangle_1,
\quad x,y\in\cal D,
\]
is densely defined, positive and closed on the Hilbert space $(\cal H,\langle\cdot,\cdot\rangle_0)$. Hence there exists a unique positive selfadjoint operator $\Delta$ on $\cal H$ such that $\dom\Delta^{\frac12}=\cal D$ and
\[
q(x,y)
=
\big\langle\Delta^{\frac12}x,\Delta^{\frac12}y\big\rangle_0,
\quad x,y \in \cal D.
\]
This proves \eqref{relative-metric-form-domain} and \eqref{relative-metric-form}.

Let now $x\in\cal D$ and $y\in\dom\Delta$. Since $y\in\dom\Delta$ if and only if $y\in\dom\Delta^{\frac12}$ and $\Delta^{\frac12}y\in\dom\Delta^{\frac12}$, the selfadjointness of $\Delta^{\frac12}$ gives
\[
\langle x,y\rangle_1
\ov{\eqref{relative-metric-form}}{=}
\big\langle\Delta^{\frac12}x,\Delta^{\frac12}y\big\rangle_0
=
\langle x,\Delta y\rangle_0,
\]
which proves \eqref{relative-metric-operator-form}.

Suppose that $x\in\ker\Delta$. Then $x\in\dom\Delta\subset\cal D$ and
\[
\la x,x \ra_1
\ov{\eqref{relative-metric-form}}{=}
\bnorm{\Delta^{\frac12}x}_{\cal H}^2
=0.
\]
Since $\langle\cdot,\cdot\rangle_1$ is positive definite, we obtain $x=0$. Hence $\ker\Delta = \{0\}$. Notice that this does not imply in general that $\Delta^{-1}$ is bounded, since $0$ may belong to the continuous spectrum of $\Delta$.

Suppose now that $\cal D=\cal H$ and that \eqref{relative-form-upper-bound} holds. For any $y \in \cal H$, the map $x\mapsto\langle x,y\rangle_1$ is a continuous linear functional on $(\cal H,\langle\cdot,\cdot\rangle_0)$, since by the Cauchy--Schwarz inequality for $\langle\cdot,\cdot\rangle_1$,
\[
|\langle x,y\rangle_1|
\leq
\langle x,x\rangle_1^{\frac12}\langle y,y\rangle_1^{\frac12}
\leq
C\norm{x}_{\cal H}\norm{y}_{\cal H}.
\]
By the Riesz representation theorem, there exists a unique element $\Delta y\in\cal H$ such that
\[
\langle x,y\rangle_1
=
\langle x,\Delta y\rangle_0,
\quad x \in \cal H.
\]
The uniqueness of the representing vector shows that the map $y\mapsto\Delta y$ is linear. Moreover,
\[
|\langle x,\Delta y\rangle_0|
=
|\langle x,y\rangle_1|
\leq
C\norm{x}_{\cal H}\norm{y}_{\cal H},
\]
and taking the supremum over $\norm{x}_{\cal H}\leq1$ yields $\norm{\Delta y}_{\cal H}\leq C\norm{y}_{\cal H}$. Thus $\Delta$ is bounded and $\norm{\Delta}\leq C$. This bounded operator coincides with the positive selfadjoint operator associated above with the closed form, by uniqueness in the representation theorem.

The Hermitian property of $\langle\cdot,\cdot\rangle_1$ gives, for $x,y\in\cal H$,
\[
\langle\Delta x,y\rangle_0
=
\overline{\langle y,\Delta x\rangle_0}
=
\overline{\langle y,x\rangle_1}
=
\langle x,y\rangle_1
=
\langle x,\Delta y\rangle_0,
\]
so $\Delta$ is selfadjoint. Furthermore, we have
\[
\la x,\Delta x\ra_0
=
\la x,x \ra_1
\geq 0,
\]
so $\Delta$ is positive.

Finally, if \eqref{relative-form-lower-bound} also holds, then
\[
c\norm{x}_{\cal H}^2
\leq
\langle x,\Delta x\rangle_0
\leq
C\norm{x}_{\cal H}^2,
\quad x \in \cal H.
\]
Therefore $c\Id\leq\Delta\leq C\Id$. In particular, $0\notin\sigma(\Delta)$ and $\Delta$ is boundedly invertible, with $\norm{\Delta^{-1}}\leq c^{-1}$. This proves \eqref{relative-operator-bounds} and the final assertion.
\end{proof}

\begin{remark} \normalfont
\label{rem-relative-metric-operator-unbounded}
In finite dimension, every positive definite Hermitian form is continuous and any two Hilbert norms are equivalent. Hence the relative metric operator of Proposition~\ref{prop-closed-relative-metric-operator} is automatically bounded and invertible. In infinite dimension, this need no longer be the case. The natural object associated with two Hilbertian geometries is in general an unbounded positive selfadjoint operator, and the fundamental identity is then the quadratic-form identity
\[
\langle x,y\rangle_1
=
\big\langle\Delta^{\frac12}x,\Delta^{\frac12}y\big\rangle_0,
\quad
x,y \in \dom \Delta^{\frac12},
\]
rather than the operator identity $\langle x,y\rangle_1=\langle x,\Delta y\rangle_0$, which requires $y\in\dom\Delta$.
\end{remark}

\begin{prop}
\label{prop-relative-Levi-operator}
Let $\Omega$ be an open subset of a complex Banach space $X$, let $\Phi \co \Omega \to \R$ be a plurisubharmonic function of class $C^2$, and let $P\co X\to X$ be a bounded projection satisfying the assumptions of Proposition~\ref{prop-Levi-contractive-projection}.  Assume that $\Ran P$ is finite-dimensional. Let $a,b \in \Omega \cap \Ran P$ and suppose that the restrictions of $\scr{L}_{\Phi,a}$ and $\scr{L}_{\Phi,b}$ to the subspace $\Ran P$ are nondegenerate. Then the conjugate linear maps
\[
\Lambda_a
\ov{\mathrm{def}}{=}
\Lambda_{\Phi,a}|_{\Ran P}
\quad\text{and}\quad
\Lambda_b
\ov{\mathrm{def}}{=}
\Lambda_{\Phi,b}|_{\Ran P}
\]
are bijections from $\Ran P$ onto $\Ran P^*$. Consequently, the formula
\begin{equation}
\label{def-relative-Levi-operator}
\Delta_{\Phi,b|a}
\ov{\mathrm{def}}{=}
\Lambda_a^{-1}\Lambda_b
\end{equation}
defines a linear bijection $\Delta_{\Phi,b|a} \co \Ran P \to \Ran P$. It is uniquely characterized by
\begin{equation}
\label{relative-Levi-representation}
\scr{L}_{\Phi,b}(x,y)
=\scr{L}_{\Phi,a}(x,\Delta_{\Phi,b|a}y),
\quad
x,y \in \Ran P.
\end{equation}
Moreover, $\Delta_{\Phi,b|a}$ is positive and selfadjoint for the Hilbert space structure defined by $\scr{L}_{\Phi,a}$. Finally, for $a,b,c \in \Omega \cap \Ran P$ for which the corresponding Levi forms are nondegenerate on $\Ran P$, we have
\begin{equation}
\label{relative-Levi-inverse}
\Delta_{\Phi,a|b}
=\Delta_{\Phi,b|a}^{-1}
\quad \text{and} \quad
\Delta_{\Phi,c|a}
=\Delta_{\Phi,b|a}\Delta_{\Phi,c|b}.
\end{equation}
\end{prop}

\begin{proof}
By \eqref{Levi-adjoint-fixed}, we have
\[
\Lambda_{\Phi,a}(\Ran P)\subset\Ran P^*
\quad\text{and}\quad
\Lambda_{\Phi,b}(\Ran P)\subset\Ran P^*.
\]
Suppose that $\Lambda_a(y)=0$ for some $y\in\Ran P$. Then
\[
0
=
\big\langle y,\Lambda_a(y)\big\rangle_{X,X^*}
\ov{\eqref{Levi-dual-map}}{=}
\scr L_{\Phi,a}(y,y).
\]
Since the restriction of $\scr L_{\Phi,a}$ to $\Ran P$ is nondegenerate, we obtain $y=0$. Thus $\Lambda_a$ is injective. Similarly, $\Lambda_b$ is injective. Since $
\dim\Ran P^*
=
\dim\Ran P$, both maps are bijections from $\Ran P$ onto $\Ran P^*$.
The restriction map $
\rho\co\Ran P^*\to(\Ran P)^*$ defined by
$$
\rho(x^*)\ov{\mathrm{def}}{=}x^*|_{\Ran P},
$$
is an isomorphism. Indeed, if $x^*\in\Ran P^*$ and $\rho(x^*)=0$, then $P^*x^*=x^*$ and, for any $z\in X$, we have
\[
\langle z,x^*\rangle_{X,X^*}
=
\langle z,P^*x^*\rangle_{X,X^*}
=
\langle Pz,x^*\rangle_{X,X^*}
=
0.
\]
Hence $x^*=0$, so $\rho$ is injective, and its surjectivity follows from the equality of dimensions.
For $x,y\in\Ran P$, we have
\[
\big[\rho\Lambda_a(y)\big](x)
=
\big\langle x,\Lambda_a(y)\big\rangle_{X,X^*}
\ov{\eqref{Levi-dual-map}}{=}
\scr L_{\Phi,a}(x,y),
\]
and similarly
\[
\big[\rho\Lambda_b(y)\big](x)
=
\scr L_{\Phi,b}(x,y).
\]
Thus $\rho\Lambda_a$ and $\rho\Lambda_b$ are precisely the Riesz maps of the two Hermitian forms $\scr L_{\Phi,a}|_{\Ran P}$ and $\scr L_{\Phi,b}|_{\Ran P}$. Proposition~\ref{prop-relative-Riesz-operator} therefore gives the relative metric operator
\[
\Delta_{\Phi,b|a}
=
(\rho\Lambda_a)^{-1}(\rho\Lambda_b)
=
\Lambda_a^{-1}\Lambda_b.
\]
It is a positive selfadjoint linear bijection on the Hilbert space determined by $\scr L_{\Phi,a}$ and satisfies
\[
\scr L_{\Phi,b}(x,y)
=
\scr L_{\Phi,a}(x,\Delta_{\Phi,b|a}y),
\quad
x,y\in\Ran P.
\]
Finally, we have
\[
\Delta_{\Phi,a|b}
=
\Lambda_b^{-1}\Lambda_a
=
(\Lambda_a^{-1}\Lambda_b)^{-1}
=
\Delta_{\Phi,b|a}^{-1},
\]
and
\[
\Delta_{\Phi,b|a}\Delta_{\Phi,c|b}
=
\Lambda_a^{-1}\Lambda_b\Lambda_b^{-1}\Lambda_c
=
\Lambda_a^{-1}\Lambda_c
=
\Delta_{\Phi,c|a}.
\]
This proves \eqref{relative-Levi-inverse}.
\end{proof}



\begin{cor}
\label{cor-relative-Levi-group}
Under the assumptions of Proposition~\ref{prop-relative-Levi-operator}, the family
\begin{equation}
\label{relative-Levi-unitary-group}
U_{b|a,t}
\ov{\mathrm{def}}{=}
\Delta_{\Phi,b|a}^{\i t},
\quad
t \in \R,
\end{equation}
is a one-parameter group of surjective isometries for the two Hilbert space structures induced on $\Ran P$ by $\scr L_{\Phi,a}$ and $\scr L_{\Phi,b}$. More precisely,
\begin{equation}
\label{relative-Levi-group-unitary-a}
\scr{L}_{\Phi,a}(U_{b|a,t}x,U_{b|a,t}y)
=\scr{L}_{\Phi,a}(x,y)
\end{equation}
and
\begin{equation}
\label{relative-Levi-group-unitary-b}
\scr{L}_{\Phi,b}(U_{b|a,t}x,U_{b|a,t}y)
=
\scr{L}_{\Phi,b}(x,y)
\end{equation}
for any $x,y \in \Ran P$ and any $t \in \R$. For any $s \in [0,1]$, the formula
\begin{equation}
\label{interpolated-Levi-forms}
\scr{L}_{a,b}^{(s)}(x,y)
\ov{\mathrm{def}}{=}
\scr{L}_{\Phi,a}(x,\Delta_{\Phi,b|a}^s y)
\end{equation}
defines a positive definite Hermitian form on the subspace $\Ran P$, with
\[
\scr{L}_{a,b}^{(0)}
=\scr{L}_{\Phi,a}
\quad\text{and}\quad
\scr{L}_{a,b}^{(1)}
=\scr{L}_{\Phi,b}.
\]
\end{cor}

\begin{proof}
Since $\Delta_{\Phi,b|a}$ is positive and selfadjoint for $\scr{L}_{\Phi,a}$, its imaginary powers are unitary for $\scr{L}_{\Phi,a}$. This proves \eqref{relative-Levi-group-unitary-a}. Using \eqref{relative-Levi-representation} and the fact that the operator $\Delta_{\Phi,b|a}$ commutes with its imaginary powers, we obtain
\begin{align*}
\MoveEqLeft
\scr{L}_{\Phi,b}(U_{b|a,t}x,U_{b|a,t}y)
\ov{\eqref{relative-Levi-representation}}{=} \scr{L}_{\Phi,a}(U_{b|a,t}x,\Delta_{\Phi,b|a}U_{b|a,t}y) \\
&=\scr{L}_{\Phi,a}(U_{b|a,t}x,U_{b|a,t}\Delta_{\Phi,b|a}y)
=\scr{L}_{\Phi,a}(x,\Delta_{\Phi,b|a}y)
\ov{\eqref{relative-Levi-representation}}{=} \scr{L}_{\Phi,b}(x,y).
\end{align*}
This proves \eqref{relative-Levi-group-unitary-b}. For $0 \leq s \leq 1$, the operator $\Delta_{\Phi,b|a}^s$ is positive and invertible. Hence
\[
\scr{L}_{a,b}^{(s)}(x,y)
=\scr{L}_{\Phi,a}\left(\Delta_{\Phi,b|a}^{\frac s2}x,
\Delta_{\Phi,b|a}^{\frac s2}y\right),
\]
which shows that $\scr{L}_{a,b}^{(s)}$ is a positive definite Hermitian form. The endpoint identities follow immediately from the definition and \eqref{relative-Levi-representation}.
\end{proof}

The preceding construction is more informative when the relative operator is first defined on an ambient tangent Hilbert space and the range of the projection is subsequently shown to be reducing. 

\begin{prop}
\label{prop-ambient-relative-Levi-reduction}
Under the assumptions of Proposition~\ref{prop-Levi-contractive-projection}, let $a,b\in\Omega\cap\Ran P$. Suppose that the Levi forms $\scr{L}_{\Phi,a}$ and $\scr{L}_{\Phi,b}$ are nondegenerate on $X$ and that their Hilbertian norms are equivalent. Let $\cal H_a$ denote the Hilbert space completion of $X$ for the norm $x\mapsto\scr L_{\Phi,a}(x,x)^{\frac12}$. There exists a unique bounded positive invertible operator $\widehat\Delta_{\Phi,b|a} \co \cal H_a \to \cal H_a$ such that
\begin{equation}
\label{ambient-relative-Levi-representation}
\scr{L}_{\Phi,b}(x,y)
=\scr{L}_{\Phi,a}(x,\widehat\Delta_{\Phi,b|a}y),
\quad
x,y \in X.
\end{equation}
Moreover, the orthogonal extension of $P$ to $\cal H_a$ commutes with $\widehat\Delta_{\Phi,b|a}$:
\begin{equation}
\label{ambient-relative-Levi-commutation}
P\widehat\Delta_{\Phi,b|a}
=
\widehat\Delta_{\Phi,b|a}P.
\end{equation}
Consequently, the subspaces $\ovl{\Ran P}^{\cal H_a}$ and $\ovl{\ker P}^{\cal H_a}$ reduce the operator $\widehat\Delta_{\Phi,b|a}$. In particular, we have
\begin{equation}
\label{ambient-relative-Levi-group-commutation}
P\widehat\Delta_{\Phi,b|a}^{\i t}
=
\widehat\Delta_{\Phi,b|a}^{\i t}P,
\quad
t \in \R.
\end{equation}
\end{prop}

Here, we use the same notation for the continuous extensions of the two Levi forms to $\mathcal H_a$, and for the orthogonal extension of $P$ to $\mathcal H_a$. In particular, expressions involving $\widehat\Delta_{\Phi,b|a}y\in\mathcal H_a$ are evaluated using these extended forms.

\begin{proof}
The equivalence of the two Hilbertian norms implies, by the Riesz representation theorem, that there exists a unique bounded positive invertible operator $\widehat\Delta_{\Phi,b|a}$ satisfying \eqref{ambient-relative-Levi-representation}. By Corollary~\ref{cor-Levi-intertwining}, the projection $P$ is selfadjoint for both $\scr{L}_{\Phi,a}$ and $\scr{L}_{\Phi,b}$. Let $x,y \in X$. Using successively \eqref{ambient-relative-Levi-representation} and Levi selfadjointness, we obtain
\begin{align*}
\scr{L}_{\Phi,a}(x,\widehat\Delta_{\Phi,b|a}Py)
&=\scr{L}_{\Phi,b}(x,Py)
\ov{\eqref{Levi-selfadjointness}}{=} \scr{L}_{\Phi,b}(Px,y)
=\scr{L}_{\Phi,a}(Px,\widehat\Delta_{\Phi,b|a}y)
=\scr{L}_{\Phi,a}(x,P\widehat\Delta_{\Phi,b|a}y).
\end{align*}
The nondegeneracy of $\scr{L}_{\Phi,a}$ gives
\[
\widehat\Delta_{\Phi,b|a}Py
=
P\widehat\Delta_{\Phi,b|a}y, \quad
y \in X.
\]
By continuity, this identity extends to $\cal H_a$, proving \eqref{ambient-relative-Levi-commutation}. Since $P$ is an orthogonal projection on $\cal H_a$, commutation with the selfadjoint operator $\widehat\Delta_{\Phi,b|a}$ means that both its range and its kernel are reducing. Commutation with the Borel functional calculus, and in particular with the imaginary powers, follows.
\end{proof}

\subsection{Power functions}

A quasi-norm on a complex vector space $X$ is a map $\norm{\cdot}_X \co X \to [0,\infty)$ satisfying the following conditions:
\begin{enumerate}
\item $\norm{x}_X=0$ if and only if $x=0$;

\item $\norm{\lambda x}_X=|\lambda|\norm{x}_X$ for any $x\in X$ and any $\lambda\in\mathbb C$;

\item there exists a constant $C>0$ such that
\[
\norm{x+y}_X
\leq
C\bigl(\norm{x}_X+\norm{y}_X\bigr),
\quad x,y \in X.
\]
\end{enumerate}
A quasi-normed space is a vector space equipped with a quasi-norm. We refer to the survey \cite{Kal03} for more information on quasi-normed spaces. If $X$ is a topological vector space whose topology is defined by a quasi-norm which is uniformly continuous on bounded subsets of $X$, then $X$ is called a continuously quasi-normed space.

In the quasi-normed setting, plurisubharmonicity is understood through upper semicontinuity and the submean inequality on each complex affine line. Following \cite[p.~115]{DGT84} and \cite{Edg85}, a continuously quasi-normed space is said to be locally $\PL$-convex if, for any $x,y\in X$ and any sufficiently small $r>0$, we have
\[
\norm{x}_X
\leq
\frac1{2\pi}\int_0^{2\pi}
\norm{x+r\e^{\i t}y}_X
\d t.
\]
By \cite[Theorem~1.1 p.~40]{Ale81} and \cite[Proposition~2.2 p.~115]{DGT84}, this condition is equivalent to each of the following properties:
\begin{enumerate}
\item the function $
X\to[-\infty,\infty)$, $x\mapsto\log\norm{x}_X$ is plurisubharmonic, where $\log 0\ov{\mathrm{def}}{=}-\infty$;

\item for any, equivalently for some, $0 < p < \infty$, the function $X \to \R_+$, $x \mapsto \norm{x}_X^p$ is plurisubharmonic. Equivalently, we have
\[
\norm{x}_X
\leq
\left(
\frac1{2\pi}\int_0^{2\pi}
\norm{x+r\e^{\i t}y}_X^p
\d t
\right)^{\frac1p}
\]
for any $x,y\in X$ and any sufficiently small $r>0$.
\end{enumerate}
We refer also to \cite{Pav91}.

\begin{example}
\normalfont
\label{Banach-spaces}
As observed in \cite[Remark p.~116]{DGT84}, any complex Banach space is locally $\PL$-convex.
\end{example}

We now return to the Banach space setting, where the differential results of the preceding subsections apply. Let $X$ be a complex Banach space and let $1\leq p<\infty$. Consider the function
\begin{equation}
\label{def-de-Phi-p}
\Phi_p\co X \to \R_+,
\qquad
\Phi_p(x)
\ov{\mathrm{def}}{=}
\norm{x}_X^p.
\end{equation}
The function $\Phi_p$ is continuous and convex, hence plurisubharmonic. Let $P\co X\to X$ be a contractive projection. Then
\[
\Phi_p(Px)
\leq
\Phi_p(x),
\quad x\in X.
\]

Suppose that $a\in\Ran P$ and that $\Phi_p$ is of class $\C^2$ on an open neighborhood $U$ of $a$. Set
\[
\Omega
\ov{\mathrm{def}}{=}
U\cap P^{-1}(U).
\]
Then $\Omega$ is an open neighborhood of $a$. Moreover, $P(\Omega)\subset\Omega$. Indeed, if $x\in\Omega$, then $Px\in U$ and
\[
P(Px)=Px\in U,
\]
so that $Px\in P^{-1}(U)$. Hence $Px\in\Omega$.

Therefore Proposition~\ref{prop-Levi-contractive-projection} applies to the restriction of $\Phi_p$ to $\Omega$ and gives
\begin{equation}
\label{Levi-p-contractive}
\scr L_{\Phi_p,a}(Pz)
\leq
\scr L_{\Phi_p,a}(z),
\quad
z \in X.
\end{equation}
Thus every point of the range at which the power function $\Phi_p$ is locally of class $\C^2$ gives rise to a tangent Hilbertian geometry for which $P$ is contractive, and hence orthogonal after passing to the associated Hilbert completion.

\section{The case of noncommutative $\L^4$-spaces}
\label{sec-the-case-NC}

\subsection{The Levi form of $\norm{\cdot}_{\L^4(\cal{M})}^4$}

We now apply the preceding general principle to noncommutative $\L^4$-spaces. In this case, the Levi form admits an explicit algebraic expression. This is the point where the general Banach-space geometry begins to recover the Jordan triple structure of the ambient von Neumann algebra. Recall that $\Phi_4(x)
\ov{\mathrm{def}}{=} \norm{x}_{\L^4(\cal{M})}^4$ for any $x \in \L^4(\cal{M})$. The following exact fourth-moment formula is the main reason why the case $p=4$ is algebraically accessible.

\begin{lemma}
\label{lem-circular-fourth-moment}
Let $a,z \in \L^4(\cal{M})$ and let $r \geq 0$. Then
\begin{equation}
\label{circular-fourth-moment}
\frac1{2\pi}\int_0^{2\pi}
\norm{a+r\e^{\i t}z}_{\L^4(\cal{M})}^4 \d t
= \norm{a}_{\L^4(\cal{M})}^4
+ 2r^2\left(
\norm{a^*z}_{\L^2(\cal{M})}^2+\norm{za^*}_{\L^2(\cal{M})}^2
\right)
+ r^4\norm{z}_{\L^4(\cal{M})}^4.
\end{equation}
\end{lemma}

\begin{proof}
Put $A\ov{\mathrm{def}}{=}a^*a$, $B\ov{\mathrm{def}}{=}a^*z$ and $C\ov{\mathrm{def}}{=}z^*z$. Then
\[
(a+r\e^{\i t}z)^*(a+r\e^{\i t}z)
=A+r\e^{\i t}B+r\e^{-\i t}B^*+r^2C.
\]
By H\"older's inequality, $A,B,C$ belong to $\L^2(\cal M)$, and every product occurring after the square is expanded belongs to $\L^1(\cal M)$. Thus the expansion is an identity in $\L^1(\cal M)$, and we may apply the trace and integrate its finitely many terms separately. Integration over $t$ annihilates every term with a nonzero Fourier degree. Hence
\begin{align*}
\MoveEqLeft
\frac1{2\pi}\int_0^{2\pi} \norm{a+r \e^{\i t}z}_{\L^4(\cal{M})}^4 \d t 
\ov{\eqref{Def-norm-Lp}}{=}\frac1{2\pi}\int_0^{2\pi} \tr\big[(a+r \e^{\i t}z)^*(a+r\e^{\i t}z)\big]^2 \d t \\
&=
\tr(A^2)
+r^2\tr(AC+CA+BB^*+B^*B)
+r^4\tr(C^2).
\end{align*}
By the tracial property, we have
\[
\tr(AC)=\tr(CA)=\norm{za^*}_{\L^2(\cal M)}^2
\quad\text{and}\quad
\tr(BB^*)=\tr(B^*B)=\norm{a^*z}_{\L^2(\cal M)}^2.
\]
Together with $\tr(A^2)=\norm{a}_4^4$ and $\tr(C^2)=\norm{z}_4^4$, this proves \eqref{circular-fourth-moment}.
\end{proof}

By H\"older's inequality, the real four-linear form $\L^4(\cal{M})^4 \to \mathbb{R}$, $(x_1,x_2,x_3,x_4) \mapsto \Re \Tr(x_1^*x_2x_3^*x_4)$ is continuous on $\L^4(\cal{M})^4$. Consequently, $\Phi_4(x)=\Tr(x^*xx^*x)$ is a continuous real polynomial of degree four, and in particular is of class $\C^\infty$.

\begin{cor}
\label{Levi-form-noncom-L4}
For any $a \in \L^4(\cal{M})$, we have
\begin{equation}
\label{Levi-L4-diagonal}
\scr{L}_{\Phi_4,a}(z)
=2\left(\norm{a^*z}_{\L^2(\cal{M})}^2+\norm{za^*}_{\L^2(\cal{M})}^2\right), \quad z \in \L^4(\cal{M}),
\end{equation}
and
\begin{equation}
\label{Levi-L4-polar}
\scr{L}_{\Phi_4,a}(y,z)
=2\left(\tr(z^*aa^*y)+\tr(az^*ya^*)\right), \quad y,z \in \L^4(\cal{M}).
\end{equation}
\end{cor}

\begin{proof}
We have
\begin{align*}
\MoveEqLeft
\scr{L}_{\Phi_4,a}(z)
\ov{\eqref{Levi-average-limit}}{=} \lim_{r \to 0}\frac1{r^2}
\left(
\frac1{2\pi}\int_0^{2\pi}\Phi_4(a+r\e^{\i t}z) \d t-\Phi_4(a)
\right) \\        
&\ov{\eqref{circular-fourth-moment}}{=} \lim_{r \to 0}\frac1{r^2}
\left(
2r^2(\norm{a^*z}_2^2+\norm{za^*}_2^2)
+ r^4\norm{z}_4^4
\right) 
=2\left(\norm{a^*z}_{\L^2(\cal{M})}^2+\norm{za^*}_{\L^2(\cal{M})}^2\right).
\end{align*}
The second formula follows by polarization.
\end{proof}

\subsection{Levi form and Peirce decomposition of a partial isometry}
\label{sec-tangent-Peirce}

There is a particularly transparent interpretation of the tangent Hilbert form at a tripotent. If $w \in \cal{M}$ is a tripotent, then, by Example~\ref{ex-trip-partial-iso}, $w$ is a partial isometry. Set $e \ov{\mathrm{def}}{=} ww^*$ and $f \ov{\mathrm{def}}{=} w^*w$. The multiplication operator $\mathrm{D}(w,w) \co \cal{M}\to \cal{M}$ associated with $w$ is
\begin{equation}
\label{Peirce-operator-partial-isometry}
\mathrm{D}(w,w)z
\ov{\eqref{Left-mult-JBW-triple}}{=} \{w,w,z\}
\ov{\eqref{triple product-C-star}}{=} \frac12(ww^*z+ zw^*w)
=\frac12(ez+zf), \quad z \in \cal M.
\end{equation}
The corresponding Peirce decomposition \eqref{Peirce-2} is especially explicit. If $e^\perp=1-e$ and $f^\perp=1-f$, the Peirce projections are given by
\begin{align*}
\MoveEqLeft
\P_2^w(z)         
=\mathrm{D}(w,w)(2 \mathrm{D}(w,w) -\Id_{\cal{M}})(z) 
\ov{\eqref{Peirce-operator-partial-isometry}}{=} \mathrm{D}(w,w)(ez+zf-z)\\
&\ov{\eqref{Peirce-operator-partial-isometry}}{=}\frac12\big(e(ez+zf-z)+(ez+zf-z)f\big)
=ezf.
\end{align*}
Similarly, we have
\begin{equation}
\label{Peirce-projections-partial-isometry}
\P_1^w(z)
=ezf^\perp+e^\perp zf, 
\quad \text{and} \quad
\P_0^w(z)
= e^\perp zf^\perp.
\end{equation}
We introduce the corresponding $\L^2$-Peirce subspaces 
\[
\cal{H}_2(w) \ov{\mathrm{def}}{=} e\L^2(\cal M)f,\qquad
\cal{H}_1(w) \ov{\mathrm{def}}{=} e\L^2(\cal M)f^\perp\oplus_2e^\perp\L^2(\cal M)f,\qquad
\cal{H}_0(w) \ov{\mathrm{def}}{=} e^\perp\L^2(\cal M)f^\perp.
\]
These subspaces are mutually orthogonal in the Hilbert space $\L^2(\cal{M})$. In the sequel, we consider the operators defined on $\L^2(\cal{M})$ instead of $\cal{M}$. From \eqref{Peirce-operator-partial-isometry}, we obtain
\begin{equation}
\label{Peirce-operator-spectral-decomposition}
\mathrm{D}(w,w)
=\P_2^w+\frac12 \P_1^w.
\end{equation}
In particular, the operator $\mathrm{D}(w,w)$ acts as multiplication by $1$, $\frac12$ and $0$ on $\cal{H}_2(w)$, $\cal{H}_1(w)$ and $\cal{H}_0(w)$, respectively; a value is an eigenvalue precisely when the corresponding subspace is nonzero.

Now, we show that the tangent Hilbert form at a tripotent is exactly the quadratic form associated with four times the Peirce operator. 

\begin{prop}
Suppose that $w \in \L^4(\cal{M}) \cap \cal{M}$ is a tripotent. The restriction of the Levi form $\scr L_{\Phi_4,w}$ to $\L^4(\cal M)\cap\L^2(\cal M)$ extends uniquely to a continuous Hermitian form on $\L^2(\cal M)$, and
\begin{equation}
\label{tangent-form-Peirce-operator}
\scr{L}_{\Phi_4,w}(z,y)
=4\big\la \mathrm{D}(w,w)z,y \big\ra_{\L^2(\cal{M})},
\quad y,z \in \L^2(\cal{M}).
\end{equation}
Moreover, we have
\begin{align}
\label{tangent-form-Peirce-decomposition}
\MoveEqLeft
\scr{L}_{\Phi_4,w}(z,z)
= 4\norm{\P_2^w(z)}_{\L^2(\cal{M})}^2+2\norm{\P_1^w(z)}_{\L^2(\cal{M})}^2,\quad z \in \L^2(\cal{M}).
\end{align}
\end{prop}

\begin{proof}
For $y,z\in\L^4(\cal M)\cap\L^2(\cal M)$, the equalities $ww^*=e$ and $w^*w=f$, the tracial property and \eqref{Levi-L4-polar} give
\begin{align*}
\MoveEqLeft
\scr{L}_{\Phi_4,w}(z,y)
\ov{\eqref{Levi-L4-polar}}{=}2\tr(y^*ww^*z)+2\tr(wy^*zw^*) 
=2\tr(y^*ez)+2\tr(y^*zf) \\
&=2\tr(y^*(ez+zf)) 
\ov{\eqref{Peirce-operator-partial-isometry}}{=}4\big\la \mathrm{D}(w,w)z,y \big\ra_{\L^2(\cal{M})}. 
\end{align*}
The last expression defines a continuous Hermitian form on $\L^2(\cal M)$. Since $\L^4(\cal M)\cap\L^2(\cal M)$ is dense in $\L^2(\cal M)$, it is the unique continuous extension of the Levi form and proves \eqref{tangent-form-Peirce-operator}.

Write $
z=z_2+z_1^\ell+z_1^r+z_0$, where
\[
z_2 \ov{\mathrm{def}}{=} ezf,\qquad
z_1^\ell \ov{\mathrm{def}}{=}  ezf^\perp,\qquad
z_1^r \ov{\mathrm{def}}{=}  e^\perp zf,\qquad
z_0 \ov{\mathrm{def}}{=}  e^\perp zf^\perp.
\]
Since these four vectors are mutually orthogonal in the Hilbert space $\L^2(\cal{M})$, we have $ez=z_2+z_1^\ell$ and $zf=z_2+z_1^r$. Consequently, we have
\begin{align*}
\MoveEqLeft
\scr{L}_{\Phi_4,w}(z,z)
=2\norm{ez}_2^2+2\norm{zf}_2^2
=4\norm{z_2}_2^2+2\norm{z_1^\ell}_2^2+2\norm{z_1^r}_2^2
=4\norm{\P_2^w(z)}_2^2+2\norm{\P_1^w(z)}_2^2.
\end{align*}
This proves \eqref{tangent-form-Peirce-decomposition}.
\end{proof}

In particular, we have
\begin{equation}
\label{kernel-tangent-Peirce-zero}
\{z \in \L^2(\cal{M}) : \scr{L}_{\Phi_4,w}(z,z)=0\}
=\cal{H}_0(w).
\end{equation}
Hence the $\L^2$-Peirce-$0$ component is precisely the part which is invisible to the tangent Hilbert geometry at $w$, whereas the $\L^2$-Peirce-$2$ and $\L^2$-Peirce-$1$ components are detected with respective weights $4$ and $2$. Equivalently, $4\mathrm{D}(w,w)$ acts with the weights $4$, $2$ and $0$ on the three Peirce components, omitting any value whose component is zero. This gives a concrete link between the infinitesimal geometry of the $\L^4$-norm and the Peirce geometry of the ambient $\JB^*$-triple. In particular, any $z \in e\L^2(\cal M)f$ belongs to the $\L^2$-Peirce-$2$ subspace and
\begin{equation}
\label{full-tripotent-tangent-form}
\scr{L}_{\Phi_4,w}(z,y)
=4\la z,y \ra_{\L^2(\cal{M})},
\quad
z,y \in e\L^2(\cal{M})f.
\end{equation}
This explains why the existence of a full partial isometry in the range turns the tangent Hilbert geometry into the ordinary $\L^2$-geometry.

\section{Contractive projections on noncommutative $\L^4$-spaces}
\label{sec-L4}

\subsection{First order geometry and triple product}

Let $\cal{M}$ be a von Neumann algebra and let $P\co \L^4(\cal{M}) \to \L^4(\cal{M})$ be a bounded projection. We set $Y \ov{\mathrm{def}}{=} \Ran P$ and $Z \ov{\mathrm{def}}{=} \ker P$. We use the sesquilinear duality
\[
\la x,\eta\ra
\ov{\mathrm{def}}{=}
\tr(\eta^*x),
\quad
x \in \L^4(\cal{M}),
\eta \in \L^{\frac43}(\cal{M}).
\]
We denote by $P^\dagger \co \L^{\frac43}(\cal{M}) \to \L^{\frac43}(\cal{M})$ the adjoint map determined by
\begin{equation}
\label{def-adjoint-P-L4}
\tr\big((P^\dagger y)^*x\big)
=
\tr(y^*P(x)),
\quad
x \in \L^4(\cal{M}), y\in\L^{\frac43}(\cal{M}).
\end{equation}
Thus $P^\dagger$ is a projection, $\Ran P^\dagger=(\ker P)^\perp$, and we also set
\begin{equation}
\label{range-adjoint-annihilator}
Y^\sharp
\ov{\mathrm{def}}{=}\Ran P^\dagger=(\ker P)^\perp.
\end{equation}
For any element $x \in \L^4(\cal{M})$ with polar decomposition $x=u|x|$, define 
\begin{equation}
\label{duality-map-L4-recalled}
\frak{j}_4(x)
\ov{\mathrm{def}}{=}\frak{j}_{4,\L^{4}(\cal{M})}(x)  
=u|x|^3
=x x^* x \in \L^{\frac43}(\cal{M}).
\end{equation}
We have
\begin{equation}
\label{duality-map-norm-L4}
\norm{\mathfrak j_4(x)}_{\frac43}
=
\norm{x}_4^3
\quad\text{and}\quad
\tr(\mathfrak j_4(x)^*x)
=
\norm{x}_4^4.
\end{equation}
Likewise, if $y \in \L^{\frac43}(\cal{M})$ with polar decomposition $y=v|y|$, we set
\[
\mathfrak j_{\frac43}(y)
\ov{\mathrm{def}}{=} v|y|^{\frac13} \in \L^4(\cal{M}).
\]
Then
\begin{equation}
\label{inverse-duality-maps}
\mathfrak j_{\frac43}\mathfrak j_4(x)
\ov{\eqref{inverse-power-duality-mapping-NC}}{=} x
\quad \text{and} \quad
\mathfrak j_4\mathfrak j_{\frac43}(y)
\ov{\eqref{inverse-power-duality-mapping-NC}}{=}y.
\end{equation}

Recall that by \cite[(4.1.25) p.~465]{CGRP14} a $\C^*$-algebra $\cal{A}$ becomes a $\JB^*$-triple under its own norm and the triple product $\{x,y,z\} \ov{\mathrm{def}}{=} \frac12(xy^*z+zy^*x)$, where $x,y,z \in \cal{A}$. Inspired by this result, we introduce the $\L^4$-Jordan triple product $\{\cdot,\cdot,\cdot\} \co \L^4(\cal{M}) \times \L^4(\cal{M}) \times \L^4(\cal{M}) \to \L^{\frac43}(\cal{M})$ defined by the same formula
\begin{equation}
\label{ambient-Jordan-triple-product}
\{x,y,z\}
\ov{\mathrm{def}}{=}
\frac12(xy^*z+zy^*x),
\quad
x,y,z \in \L^4(\cal{M}).
\end{equation}
Note this map is well-defined by H\"older's inequality.

\begin{prop}
\label{prop-contractivity-cubic-characterization}
Let $P \co \L^4(\cal{M}) \to \L^4(\cal{M})$ be a bounded projection. The following assertions are equivalent.
\begin{enumerate}
\item The projection $P$ is contractive.
\item For any $x \in \Ran P$, we have $\frak{j}_4(x) \in \Ran P^\dagger$, i.e.,
\begin{equation}
\label{duality-vector-fixed-adjoint}
P^\dagger(\frak{j}_4(x))=\frak{j}_4(x).
\end{equation}
\item For any $x \in \Ran P$ and any $w \in \ker P$, we have
\begin{equation}
\label{cubic-orthogonality-range-kernel}
\tr\big((x x^* x)^*w\big)
=0.
\end{equation}
\item For any $x,y,z \in \Ran P$ and any $w \in \ker P$, we have
\begin{equation}
\label{triple-orthogonality-range-kernel}
\tr\big(\{x,y,z\}^* w\big)
=0.
\end{equation}
\end{enumerate}
If these conditions hold, then
\begin{equation}
\label{range-adjoint-duality-image}
\frak{j}_4(\Ran P)
=\Ran P^\dagger.
\end{equation}
In particular, although $\frak{j}_4$ is nonlinear, its range $\frak{j}_4(\Ran P)$ is a linear subspace of the Banach space $\L^{\frac43}(\cal{M})$.
\end{prop}

\begin{proof}
1 $\Rightarrow$ 2. Under the conjugate linear isometric identification of $\L^{\frac43}(\cal M)$ with $(\L^4(\cal M))^*$ induced by the trace pairing, the Banach adjoint $P^*$ corresponds to $P^\dagger$. Hence the implication follows from Lemma~\ref{Lemma-smooth}.

2 $\iff$ 3. By \eqref{range-adjoint-annihilator}, condition \eqref{duality-vector-fixed-adjoint} is equivalent to
$\frak{j}_4(x) \in (\ker P)^\perp$, i.e., $x x^* x \in (\ker P)^\perp$, which is precisely \eqref{cubic-orthogonality-range-kernel}. Thus the second and third assertions are equivalent.

3 $\Rightarrow$ 1. We next show that the third assertion implies the first one. Let $x \in \L^4(\cal{M})$ and write
\[
x=y+z,
\quad
y=P(x)\in \Ran P,\quad z=x-P(x) \in \ker P.
\]
If $y\neq0$, then \eqref{cubic-orthogonality-range-kernel} and H\"older's inequality yield
\[
\norm{y}_4^4
=
\tr(\frak{j}_4(y)^*y)
=
\tr(\mathfrak j_4(y)^*x)
\leq
\norm{\mathfrak j_4(y)}_{\frac43}\norm{x}_4
=
\norm{y}_4^3\norm{x}_4.
\]
Consequently, $\norm{P(x)}_4=\norm{y}_4 \leq \norm{x}_4$. The same conclusion is trivial if $y=0$. Hence $P$ is contractive.

3 $\iff$ 4. It remains to compare the third and fourth assertions. Direct expansion gives the polarization identity
\begin{equation}
\label{polarization-cubic-map}
\{x,y,z\}
=
\frac1{32}
\sum_{\alpha,\beta \in \{1,-1,\i,-\i \}}
\alpha\overline{\beta}\,
\mathfrak j_4(x+\alpha y+\beta z).
\end{equation}

Indeed, expanding the right-hand side and summing over the fourth roots of unity eliminates every term except
\[
\frac12(xy^*z+zy^*x).
\]
Since $\Ran P$ is a subspace, condition \eqref{cubic-orthogonality-range-kernel} and \eqref{polarization-cubic-map} imply \eqref{triple-orthogonality-range-kernel}. Conversely, taking $x=y=z$ in \eqref{triple-orthogonality-range-kernel} gives \eqref{cubic-orthogonality-range-kernel}. This proves the equivalence of the four assertions.

Finally, we already know that $\mathfrak j_4( \Ran P) \subset \Ran P^\dagger$. Let $y \in \Ran P^\dagger$. Since $\norm{P^\dagger}=\norm{P}\leq1$, Lemma~\ref{Lemma-smooth}, applied to the contractive projection $P^\dagger \co \L^{\frac43}(\cal{M}) \to \L^{\frac43}(\cal{M})$, gives
\[
\frak{j}_{\frac43}(y)
\in
\Ran (P^\dagger)^\dagger
=
\Ran P.
\]
Here $(P^\dagger)^\dagger=P$ under the canonical reflexive identification coming from the trace duality.
Using \eqref{inverse-duality-maps}, we obtain
\[
y
\ov{\eqref{inverse-duality-maps}}{=}
\frak{j}_4\big(\frak{j}_{\frac43}(y)\big)
\in\frak{j}_4(\Ran P).
\]
This proves \eqref{range-adjoint-duality-image}.
\end{proof}

\begin{cor}
\label{cor-dual-triple-closure}
Let $P \co \L^4(\cal{M}) \to \L^4(\cal{M})$ be a contractive projection. Then
\begin{equation}
\label{dual-triple-closure}
\{\Ran P,\Ran P,\Ran P\}
= \Ran P^\dagger
=\frak {j}_4(\Ran P).
\end{equation}
Equivalently, for any $x,y,z \in \Ran P$, we have
\begin{equation}
\label{triple-fixed-by-adjoint}
P^\dagger\left(\frac12(xy^*z+zy^*x)\right)
=\frac12(xy^*z+zy^*x).
\end{equation}
\end{cor}

\begin{proof}
Proposition~\ref{prop-contractivity-cubic-characterization} gives the inclusion from left to right. Conversely, any element of $\Ran P^\dagger$ has the form $\frak j_4(y)=\{y,y,y\}$ for some $y\in\Ran P$, by \eqref{range-adjoint-duality-image}.
\end{proof}

\subsection{First order geometry and $2$-contractive projections on $\L^4$-spaces}

For $n \geq 1$, let $P_n \ov{\mathrm{def}}{=} \Id_{S_n^4} \ot P \co \L^4(\M_n(\cal M)) \to \L^4(\M_n(\cal{M}))$. Under the canonical trace duality, we have $
P_n^\dagger
=
\Id_{S_n^{\frac43}} \ot P^\dagger$ and hence
\begin{equation}
\label{range-adjoint-amplification}
\Ran P_n^\dagger
=
\M_n(\Ran P^\dagger).
\end{equation}

\begin{thm}
\label{thm-two-contractive-completely-contractive-L4}
Let $\cal{M}$ be a von Neumann algebra and let $P \co \L^4(\cal{M}) \to \L^4(\cal{M})$ be a bounded projection. The following assertions are equivalent.
\begin{enumerate}
\item $P$ is $2$-contractive.
\item For any $x,y,z \in \Ran P$, we have
\begin{equation}
\label{unsymmetrized-triple-dual-range}
xy^*z \in \Ran P^\dagger.
\end{equation}
Equivalently,
\begin{equation}
\label{unsymmetrized-triple-fixed}
P^\dagger(xy^*z)=xy^*z,
\quad
x,y,z \in \Ran P.
\end{equation}
\item $P$ is completely contractive.
\item For any integer $n \geq 1$, we have
\begin{equation}
\label{matrix-duality-map-range}
xx^*x\in \M_n(\Ran P^\dagger),
\quad
x \in \M_n(\Ran P).
\end{equation}
\end{enumerate}
In this case, for any $n \geq 1$,
\begin{equation}
\label{matrix-duality-map-equality}
\{xx^*x : x \in \M_n(\Ran P)\}
=
\M_n(\Ran P^\dagger).
\end{equation}
\end{thm}

\begin{proof}
 $3 \Rightarrow 1$. It is immediate.

$1 \Rightarrow 2$. Suppose that $P$ is $2$-contractive. Then $P_2 \co \L^4(\M_2(\cal{M})) \to \L^4(\M_2(\cal{M}))$ is a contractive projection. By Proposition~\ref{prop-contractivity-cubic-characterization}, we have
\[
aa^*a \in \Ran P_2^\dagger
=
\M_2(\Ran P^\dagger),
\quad
a \in \M_2(\Ran P).
\]
Since $\M_2(\Ran P^\dagger)$ is linear, the polarization formula \eqref{polarization-cubic-map}, applied in $\M_2(\cal M)$, gives
\begin{equation}
\label{matrix-Jordan-triple-dual-range}
\frac12(ab^*c+cb^*a) \in \M_2(\Ran P^\dagger),
\quad
a,b,c \in \M_2(\Ran P).
\end{equation}

Let $x,y,z \in \Ran P$ and consider the elements $
a\ov{\mathrm{def}}{=}\begin{bmatrix}
 0    &  x \\
  0   &  0 \\
\end{bmatrix}$, $b \ov{\mathrm{def}}{=} \begin{bmatrix}
   0  &  0 \\
   0  &  y \\
\end{bmatrix}$ and $c \ov{\mathrm{def}}{=}
\begin{bmatrix}
    0& 0  \\
   z  & 0  \\
\end{bmatrix}$.
Then
$$
ab^*c
=\begin{bmatrix}
 0    &  x \\
  0   &  0 \\
\end{bmatrix}\begin{bmatrix}
   0  &  0 \\
   0  &  y \\
\end{bmatrix}^*\begin{bmatrix}
    0& 0  \\
   z  & 0  \\
\end{bmatrix}
=\begin{bmatrix}
  0  & xy^*  \\
  0   & 0  \\
\end{bmatrix}\begin{bmatrix}
    0& 0  \\
   z  & 0  \\
\end{bmatrix}
=\begin{bmatrix}
  xy^*z   &  0 \\
   0  &  0 \\
\end{bmatrix}.
$$
and
$$
cb^*a
=\begin{bmatrix}
    0& 0  \\
   z  & 0  \\
\end{bmatrix}
\begin{bmatrix}
   0  &  0 \\
   0  &  y \\
\end{bmatrix}^*
\begin{bmatrix}
 0    &  x \\
  0   &  0 \\
\end{bmatrix}
=\begin{bmatrix}
  0   &  0 \\
  0   &  0 \\
\end{bmatrix}\begin{bmatrix}
 0    &  x \\
  0   &  0 \\
\end{bmatrix}
=0.
$$
It follows from \eqref{matrix-Jordan-triple-dual-range} that
\[
\frac12\begin{bmatrix}
xy^*z&0\\
0&0
\end{bmatrix}
=\frac12(ab^*c+cb^*a)
\in \M_2(\Ran P^\dagger).
\]
Since $\M_2(\Ran P^\dagger)$ is a linear space, it follows that $xy^*z \in \Ran P^\dagger$. 

$2 \Rightarrow 3$. Now, assume that \eqref{unsymmetrized-triple-dual-range} holds. Let $n \geq 1$ and let $x=[x_{ij}] \in \M_n(\Ran P)$. For any integers $1 \leq i,l \leq n$, we have
\begin{equation}
\label{entries-cubic-matrix}
(xx^*x)_{il}
=
\sum_{j,k=1}^n x_{ij}x_{kj}^*x_{kl}.
\end{equation}
Each term on the right-hand side belongs to $\Ran P^\dagger$ by assumption. Hence
\[
xx^*x \in \M_n(\Ran P^\dagger)
=
\Ran P_n^\dagger.
\]
Proposition~\ref{prop-contractivity-cubic-characterization}, applied to the projection $P_n$, shows that $P_n$ is contractive. Since this holds for any $n \geq 1$, the projection $P$ is completely contractive. Thus $(2) \Rightarrow(3)$.

The equivalence with the fourth assertion follows from the same argument. Finally, once $P_n$ is contractive, Proposition~\ref{prop-contractivity-cubic-characterization} gives
\[
\Ran P_n^\dagger
=\{xx^*x : x \in \Ran P_n\}.
\]
Using \eqref{range-adjoint-amplification} and $\Ran P_n = \M_n(\Ran P)$, we obtain \eqref{matrix-duality-map-equality}.
\end{proof}

\begin{remark} \normalfont
The matricial hypothesis in Theorem~\ref{thm-two-contractive-completely-contractive-L4} cannot be replaced by scalar contractivity. Indeed, let $\sigma$ be the transpose on $S_2^4$ and let
\[
P
\ov{\mathrm{def}}{=}
\frac12(\Id+\sigma).
\]
The projection $P$ is contractive and its range is the space of symmetric $2\times2$ matrices. Moreover, $P^\dagger$ is the same symmetrization on $S_2^{\frac43}$. If $
x=e_{11}$, $y=e_{12}+e_{21}$, $z=e_{22}$ then $x,y,z\in\Ran P$, whereas
\[
xy^*z=e_{12}\notin\Ran P^\dagger.
\]
Consequently, Theorem~\ref{thm-two-contractive-completely-contractive-L4} shows that $P$ is not $2$-contractive.
\end{remark}

\begin{cor}
\label{cor-dual-TRO-product-L4}
Under the assumptions of Theorem~\ref{thm-two-contractive-completely-contractive-L4}, suppose that one of its equivalent conditions holds. Then
\begin{equation}
\label{dual-range-generated-by-triples}
\Ran P^\dagger
=
\{xy^*z : x,y,z \in \Ran P\}
=
\{yy^*y : y \in \Ran P\}.
\end{equation}
Consequently,
\begin{equation}
\label{kernel-ternary-annihilator}
\ker P
=\{w \in \L^4(\cal M) : \tr((xy^*z)^*w)=0 \text{ for all }x,y,z \in \Ran P\}.
\end{equation}
\end{cor}

\begin{proof}
By Theorem~\ref{thm-two-contractive-completely-contractive-L4}, any product $xy^*z$ belongs to $\Ran P^\dagger$. Conversely, Proposition~\ref{prop-contractivity-cubic-characterization} gives
\[
\Ran P^\dagger
=\{yy^*y : y \in \Ran P\}.
\]
Since each diagonal product $yy^*y$ is a product of the required form, this proves \eqref{dual-range-generated-by-triples}. Equality \eqref{kernel-ternary-annihilator} follows from $\Ran P^\dagger=(\ker P)^\perp$.
\end{proof}

\subsection{The conjugate exponent $\frac43$}

The preceding result has the following dual consequence.

\begin{cor}
\label{cor-dual-two-contractive-L43}
Let $Q\co\L^{\frac43}(\cal M)\to\L^{\frac43}(\cal M)$ be a bounded projection. Then $Q$ is $2$-contractive if and only if it is completely contractive.

Suppose, in addition, that $Q$ is contractive, and set $Z\ov{\mathrm{def}}{=}\Ran Q$. Then
\begin{equation}
\label{dual-range-L43}
\Ran Q^\dagger
=
\frak j_{\frac43}(Z).
\end{equation}
In particular, $\frak j_{\frac43}(Z)$ is a linear subspace of $\L^4(\cal M)$ and
\begin{equation}
\label{dual-Jordan-condition-L43}
\big\{\frak j_{\frac43}(Z),\frak j_{\frac43}(Z),\frak j_{\frac43}(Z)\big\}
=Z.
\end{equation}
Under this scalar contractivity assumption, the following assertions are equivalent:
\begin{enumerate}
\item $Q$ is $2$-contractive;
\item $Q$ is completely contractive;
\item for any $x,y,z\in Z$, we have
\begin{equation}
\label{dual-TRO-condition-L43}
\frak j_{\frac43}(x)\frak j_{\frac43}(y)^*\frak j_{\frac43}(z)
\in Z.
\end{equation}
\end{enumerate}
If these conditions hold, then
\begin{equation}
\label{dual-exact-TRO-image-L43}
Z
=
\big\{\frak j_{\frac43}(x)\frak j_{\frac43}(y)^*\frak j_{\frac43}(z):x,y,z\in Z\big\}.
\end{equation}
\end{cor}

\begin{proof}
Let
\[
P\ov{\mathrm{def}}{=}Q^\dagger
\co\L^4(\cal M)\to\L^4(\cal M)
\]
be the adjoint of $Q$ under the trace duality. Thus
\begin{equation}
\label{adjoint-Q-L43}
\tr\big(P(a)^*x\big)
=
\tr\big(a^*Q(x)\big),
\quad
x \in \L^{\frac43}(\cal M), a \in\L^4(\cal M).
\end{equation}
The map $P$ is a bounded projection and, by reflexivity, $P^\dagger=Q$. For any integer $n \geq 1$, set
\[
Q_n\ov{\mathrm{def}}{=}\Id_{S_n^{\frac43}}\ot Q
\quad\text{and}\quad
P_n\ov{\mathrm{def}}{=}\Id_{S_n^4}\ot P.
\]
Under the duality associated with $\Tr_n\ot\tr$, we have
\begin{equation}
\label{adjoint-amplification-L43}
Q_n^\dagger=P_n.
\end{equation}
Indeed, if $X=[x_{ij}]\in\L^{\frac43}(\M_n\otvn\cal M)$ and $A=[a_{ij}]\in\L^4(\M_n\otvn\cal M)$, then
\begin{align*}
(\Tr_n\ot\tr)\big(A^*Q_n(X)\big)
=
\sum_{i,j=1}^n\tr\big(a_{ij}^*Q(x_{ij})\big)
\ov{\eqref{adjoint-Q-L43}}{=}
\sum_{i,j=1}^n\tr\big(P(a_{ij})^*x_{ij}\big)
=
(\Tr_n\ot\tr)\big(P_n(A)^*X\big).
\end{align*}
Consequently, we have
\begin{equation}
\label{amplification-norm-duality-L43}
\norm{Q_n}
=
\norm{P_n},
\quad n \geq 1.
\end{equation}
It follows that $Q$ is $2$-contractive if and only if $P$ is $2$-contractive, and that $Q$ is completely contractive if and only if $P$ is completely contractive. Theorem~\ref{thm-two-contractive-completely-contractive-L4}, applied to $P$, proves the first assertion.

Suppose now that $Q$ is contractive. Then $P$ is contractive by \eqref{amplification-norm-duality-L43}. Set $Y\ov{\mathrm{def}}{=}\Ran P$. Since $P^\dagger=Q$, Proposition~\ref{prop-contractivity-cubic-characterization} gives
\[
\frak j_4(Y)
=
\Ran P^\dagger
=Z.
\]
Applying $\frak j_{\frac43}$ and using \eqref{inverse-duality-maps}, we obtain
\[
Y
=
\frak j_{\frac43}(Z).
\]
Since $Y=\Ran Q^\dagger$, this proves \eqref{dual-range-L43}. Corollary~\ref{cor-dual-triple-closure}, applied to $P$, yields
\[
\{Y,Y,Y\}
=
\Ran P^\dagger
=Z,
\]
which is precisely \eqref{dual-Jordan-condition-L43}.

Finally, Theorem~\ref{thm-two-contractive-completely-contractive-L4} gives
\[
P\text{ is $2$-contractive}
\quad\Longleftrightarrow\quad
YY^*Y\subset\Ran P^\dagger.
\]
Using $Y=\frak j_{\frac43}(Z)$, $\Ran P^\dagger=Z$ and \eqref{amplification-norm-duality-L43}, this is exactly the equivalence of the three assertions. Under these conditions, Corollary~\ref{cor-dual-TRO-product-L4}, applied to $P$, also gives \eqref{dual-exact-TRO-image-L43}.
\end{proof}

\subsection{Left and right tangent Hilbert spaces and 2-contractive projections}
\label{sec-left-right-contractivity}

By \eqref{Levi-L4-polar}, we have
\begin{equation}
\scr{L}_{\Phi_4,a}(y,z)
=2\left(\tr(z^*aa^*y)+\tr(az^*ya^*)\right), \quad y,z \in \L^4(\cal{M}).
\end{equation}
In particular, we have
\begin{equation}
\label{norm-Levi-2}
\frac12\scr L_{\Phi_4,a}(z,z)
=\norm{a^*z}_{\L^2(\cal{M})}^2+\norm{za^*}_{\L^2(\cal{M})}^2.
\end{equation}
We introduce the following positive semidefinite Hermitian forms
\begin{equation}
\label{left-right-inner}
\la z,w \ra_{\ell,a}
\ov{\mathrm{def}}{=}
\tr(w^*aa^*z)
\quad \text{and} \quad
\la z,w \ra_{r,a}
\ov{\mathrm{def}}{=}
\tr(aw^*za^*).
\end{equation}
By quotienting and completing, we obtain two Hilbert spaces $\cal{H}_{\ell,a}$ and $\cal{H}_{r,a}$ with norms
\begin{equation}
\norm{z}_{\ell,a}
=\norm{a^*z}_{\L^2(\cal{M})},
\quad \text{and} \quad
\norm{z}_{r,a}
=\norm{za^*}_{\L^2(\cal{M})}.
\end{equation}

The next result gives separate left and right tangent estimates. This is stronger than the estimate obtained from scalar contractivity, which only controls the sum of the two forms. The additional $2$-contractivity assumption separates the two sides.

\begin{prop}
\label{prop-left-right-tangent-estimates-two-contractive}
Let $P \co \L^4(\cal{M}) \to \L^4(\cal{M})$ be a $2$-contractive projection. For any $a \in \Ran P$ and any $z \in \L^4(\cal{M})$, we have
\begin{equation}
\label{separate-left-right-tangent-estimates}
\norm{P(z)}_{\ell,a}
\leq
\norm{z}_{\ell,a}
\quad \text{and} \quad
\norm{P(z)}_{r,a}
\leq
\norm{z}_{r,a}.
\end{equation}
Moreover, for any $a,y \in \Ran P$, we have
\begin{equation}
\label{left-right-Riesz-vectors-dual-range}
P^\dagger(aa^*y)
=aa^*y
\quad \text{and} \quad
P^\dagger(ya^*a)
=ya^*a.
\end{equation}
\end{prop}

\begin{proof}
Let $a \in \Ran P$ and let $z \in \L^4(\cal{M})$. For any $t > 0$, consider the element $X_t \ov{\mathrm{def}}{=}
\begin{bmatrix}
a & tz\\
0 & 0
\end{bmatrix}$ of the space $\L^4(\M_2(\cal{M}))$. Since the projection $P_2 \co \L^4(\M_2(\cal M)) \to \L^4(\M_2(\cal{M}))$ is contractive and $P(a)=a$, we have
\[
\norm{
\begin{bmatrix}
a & tP(z)\\
0 & 0
\end{bmatrix}}_{\L^4(\M_2(\cal M))}^4
=\norm{P_2(X_t)}_{\L^4(\M_2(\cal M))}^4
\leq \norm{X_t}_{\L^4(\M_2(\cal M))}^4
=\norm{
\begin{bmatrix}
a & tz\\
0 & 0
\end{bmatrix}}_{\L^4(\M_2(\cal M))}^4.
\]
Since $
X_tX_t^*
=
\begin{bmatrix}
aa^*+t^2zz^*&0\\
0&0
\end{bmatrix}$, we obtain
\begin{align*}
\MoveEqLeft
\norm{X_t}_{\L^4(\M_2(\cal M))}^4
=\tr\big[(aa^*+t^2zz^*)^2\big] 
= \norm{a}_{\L^4(\cal{M})}^4
+2t^2\norm{a^*z}_{\L^2(\cal{M})}^2
+t^4\norm{z}_{\L^4(\cal{M})}^4.
\end{align*}
Consequently, we have
\[
\norm{a}_{\L^4(\cal{M})}^4
+
2t^2\norm{a^*P(z)}_{\L^2(\cal{M})}^2
+
t^4\norm{P(z)}_{\L^4(\cal{M})}^4
\leq
\norm{a}_{\L^4(\cal{M})}^4
+
2t^2\norm{a^*z}_{\L^2(\cal{M})}^2
+
t^4\norm{z}_{\L^4(\cal{M})}^4.
\]
After division by $2t^2$ and passage to the limit $t\to 0^+$, we obtain the inequality
\[
\norm{a^*P(z)}_{\L^2(\cal{M})}
\leq
\norm{a^*z}_{\L^2(\cal{M})}.
\]
Using instead the column matrix $Y_t
\ov{\mathrm{def}}{=}
\begin{bmatrix}
a & 0\\
tz & 0
\end{bmatrix}$ and the identity $
Y_t^*Y_t
=
\begin{bmatrix}
a^*a+t^2z^*z&0\\
0&0
\end{bmatrix}$, we similarly obtain the inequality
\[
\norm{P(z)a^*}_{\L^2(\cal{M})}
\leq
\norm{za^*}_{\L^2(\cal{M})}.
\]
This proves \eqref{separate-left-right-tangent-estimates}. Finally, for $a,y \in \Ran P$, the two products $aa^*y$ and $ya^*a$ are fixed by $P^\dagger$ by \eqref{unsymmetrized-triple-fixed}. This proves \eqref{left-right-Riesz-vectors-dual-range}.
\end{proof}

Conversely, either one of the two families of tangent estimates already characterizes complete contractivity.

\begin{prop}
\label{prop-one-sided-tangent-criterion}
Let $P \co \L^4(\mathcal M) \to \L^4(\mathcal M)$ be a bounded projection. The following assertions are equivalent:
\begin{enumerate}
\item $P$ is completely contractive.
\item $\|u^*P(x)\|_{\L^2(\cal{M})}\leq\|u^*x\|_{\L^2(\cal{M})}$ for any $u\in \Ran P$ and any $x\in\L^4(\mathcal M)$.
\item $\|P(x)u^*\|_{\L^2(\cal{M})}\leq\|xu^*\|_{\L^2(\cal{M})}$ for any $u\in \Ran P$ and any $x\in\L^4(\mathcal M)$.
\end{enumerate}
\end{prop}

\begin{proof}
The first assertion implies the other two by Proposition~\ref{prop-left-right-tangent-estimates-two-contractive}. Suppose that the second assertion holds. For fixed $u\in \Ran P$, the projection $P$ induces a contractive, hence orthogonal, projection on the Hilbert completion associated with the positive semidefinite form $h_u(x,z)=\Tr(z^*uu^*x)$. Therefore $h_u(x-Px,z)=0$ for any $z \in \Ran P$ and every $x \in \L^4(\cal{M})$. Equivalently, $uu^*z \in \Ran P^\dagger$. Polarization gives
\begin{equation*}
xy^*z
=\frac14\sum_{k=0}^3 \i^k(x+\i^ky)(x+\i^ky)^*z \in \Ran P^\dagger,
\quad x,y,z \in \Ran P.
\end{equation*}
Theorem~\ref{thm-two-contractive-completely-contractive-L4} implies complete contractivity. The third assertion is treated in the same way, using the forms $k_u(x,z)=\Tr(u z^*x u^*)$ and the polarization formula $xy^*z
=
\frac14\sum_{k=0}^3(-\i)^k
x(y+\i^kz)^*(y+\i^kz)$.
\end{proof}

Recall that, if $x=u|x|$ is the polar decomposition of a measurable operator, its left and right support projections are
\[
s_\ell(x)\ov{\mathrm{def}}{=}uu^*
\quad\text{and}\quad
s_r(x)\ov{\mathrm{def}}{=}u^*u.
\]
If $h$ is positive, we write
\[
s(h)\ov{\mathrm{def}}{=}s_\ell(h)=s_r(h).
\]
For a subspace $Y\subset\L^p(\cal M)$, we set
\[
s_\ell(Y)\ov{\mathrm{def}}{=}\bigvee_{x\in Y}s_\ell(x)
\quad\text{and}\quad
s_r(Y)\ov{\mathrm{def}}{=}\bigvee_{x\in Y}s_r(x).
\]
If $P$ is a projection, we write $s_\ell(P)\ov{\mathrm{def}}{=}s_\ell(\Ran P)$ and $s_r(P)\ov{\mathrm{def}}{=}s_r(\Ran P)$. Thus
\begin{equation}
\label{range-support-corner}
\Ran P
\subset
s_\ell(P)\L^p(\cal M)s_r(P).
\end{equation}

\subsection{Left and right parts of the Levi form}
\label{sec-modular-finite}
\begin{prop}
\label{prop-global-left-right-Levi-finite-dimensional}
Let $\cal M$ be a semifinite von Neumann algebra equipped with a normal semifinite faithful trace. Let
$
P\co\L^4(\cal M)\to\L^4(\cal M)$ be a finite-rank $2$-contractive projection with nonzero range $Y$. Put $
e\ov{\mathrm{def}}{=}s_\ell(Y)$ and $f\ov{\mathrm{def}}{=}s_r(Y)$. Let $(x_1,\ldots,x_k)$ be a linear basis of the space $Y$ and define
\begin{equation}
\label{global-Gram-densities}
a
\ov{\mathrm{def}}{=}
\sum_{j=1}^k x_jx_j^*,
\quad \text{and} \quad
b
\ov{\mathrm{def}}{=}
\sum_{j=1}^k x_j^*x_j.
\end{equation}
Then $a,b \in \L^2(\cal M)$,
\begin{equation}
\label{a-b-supports}
s(a)=s_\ell(P) 
\quad \text{and} \quad s(b)=s_r(P)
\end{equation}
and
\begin{equation}
\label{global-exact-density-ranges}
a(\Ran P)
=\Ran P^\dagger
=(\Ran P)b.
\end{equation}
If, in addition, $\cal M$ is finite-dimensional, then $a$ and $b$ are invertible in the reduced algebras $e\cal Me$ and $f\cal Mf$, respectively.
\end{prop}

%
%

\begin{proof}
The support identities follow from
\[
s(a)
\ov{\eqref{global-Gram-densities}}{=}
\bigvee_{j=1}^k s_\ell(x_j)
= s_\ell(P)
\quad\text{and}\quad
s(b)
\ov{\eqref{global-Gram-densities}}{=}
\bigvee_{j=1}^k s_r(x_j)
= s_r(P).
\]
By Theorem~\ref{thm-two-contractive-completely-contractive-L4}, for $x,y,z \in \Ran P$ we have $xy^*z \ov{\eqref{unsymmetrized-triple-dual-range}}{\in} \Ran P^\dagger$. Consequently, for any $y \in \Ran P$, we have
\[
ay
\ov{\eqref{global-Gram-densities}}{=}
\sum_{j=1}^k x_jx_j^*y
\in \Ran P^\dagger
\quad \text{and} \quad
yb
\ov{\eqref{global-Gram-densities}}{=}
\sum_{j=1}^k yx_j^*x_j
\in \Ran P^\dagger.
\]
Thus the left and right multiplication maps $L_a \co \Ran P \to \Ran P^\dagger$, $y \mapsto ay$ and $R_b \co \Ran P \to \Ran P^\dagger$, $y \mapsto yb$ are well-defined. 

We prove the injectivity without assuming that the densities have bounded inverses. Suppose that $ay=0$ for some $y\in Y$. For $n\geq1$, put $e_n=\mathbf1_{[\frac{1}{n},n]}(a)$ and let $g_n(t)=t^{-1}\mathbf1_{[\frac{1}{n},n]}(t)$, with $g_n(0)=0$. Since $g_n(a) \in \cal M$ and $g_n(a)a=e_n$, we obtain $e_ny=g_n(a)(ay)=0$. The projections $e_n$ increase strongly to $e$, so $e_ny\to ey=y$ in $\L^4(\cal M)$. Thus $y=0$. The same argument on the right shows that $yb=0$ implies $y=0$.

The two multiplication maps from $Y$ into $Y^\sharp$ are therefore injective. Since $P$ has finite rank and $\dim Y^\sharp=\dim Y$, both maps are surjective. This proves \eqref{global-exact-density-ranges}. Finally, if $\cal M$ is finite-dimensional, the support identities imply that $a$ and $b$ are invertible in $e\cal Me$ and $f\cal Mf$, respectively.
\end{proof}

\begin{prop}
\label{prop-global-left-right-Levi-finite-dimensional-bis}
Assume, in addition to the assumptions and notation of Proposition~\ref{prop-global-left-right-Levi-finite-dimensional}, that $\cal M$ is finite-dimensional. Define the forms
\begin{align}
\label{global-left-Levi-form}
\frak{H}^\ell(x,y)
&\ov{\mathrm{def}}{=}
\sum_{j=1}^k \tr(y^*x_jx_j^*x)
\ov{\eqref{global-Gram-densities}}{=} \tr(y^*ax), \quad x,y \in \Ran P, \\
\label{global-right-Levi-form}
\frak{H}^r(x,y)
&\ov{\mathrm{def}}{=}
\sum_{j=1}^k \tr(y^*xx_j^*x_j)
\ov{\eqref{global-Gram-densities}}{=} \tr(y^*xb), \quad x,y \in \Ran P.
\end{align}
These Hermitian forms are positive definite on the subspace $\Ran P$. The relative metric endomorphism defined in Proposition \ref{prop-relative-Riesz-operator} by
\begin{equation}
\label{global-Levi-comparison}
\frak{H}^r(x,y)
=\frak{H}^\ell(x,\Delta y), \quad x,y \in \Ran P,
\end{equation}
is given by
\begin{equation}
\label{global-relative-modular-operator}
\Delta(y)
\ov{\mathrm{def}}{=}
a^{-1}yb,
\quad
y\in \Ran P. 
\end{equation}
Furthermore, we have
\begin{equation}
\label{global-relative-complex-powers}
\Delta^z(y)
=a^{-z} y b^z,
\quad y \in \Ran P, z \in \mathbb{C}.
\end{equation}
\end{prop}

Here $a^{-1}$ denotes the inverse of $a$ in the reduced algebra $e\cal{M}e$. More generally, the powers $a^{-z}$ and $b^z$ in \eqref{global-relative-complex-powers} are computed in the reduced algebras $e\cal{M}e$ and $f\cal{M}f$, respectively.

\begin{proof}
Set $Y \ov{\mathrm{def}}{=} \Ran P$ and $Y^\sharp\ov{\mathrm{def}}{=} \Ran P^\dagger$. Since $a$ and $b$ have full supports, the two forms are positive definite. Indeed, we have
\[
\frak H^\ell(x,x)
=
\tr(x^*ax)
=
\bnorm{a^{\frac12}x}_2^2
\quad\text{and}\quad
\frak H^r(x,x)
=
\tr(x^*xb)
=
\bnorm{xb^{\frac12}}_2^2.
\]
Consider the trace-duality map $\iota \co Y^\sharp \to Y^*$ defined by
$$
[\iota(\eta)](x)
\ov{\mathrm{def}}{=}
\tr(\eta^*x),
\quad x \in Y,\eta \in Y^\sharp.
$$
This map is conjugate linear and injective. Indeed, if $\iota(\eta)=0$ for some $\eta\in Y^\sharp$, then $P^\dagger(\eta)=\eta$ and, for any $z \in \L^4(\cal M)$,
\[
\tr(\eta^*z)
=
\tr((P^\dagger\eta)^*z)
\ov{\eqref{def-adjoint-P-L4}}{=}
\tr(\eta^*Pz)
=
0.
\]
The nondegeneracy of the $\L^4$--$\L^{4/3}$ duality gives $\eta=0$. Since $\dim Y^\sharp=\dim Y$, the map $\iota$ is a conjugate linear bijection from $Y^\sharp$ onto $Y^*$. For $y\in Y$, define
\[
L_a(y)\ov{\mathrm{def}}{=}ay
\quad\text{and}\quad
R_b(y)\ov{\mathrm{def}}{=}yb.
\]
By \eqref{global-exact-density-ranges}, both maps are linear bijections from $Y$ onto $Y^\sharp$. Moreover, for $x,y\in Y$,
\[
[\iota L_a(y)](x)
=
\tr((ay)^*x)
=
\tr(y^*ax)
=
\frak H^\ell(x,y),
\]
whereas
\[
[\iota R_b(y)](x)
=
\tr((yb)^*x)
=
\tr(by^*x)
=
\tr(y^*xb)
=
\frak H^r(x,y).
\]
Thus $\iota L_a$ and $\iota R_b$ are the Riesz maps of $\frak H^\ell$ and $\frak H^r$, respectively. Proposition~\ref{prop-relative-Riesz-operator} gives
\[
\Delta
=
(\iota L_a)^{-1}(\iota R_b)
=
L_a^{-1}R_b.
\]
Consequently, we have
\[
\Delta(y)
=
a^{-1}yb,
\quad
y\in Y,
\]
which proves \eqref{global-relative-modular-operator}. Moreover,
\[
\frak H^r(x,y)
=
\frak H^\ell(x,\Delta y),
\quad
x,y\in Y,
\]
and $\Delta$ is positive and selfadjoint for the Hilbertian structure defined by $\frak H^\ell$. It remains to prove \eqref{global-relative-complex-powers}. On the full corner $e\cal Mf$, define
\[
T(\xi)
\ov{\mathrm{def}}{=}
a^{-1}\xi b.
\]
Equip $e\cal Mf$ with the Hermitian inner product
\[
\langle\xi,\eta\rangle_a
\ov{\mathrm{def}}{=}
\tr(\eta^*a\xi).
\]
For $\xi,\eta \in e\cal Mf$, we have
\begin{align*}
\langle T\xi,\eta\rangle_a
=
\tr(\eta^*a a^{-1}\xi b)
=
\tr(\eta^*\xi b)
\quad \text{and} \quad \langle\xi,T\eta\rangle_a
=
\tr((a^{-1}\eta b)^*a\xi)
=
\tr(b\eta^*\xi)
=
\tr(\eta^*\xi b).
\end{align*}
Hence $T$ is selfadjoint. Furthermore, we have
\[
\la T\xi,\xi \ra_a
=
\tr(\xi^*\xi b)
= \bnorm{\xi b^{\frac12}}_2^2
\geq0,
\]
so $T$ is positive. Equality \eqref{global-exact-density-ranges} gives
\[
T(Y)
=
a^{-1}Yb
=
Y.
\]
Thus $Y$ is an invariant subspace of the selfadjoint operator $T$ and is therefore reducing.
More explicitly, $L_{a^{-1}}$ and $R_b$ are commuting positive selfadjoint operators for $\langle\cdot,\cdot\rangle_a$. Indeed,
\[
\langle L_{a^{-1}}\xi,\eta\rangle_a
=
\tr(\eta^*\xi)
=
\langle\xi,L_{a^{-1}}\eta\rangle_a
\]
and
\[
\langle R_b\xi,\eta\rangle_a
=
\tr(\eta^*a\xi b)
=
\langle\xi,R_b\eta\rangle_a,
\]
while their quadratic forms are respectively $\norm{\xi}_2^2$ and $\bnorm{a^{\frac12}\xi b^{\frac12}}_2^2$. Consequently, the joint functional calculus gives
\[
T^z(\xi)
=
a^{-z}\xi b^z,
\quad
\xi\in e\cal Mf, z\in\mathbb C.
\]
Since $\Delta=T|_Y$, restricting this identity to $Y$ proves \eqref{global-relative-complex-powers}.
\end{proof}

\begin{remark} \normalfont
\label{rem-basis-dependence}
The Gram densities $a,b$ depend on the chosen basis of $\Ran P$ only through the Hermitian inner product on $\Ran P$ for which this basis is orthonormal. In particular, they are unchanged by unitary changes of basis, but not by arbitrary invertible changes of basis. Thus the construction attaches modular data to a choice of Hilbertian structure on $\Ran P$. It does not produce a single relative operator determined by $P$ alone. All the results below hold for any such choice.
\end{remark}

\begin{cor}
\label{cor-modular-complete-isometries}
Under the assumptions and notation of Proposition~\ref{prop-global-left-right-Levi-finite-dimensional-bis}, the linear maps $\Delta^{\i t}\co\Ran P\to\Ran P$, $t\in\R$, are surjective complete isometries for the inherited $\L^4$ operator space structure. Put $u_t \ov{\mathrm{def}}{=} a^{-\i t}+(1-e)$ and $v_t \ov{\mathrm{def}}{=} b^{\i t}+(1-f)$. Then
\[
P(u_txv_t)=u_tP(x)v_t,
\quad x \in \L^4(\cal M),\ t\in\R.
\]
In particular, we have
\[
P(a^{-\i t}xb^{\i t})
=a^{-\i t}P(x)b^{\i t},
\quad x \in e\L^4(\cal M)f.
\]
\end{cor}

\begin{proof}
The identity $\Delta^{\i t}(y)=a^{-\i t}yb^{\i t}$ shows that these maps are restrictions of left and right multiplication by unitaries in the reduced algebras, so they are complete isometries. To obtain the covariance relation on the whole corner, put $u_t=a^{-\i t}+(1-e)$ and $v_t=b^{\i t}+(1-f)$. The map $\mathcal U_t(x)=u_txv_t$ is a surjective complete isometry on $\L^4(\cal{M})$ and preserves $Y$. Thus $\mathcal U_t^{-1}P\cal{U}_t$ is another contractive projection onto $Y$. Uniqueness gives $P \cal{U}_t= \cal{U}_tP$, which restricts to the displayed identity.
\end{proof}

\subsection{2-contractive projections on $\L^4$-spaces and ternary rings of operators}
\label{sec-2-contract-TRO}

\begin{prop}
\label{prop-linking-modular-interpretation}
Under the assumptions and notation of Proposition~\ref{prop-global-left-right-Levi-finite-dimensional-bis}, in particular with $\cal M$ finite-dimensional, consider the corner von Neumann algebra
\[
\cal R
\ov{\mathrm{def}}{=}
\begin{bmatrix}
e\cal{M}e&e\cal{M}f\\
f\cal{M}e&f\cal{M}f
\end{bmatrix}.
\]
Here we identify $\cal R$ with the corner $s(\M_2\otvn\cal M)s$ where $s \ov{\mathrm{def}}{=}
\begin{bmatrix}
e&0\\
0&f
\end{bmatrix}$ and equip it with the faithful trace $
\tr_{\cal R}
\ov{\mathrm{def}}{=}
(\Tr_2\ot\tau)|_{\cal R}$. The corner $e \cal{M} f$ has left support $e$ and right support $f$, so $\cal R$ is its linking von Neumann algebra. Define the positive invertible density
\begin{equation}
\label{linking-density}
D
\ov{\mathrm{def}}{=}
\begin{bmatrix}
a^2&0\\
0&b^2
\end{bmatrix}.
\end{equation}
Let $(\sigma_t^{D})_{t \in \R}$ be the modular automorphism group of the faithful positive functional $\cal{R} \to \mathbb{C}$, $X \mapsto \tr_{\cal R}(DX)$. On the upper-right corner, we have
\begin{equation}
\label{linking-modular-group-corner}
\sigma_t^{D}(y)
=a^{2\i t}yb^{-2\i t},
\quad
y\in e\cal Mf.
\end{equation}
Consequently, the relative operator $\Delta$ of Proposition~\ref{prop-global-left-right-Levi-finite-dimensional-bis} satisfies
\begin{equation}
\label{relative-group-as-linking-modular-group}
\Delta^{\i t}(y)
=
\sigma_{-\frac t2}^{D}(y), 
\quad y \in \Ran P.
\end{equation}
If $\Delta_D$ denotes the standard modular operator on $\L^2(\cal R)$ associated with $D$, then on the upper-right corner
\begin{equation}
\label{standard-linking-modular-operator}
\Delta_{D}(y)
=a^2yb^{-2},
\end{equation}
and hence
\begin{equation}
\label{relative-operator-half-modular}
\Delta
=
\Delta_{D}^{-\frac12}
\end{equation}
as operators on the invariant subspace $\Ran P$.
\end{prop}

\begin{proof}
Since $\Ran P\subset e\cal Mf$ and $e$ is the join of the left supports of the elements of $\Ran P$, the weakly closed ideal generated by $(e\cal Mf)(e\cal Mf)^*$ in $e\cal Me$ has support $e$ and is therefore equal to $e\cal Me$. Similarly, the weakly closed ideal generated by $(e\cal Mf)^*(e\cal Mf)$ is $f\cal Mf$. This proves the linking-algebra assertion.

In the tracial realization of a finite-dimensional von Neumann algebra, the modular automorphism group associated with an invertible density $D$ is
\[
\sigma_t^{D}(X)
=D^{\i t}X D^{-\i t}.
\]
Since $D^{\i t}
\ov{\eqref{linking-density}}{=}
\begin{bmatrix}
a^{2\i t}&0\\
0&b^{2\i t}
\end{bmatrix}$, we obtain
\[
\sigma_t^{D}\left(\begin{bmatrix}
0&y\\
0&0
\end{bmatrix}\right)
=D^{\i t}
\begin{bmatrix}
0&y\\
0&0
\end{bmatrix}
D^{-\i t}
=\begin{bmatrix}
a^{2\i t}&0\\
0&b^{2\i t}
\end{bmatrix}\begin{bmatrix}
0&y\\
0&0
\end{bmatrix}
\begin{bmatrix}
a^{-2\i t}&0\\
0&b^{-2\i t}
\end{bmatrix}
=
\begin{bmatrix}
0&a^{2\i t}yb^{-2\i t}\\
0&0
\end{bmatrix},
\]
which proves \eqref{linking-modular-group-corner}. Taking the parameter $-\frac t2$ gives
\[
\sigma_{-\frac t2}^{D}(y)
=
a^{-\i t}yb^{\i t}
=
\Delta^{\i t}(y),
\]
where we used \eqref{global-relative-complex-powers}. This proves \eqref{relative-group-as-linking-modular-group}. The standard modular operator on the tracial Hilbert space $\L^2(\cal R)$ is
\[
\Delta_{D}
=
L_{D}R_{D^{-1}}.
\]
Its restriction to the upper-right corner is therefore
\[
y\mapsto a^2yb^{-2},
\]
which proves \eqref{standard-linking-modular-operator}. Its inverse square root is $
y\mapsto a^{-1}yb$, which is precisely $\Delta$. This proves \eqref{relative-operator-half-modular}.
\end{proof}

Following \cite{Arh24b} (see also \cite[p.~869]{KaR02}), we recall the rectangular noncommutative $\L^p$-space of a $\W^*$-$\TRO$. Let $\cal V$ be a $\W^*$-$\TRO$, let $\cal R(\cal V)$ be its linking von Neumann algebra, and assume that $\cal R(\cal V)$ is $\sigma$-finite and equipped with a normal faithful state $\varphi$. If $q\in\cal R(\cal V)$ is the projection defining the upper-right corner $\cal V=q\cal R(\cal V)(1-q)$, then
\[
\L^p(\cal V,\varphi)
\ov{\mathrm{def}}{=}
q\L^p(\cal R(\cal V),\varphi)(1-q),
\quad
1 \leq p < \infty,
\]
where the ambient spaces are Kosaki's noncommutative $\L^p$-spaces introduced in \cite{Kos84}. We also set $
\L^\infty(\cal V,\varphi)
\ov{\mathrm{def}}{=}
\cal V$. If, in addition, $q$ belongs to the centralizer of $\varphi$, then the symmetric Kosaki realization preserves the four corners determined by $q$ and $1-q$. In this case, $\L^p(\cal V,\varphi)$ identifies with the closure of the algebraic corner $\cal V=q\cal R(\cal V)(1-q)$ in the symmetric realization of $\L^p(\cal R(\cal V),\varphi)$.

We shall use the following finite-dimensional realization. Let $\cal R$ be a finite-dimensional von Neumann algebra with faithful trace $\tr_{\cal R}$, and let $\psi(X)=\tr_{\cal R}(dX)$ for a positive invertible density $d\in\cal R$. In the symmetric realization of Kosaki's space, an element $X\in\cal R$ represents the vector
\begin{equation}
\label{Kosaki-symmetric-realization}
d^{\frac1{2p}}X d^{\frac1{2p}}
\in\L^p(\cal R,\tr_{\cal R}),
\end{equation}
and its $\L^p(\cal R,\psi)$-norm is the tracial $\L^p$-norm of this vector. We use the same realization at every matrix level, with density $\Id_{\M_n}\ot d$. If $q\in\cal R$ is a projection commuting with $d$, then the symmetric realization preserves the corners determined by $q$. Indeed,
\[
d^{\frac1{2p}}qX(1-q)d^{\frac1{2p}}
=
q\,d^{\frac1{2p}}Xd^{\frac1{2p}}(1-q),
\quad
X \in \cal R.
\]
Equivalently, $q$ belongs to the centralizer of the state $\psi$, and the algebraic upper-right corner is represented by the corresponding upper-right corner of the tracial $\L^p$-space.

The preceding interpretation also explains the quarter-powers appearing in the deweighting procedure. We shall use the following consequence of Takesaki's conditional expectation theorem and the $\L^p$-extension of conditional expectations from \cite[Lemma~2.2]{JuX03}.

\begin{lemma}
\label{lem-modular-invariant-Kosaki-subspace}
Let $\cal N$ be a von Neumann subalgebra of a von Neumann algebra $\cal R$ with the same unit. Let $\psi$ be a normal faithful state on $\cal R$ and suppose that
\begin{equation}
\label{modular-invariance-subalgebra}
\sigma_t^\psi(\cal N)=\cal N,
\quad t\in\R.
\end{equation}
Let $\varphi\ov{\mathrm{def}}{=}\psi|_{\cal N}$. Suppose that  $1 \leq p \leq \infty$. The canonical inclusion $\cal N\subset\cal R$ induces a complete isometry
\begin{equation}
\label{Kosaki-complete-isometric-inclusion}
J_p\co \L^p(\cal N,\varphi) \to \L^p(\cal R,\psi).
\end{equation}
Moreover, the range of $J_p$ is the range of a completely contractive projection. If $\cal R$ is finite-dimensional, let $\tr_{\cal R}$ be a faithful trace, and let $d_{\cal R}\in\cal R$ and $d_{\cal N}\in\cal N$ be the densities of $\psi$ and $\varphi$ with respect to $\tr_{\cal R}$ and $\tr_{\cal R}|_{\cal N}$, respectively. For $1\leq p<\infty$, the maps above are characterized in the symmetric tracial realizations by
\begin{align*}
J_p\big(d_{\cal N}^{\frac1{2p}}x d_{\cal N}^{\frac1{2p}}\big)
&=d_{\cal R}^{\frac1{2p}}x d_{\cal R}^{\frac1{2p}},
&&x\in\cal N,\\
\E_p\big(d_{\cal R}^{\frac1{2p}}x d_{\cal R}^{\frac1{2p}}\big)
&=d_{\cal N}^{\frac1{2p}}\E(x)d_{\cal N}^{\frac1{2p}},
&&x\in\cal R.
\end{align*}
These identities also hold at every matrix level, with the corresponding amplified densities. In particular, the projection $J_p\E_p$ is represented on $\L^p(\cal R,\tr_{\cal R})$ by
\begin{equation*}
\xi\mapsto d_{\cal R}^{\frac1{2p}}
\E\big(d_{\cal R}^{-\frac1{2p}}\xi d_{\cal R}^{-\frac1{2p}}\big)
d_{\cal R}^{\frac1{2p}}.
\end{equation*}
\end{lemma}

\begin{proof}
By Takesaki's theorem \cite[Theorem~IX.4.2]{Tak03}, the modular invariance assumption \eqref{modular-invariance-subalgebra} yields a normal faithful conditional expectation $\E\co\cal R\to\cal N$ such that $\psi\circ E=\psi$. At $p=\infty$, the inclusion and $\E$ are completely contractive and their composition on $\cal N$ is the identity. At $p=1$, the corresponding statement follows by passing to the preadjoint and the restriction map. Since these endpoint maps are compatible with the Kosaki interpolation couples, complex interpolation \cite{BeL76}, or equivalently \cite[Lemma~2.2]{JuX03}, gives completely contractive maps $
J_p\co\L^p(\cal N,\varphi)\to\L^p(\cal R,\psi)$, $
\E_p\co\L^p(\cal R,\psi)\to\L^p(\cal N,\varphi)$ with $\E_pJ_p=\Id_{\L^p(\cal N,\varphi)}$. Indeed, the contractive extensions are the state-preserving extensions of $E$ to noncommutative $\L^p$-spaces, and applying the same construction to $\Id_{\M_n}\ot E$ gives complete contractivity at every matrix level. Consequently, we have
\[
\norm{x}_p
=\norm{\E_pJ_p(x)}_p
\leq\norm{J_p(x)}_p
\leq\norm{x}_p,
\quad x \in \L^p(\cal N,\varphi),
\]
and the same inequalities hold at every matrix level. Thus $J_p$ is a complete isometry. Finally, the map $J_p\E_p$ is a completely contractive projection whose range is $J_p(\L^p(\cal N,\varphi))$.

The formulas in the symmetric realizations follow from the common-core construction in \cite[Section~2, Lemma~2.1 and Lemma~2.2]{JuX03}, taking the symmetric parameter $\theta = \frac{1}{2}$. Applying the same construction to the matrix amplifications gives the corresponding formulas at every matrix level.
\end{proof}

\begin{cor}
\label{cor-finite-dimensional-modular-deweighting}
Let $\cal{M}$ be a finite-dimensional von Neumann algebra. Let $P \co \L^4(\cal M) \to \L^4(\cal M)$ be a $2$-contractive projection with $\Ran P\neq\{0\}$. Under the assumptions and notation of Section~\ref{sec-modular-finite}, let $e \ov{\mathrm{def}}{=} s_\ell(P)$, $f \ov{\mathrm{def}}{=} s_r(P)$ and define the subspace
\begin{equation}
\label{finite-dimensional-deweighted-space}
\cal V
\ov{\mathrm{def}}{=}
a^{-\frac14}(\Ran P)b^{-\frac14}
\end{equation}
of $e\cal M f$. Here and below, all negative powers of $a$ and $b$ are taken in the reduced algebras $e\cal Me$ and $f\cal Mf$, respectively. Then $\cal V$ is a finite-dimensional $\W^*$-$\TRO$ and
\begin{equation}
\label{finite-dimensional-weighted-range}
\Ran P=a^{\frac14}\cal Vb^{\frac14}.
\end{equation}
Consider the von Neumann algebra $
\cal R
\ov{\mathrm{def}}{=}
\begin{bmatrix}
e\cal Me&e\cal Mf\\
f\cal Me&f\cal Mf
\end{bmatrix}$ and let
$D
\ov{\mathrm{def}}{=}
\begin{bmatrix}
a^2&0\\
0&b^2
\end{bmatrix}$. Then the linking von Neumann algebra $\cal R(\cal V)$ is a unital von Neumann subalgebra of $\cal R$ and is invariant under the modular automorphism group associated with $D$:
\begin{equation}
\label{linking-algebra-V-modular-invariance}
\sigma_t^{D}(\cal R(\cal V))
=
\cal R(\cal V),
\quad t \in \R.
\end{equation}
Set
\begin{equation}
\label{normalization-linking-state}
c
\ov{\mathrm{def}}{=}
\tr_{\cal R}(D)
\quad \text{and} \quad
\psi(X)
\ov{\mathrm{def}}{=}
c^{-1}\tr_{\cal R}(D X),
\quad X \in \cal R,
\end{equation}
and let $\varphi\ov{\mathrm{def}}{=}\psi|_{\cal R(\cal V)}$. Then $\varphi$ is a normal faithful state on $\cal R(\cal V)$, and the map defined on the finite-dimensional space $\L^4(\cal{V},\varphi)=\cal V$ by
\begin{equation}
\label{complete-isometry-deweighted-space}
U(v)
\ov{\mathrm{def}}{=}
c^{-\frac14}a^{\frac14}vb^{\frac14},
\quad v \in \cal V,
\end{equation}
is a surjective complete isometry $U \co \L^4(\cal V,\varphi) \to \Ran P$. In particular, the space $\Ran P$ is completely isometric to a rectangular noncommutative $\L^4$-space.
\end{cor}

\begin{proof}
Set $Y\ov{\mathrm{def}}{=}\Ran P$. The identity \eqref{finite-dimensional-weighted-range} follows directly from the definition of $\cal V$ and from $Y=eYf$. We first prove that $\cal V$ is a $\TRO$. For $z \in \mathbb{C}$, define
\begin{equation}
\label{def-Theta-z}
\Theta^z(v)
\ov{\mathrm{def}}{=}
a^{-z}vb^z,
\quad v\in\cal V.
\end{equation}
If $v=a^{-\frac14}yb^{-\frac14}$ with $y\in Y$, then
\[
\Theta^z(v)
=\Theta^z(a^{-\frac{1}{4}}yb^{-\frac{1}{4}})
\ov{\eqref{def-Theta-z}}{=}
a^{-\frac14}a^{-z}yb^zb^{-\frac14}
\ov{\eqref{global-relative-complex-powers}}{=}
a^{-\frac14}\Delta^z(y)b^{-\frac14}.
\]
Since $\Delta^z(\Ran P)=\Ran P$ by \eqref{global-relative-complex-powers}, we obtain $\Theta^z(\cal V)=\cal V$, and $\Theta^z$ is a bijection of $\cal V$ with inverse $\Theta^{-z}$. Let $v_1,v_2,v_3 \in \cal V$ and consider the elements
\begin{equation}
\label{def-des-yi}
y_1
\ov{\mathrm{def}}{=}
a^{\frac14}v_1b^{\frac14}, \quad 
y_2
\ov{\mathrm{def}}{=} a^{\frac14}v_2b^{\frac14}
\quad \text{and} \quad 
y_3 \ov{\mathrm{def}}{=} a^{\frac14}v_3b^{\frac14}
\end{equation}
of the subspace $\Ran P$. By Theorem~\ref{thm-two-contractive-completely-contractive-L4}, the element $
r
\ov{\mathrm{def}}{=}
y_1y_2^*y_3$ belongs to the subspace $\Ran P^\dagger$. Since $\Ran P^\dagger \ov{\eqref{global-exact-density-ranges}}{=} a(\Ran P)$, there exists $x \in \Ran P$ such that 
\begin{equation}
\label{r=ax}
r 
=a x.
\end{equation}
By definition \eqref{finite-dimensional-deweighted-space} of the subspace $\cal{V}$, we have $a^{-\frac14}x b^{-\frac14} \in \cal{V}$. A direct computation gives
\begin{align*}
a^{-\frac14} x b^{-\frac14}
\ov{\eqref{r=ax}}{=} a^{-\frac54}rb^{-\frac14}
=a^{-\frac54} y_1y_2^*y_3  b^{-\frac14}
\ov{\eqref{def-des-yi}}{=} a^{-1}v_1b^{\frac12}v_2^* a^{\frac12}v_3
\ov{\eqref{def-Theta-z}}{=} \Theta^1(v_1)\Theta^{-\frac12}(v_2)^*v_3.
\end{align*}
Thus $
\Theta^1(v_1)\Theta^{-\frac12}(v_2)^*v_3
\in \cal{V}$. Since $\Theta^1$ and $\Theta^{-\frac12}$ are bijections of $\cal{V}$, their values range independently over $\cal V$. Hence
\[
u_1u_2^*u_3 \in \cal{V},
\quad
u_1,u_2,u_3 \in \cal{V}.
\]
Consequently, we have the inclusion $\cal{V} \cal{V}^*\cal{V} \subset \cal{V}$. Since $\cal V$ is finite-dimensional, it is weak* closed, and therefore it is a $\W^*$-$\TRO$. Moreover, if we consider the map $J \co \cal{M} \to\M_2(\cal{M})$, $v \mapsto
\begin{bmatrix}
0&v\\
0&0
\end{bmatrix}$, then
\begin{equation}
\label{symmetric-L4-linking-realization}
J(a^{\frac14}vb^{\frac14})
=D^{\frac18} J(v) D^{\frac18},
\quad
v \in \cal V.
\end{equation}
Indeed, we have
\begin{align*}
\MoveEqLeft
D^{\frac18}
J(v)
D^{\frac18}
=D^{\frac18}
\begin{bmatrix}
0&v\\
0&0
\end{bmatrix}
D^{\frac18}
\ov{\eqref{linking-density}}{=}
\begin{bmatrix}
a^{\frac14}&0\\
0&b^{\frac14}
\end{bmatrix}
\begin{bmatrix}
0&v\\
0&0
\end{bmatrix}
\begin{bmatrix}
a^{\frac14}&0\\
0&b^{\frac14}
\end{bmatrix} \\
&=\begin{bmatrix}
0 & a^{\frac14}v \\
0 & 0
\end{bmatrix}\begin{bmatrix}
a^{\frac14}&0\\
0&b^{\frac14}
\end{bmatrix}
=
\begin{bmatrix}
0&a^{\frac14}vb^{\frac14}\\
0&0
\end{bmatrix}
=J(a^{\frac14}vb^{\frac14}),
\end{align*}
which proves \eqref{symmetric-L4-linking-realization}. We next check that the linking algebra of $\cal V$ has the same unit as $\cal R$. Let $(x_1,\ldots,x_k)$ be the basis used in \eqref{global-Gram-densities} and put
\[
v_j
\ov{\mathrm{def}}{=}
a^{-\frac14}x_jb^{-\frac14}
\in\cal V,
\quad
j\in\{1,\ldots,k\}.
\]
Since $b$ is positive and invertible in the reduced finite-dimensional algebra $f\cal Mf$, the functional calculus gives
\begin{equation}
\label{lower-bound-b-inverse}
b^{-\frac12}
\geq
\norm{b}_{\L^\infty(\cal M)}^{-\frac12}f.
\end{equation}
Consequently, we deduce that
\begin{align}
\label{inter-765}
\sum_{j=1}^k v_jv_j^*
= a^{-\frac14}
\Bigg(\sum_{j=1}^k x_jb^{-\frac12}x_j^*\Bigg)
a^{-\frac14}
\ov{\eqref{lower-bound-b-inverse}}{\geq}
\norm{b}_{\L^\infty(\cal M)}^{-\frac12}
a^{-\frac14}
\Bigg(\sum_{j=1}^k x_jx_j^*\Bigg)
a^{-\frac14}
\ov{\eqref{global-Gram-densities}}{=}
\norm{b}_{\L^\infty(\cal M)}^{-\frac12}a^{\frac12}.
\end{align}
The operator on the left-hand side of \eqref{inter-765} belongs to $e\cal Me$, so its support is dominated by $e$. On the other hand, \eqref{inter-765} and $s(a^{\frac{1}{2}})=e$ imply that its support dominates $e$. Hence
\[
s\bigg(\sum_{j=1}^k v_jv_j^*\bigg)
= e,
\]
and therefore $s_\ell(\cal V)=e$. Similarly, we have $
a^{-\frac12}
\geq
\norm{a}_{\L^\infty(\cal M)}^{-\frac12}e$, and hence
\begin{align*}
\sum_{j=1}^k v_j^*v_j
=
b^{-\frac14}
\Bigg(\sum_{j=1}^k x_j^*a^{-\frac12}x_j\Bigg)
b^{-\frac14}
\geq
\norm{a}_{\L^\infty(\cal M)}^{-\frac12}
b^{-\frac14}
\Bigg(\sum_{j=1}^k x_j^*x_j\Bigg)
b^{-\frac14}
=
\norm{a}_{\L^\infty(\cal M)}^{-\frac12}b^{\frac12}.
\end{align*}
It follows that
\[
s\left(\sum_{j=1}^k v_j^*v_j\right)
=
f,
\]
and therefore $s_r(\cal V) = f$. Consequently, we have $1_{\cal R(\cal V)}
=
\begin{bmatrix}
e&0\\
0&f
\end{bmatrix}
=
1_{\cal R}$. Thus $\cal R(\cal V)$ is a unital von Neumann subalgebra of $\cal R$. Now, we use Proposition~\ref{prop-linking-modular-interpretation}. If $v \in \cal V$ and $t \in \R$, then
\begin{align*}
\sigma_t^{D}(J(v))
&\ov{\eqref{linking-modular-group-corner}}{=}
J(a^{2\i t}vb^{-2\i t})
 \ov{\eqref{def-Theta-z}}{=} 
J(\Theta^{-2\i t}(v)).
\end{align*}
Since $\Theta^{-2\i t}(\cal V)=\cal V$, we obtain
\[
\sigma_t^{D}(J(\cal V))=J(\cal V),
\quad t \in \R.
\]
The map $\sigma_t^{D} \co \cal R \to \cal R$ is a $*$-automorphism. It therefore also leaves $J(\cal V)^*$ invariant, as well as the von Neumann algebras generated by $J(\cal V)J(\cal V)^*$ and $J(\cal V)^*J(\cal V)$. Since $
\cal R(\cal V)
 \ov{\eqref{Linking-algebra}}{=} 
\begin{bmatrix}
\ovl{\Span\cal V\cal V^*}^{\w^*}&\cal V\\
\cal V^*&\ovl{\Span\cal V^*\cal V}^{\w^*}
\end{bmatrix}$, we conclude that $
\sigma_t^{D}(\cal R(\cal V))
=
\cal R(\cal V)$ for any $t \in \R$. This proves \eqref{linking-algebra-V-modular-invariance}. Let $
q
\ov{\mathrm{def}}{=}
\begin{bmatrix}
e&0\\
0&0
\end{bmatrix}
\in\cal R(\cal V)$. Then
\[
1_{\cal R(\cal V)}-q
=
\begin{bmatrix}
0&0\\
0&f
\end{bmatrix}
\quad \text{and} \quad
\cal V
=
q\cal R(\cal V)(1_{\cal R(\cal V)}-q).
\]
Since the density $D$ is block diagonal, we have $
qD=Dq$. Hence $q$ is fixed by $(\sigma_t^D)_{t\in\R}$. Since the modular group of $\psi$ is $(\sigma_t^D)_{t\in\R}$ and $\cal R(\cal V)$ is globally invariant under this group, Takesaki's theorem gives
\[
\sigma_t^\varphi
=
\sigma_t^\psi|_{\cal R(\cal V)}
=
\sigma_t^D|_{\cal R(\cal V)},
\quad
t \in \R.
\]
Consequently, we have
\[
\sigma_t^\varphi(q)=q,
\quad
t \in \R,
\]
so that $q$ belongs to the centralizer of $\varphi$. In particular, the symmetric Kosaki realization of $\L^p(\cal R(\cal V),\varphi)$ preserves the upper-right corner defining $\L^p(\cal V,\varphi)$. Applying Lemma~\ref{lem-modular-invariant-Kosaki-subspace}, the inclusion $\cal R(\cal V)\subset\cal R$ induces a complete isometry 
$$
J_4\co
\L^4(\cal R(\cal V),\varphi)
\to
\L^4(\cal R,\psi).
$$ 
Since $q$ belongs to the centralizer of $\varphi$ and commutes with the density $d \ov{\mathrm{def}}{=} c^{-1}D$ of $\psi$, the symmetric realizations on both sides preserve the corners determined by $q$. Hence $J_4$ restricts to a complete isometry
\[
J_4 \co \L^4(\cal{V},\varphi) \to q\L^4(\cal{R},\psi)(1-q).
\]
Recall that $d = c^{-1}D$, the density of $\psi$ with respect to $\tr_{\cal R}$. By the definition of the canonical inclusion $J_4$ in Lemma~\ref{lem-modular-invariant-Kosaki-subspace}, the image of $v\in\cal V$ is represented in the ambient symmetric realization by $d^{\frac18} J(v) d^{\frac18}$. Hence, for any $v \in \cal V$, we have
\begin{align*}
\MoveEqLeft
\norm{v}_{\L^4(\cal V,\varphi)}
=\bnorm{d^{\frac18}J(v)d^{\frac18}}_{\L^4(\cal R,\tr_{\cal{R}})}
=c^{-\frac14}\bnorm{D^{\frac18}J(v)D^{\frac18}}_{\L^4(\cal{R},\tr_{\cal R})}\\
&\ov{\eqref{symmetric-L4-linking-realization}}{=}
c^{-\frac14}\bnorm{J(a^{\frac14}vb^{\frac14})}_{\L^4(\cal{R},\tr_{\cal R})}
=c^{-\frac14}\bnorm{a^{\frac14}vb^{\frac14}}_{\L^4(\cal M)}
\ov{\eqref{complete-isometry-deweighted-space}}{=}
\norm{U(v)}_{\L^4(\cal{M})}.
\end{align*}
The same calculation at the $n$th matrix level uses the density $\Id_{\M_n}\ot d$ and the map $\Id_{\M_n}\ot J$. Hence $U$ is a complete isometry. Moreover, we have
\[
U(\cal V)
\ov{\eqref{complete-isometry-deweighted-space}}{=} 
c^{-\frac14}a^{\frac14}\cal Vb^{\frac14}
\ov{\eqref{finite-dimensional-weighted-range}}{=}
c^{-\frac14}\Ran P
=\Ran P,
\]
since $\Ran P$ is a linear subspace. Hence $U$ is surjective. This completes the proof.
\end{proof}

\begin{prop}
\label{prop-explicit-weighted-TRO-expectation}
Retain the assumptions and notation of Corollary~\ref{cor-finite-dimensional-modular-deweighting}. Put $\cal N\ov{\mathrm{def}}{=}\cal R(\cal V)$, and let $\E\co\cal R\to\cal N$ be the $\psi$-preserving conditional expectation provided by Takesaki's theorem. Define $\E_{12}\co e\cal M f\to\cal V$ by
\begin{equation*}
\E\left(\begin{bmatrix}
0&z\\0&0\end{bmatrix}\right)
=\begin{bmatrix}0&\E_{12}(z)\\0&0
\end{bmatrix}.
\end{equation*}
Then the original projection satisfies
\begin{equation}
\label{explicit-original-projection-formula}
P(x)
=a^{\frac{1}{4}}\E_{12}\big(a^{-\frac{1}{4}}e x f b^{-\frac{1}{4}}\big)b^{\frac{1}{4}},
\quad x \in \L^4(\mathcal M).
\end{equation}
In particular, $P$ is contractively decomposable.
\end{prop}

\begin{proof}
The diagonal corner projections belong to $\cal{N}$. By $\cal{N}$-bimodularity, $\E$ preserves the four corners, so $\E_{12}$ is well-defined. The symmetric $\mathrm L^4$ realization of the state-preserving conditional expectation is
\begin{equation*}
\cal{T}(X)
=D^{\frac{1}{8}}\E\big(D^{-\frac{1}{8}}XD^{-\frac{1}{8}}\big)D^{\frac{1}{8}},
\quad X \in \L^4(\mathcal R,\Tr_{\mathcal R}).
\end{equation*}
Normalization of the density by the scalar $c$ cancels from this formula. The map $\mathcal T$ is completely positive and completely contractive: complete positivity follows directly from the displayed formula, and complete contractivity follows from the state-preserving $\L^4$ extension of $\E$. It is a projection because $\E^2=\E$.

Let $s=\operatorname{diag}(e,f) \in \M_2(\cal{M})$. Identifying $\cal{R}$ with $s\mathrm M_2(\cal{M})s$, extend $\mathcal T$ to $\L^4(\M_2\overline\otimes\mathcal M)$ by $\widetilde{\mathcal T}(X)=\mathcal T(sXs)$. This remains completely positive and completely contractive. Since $D$ is block diagonal and $\E$ preserves corners, the extension has block form
\begin{equation*}
\widetilde{\mathcal T}\left(\begin{bmatrix}
x_{11}&x_{12}\\
x_{21}&x_{22}
\end{bmatrix}\right)
=\begin{bmatrix}
v_1(x_{11})&Q(x_{12})\\
Q^\circ(x_{21})&v_2(x_{22})
\end{bmatrix},
\end{equation*}
where $\|v_1\|,\|v_2\|\leq1$ and $Q$ is the right-hand side of \eqref{explicit-original-projection-formula}. Thus $\|Q\|_{\dec} \leq 1$. Its range is $a^{\frac{1}{4}}\mathcal Vb^{\frac{1}{4}}=\Ran P$, and it fixes this range because $\E_{12}|_{\mathcal V}=\Id_{\mathcal V}$. Therefore $Q$ is a contractive projection onto $\Ran P$. Uniqueness of contractive projections onto a fixed nonzero subspace of the smooth space $\L^4(\mathcal M)$ gives $Q=P$.
\end{proof}

\begin{remark} \normalfont
\label{rem-Levi-modular-summary}
By Example \ref{E-finite-dim-TRO}, the $\W^*$-TRO $\cal{V}$ is completely isometric to a finite $\ell^\infty$-direct sum of rectangular matrix spaces.
\end{remark}


The preceding explicit representation immediately yields the following consequence.

\begin{cor}
\label{cor-contractively-decomposable}
Assume that the von Neumann algebra $\cal{M}$ is finite-dimensional. A $2$-contractive projection $P \co \L^4(\cal M) \to \L^4(\cal{M})$ is contractively decomposable.
\end{cor}

%

\begin{proof}
If $\Ran P=\{0\}$, then $P=0$ and the assertion is immediate. Otherwise, the result follows from Proposition~\ref{prop-explicit-weighted-TRO-expectation}.
\end{proof}

\section{Finite matrix-level criteria at even exponents}
\label{sec-even-exponents}

The polynomial argument used for the exponent $4$ extends to every even integral exponent. The matrix level obtained below is sufficient. We do not claim that it is optimal.

\begin{thm}
\label{thm-even-exponent-finite-level}
Let $m$ be an integer with $m \geq 2$. Set $p \ov{\mathrm{def}}{=}2m$ and $N_m \ov{\mathrm{def}}{=}
\left\lceil\frac{2m}{3}\right\rceil$. Let $P \co \L^p(\cal M) \to \L^p(\cal M)$ be a bounded projection. We denote by $P^\dagger$ its adjoint under the sesquilinear trace duality and set $
Y\ov{\mathrm{def}}{=}\Ran P$ and $Y^\sharp\ov{\mathrm{def}}{=}\Ran P^\dagger$. The following assertions are equivalent.
\begin{enumerate}
\item The projection $P$ is $N_m$-contractive.
\item For any $y_1,\ldots,y_{2m-1}\in Y$, we have
\begin{equation}
\label{alternating-word-dual-range-even}
y_1y_2^*y_3\cdots y_{2m-2}^*y_{2m-1}
\in Y^\sharp.
\end{equation}
\item The projection $P$ is completely contractive.
\end{enumerate}
If these conditions hold, then
\begin{equation}
\label{exact-alternating-word-dual-range-even}
\begin{split}
Y^\sharp
&=
\left\{
y_1y_2^*y_3\cdots y_{2m-2}^*y_{2m-1}
:
y_1,\ldots,y_{2m-1}\in Y
\right\}
=
\left\{
y(y^*y)^{m-1}:y\in Y
\right\}.
\end{split}
\end{equation}
Moreover, for any integer $n \geq 1$, we have
\begin{equation}
\label{matrix-duality-map-equality-even}
\Ran P_n^\dagger
=
\left\{
X(X^*X)^{m-1}:X\in\M_n(Y)
\right\}
=
\M_n(Y^\sharp),
\end{equation}
where $P_n\ov{\mathrm{def}}{=}\Id_{S_n^p}\ot P$.
\end{thm}

\begin{proof}
The implication $3\Rightarrow1$ is immediate. We first prove $1\Rightarrow2$. Put $
r\ov{\mathrm{def}}{=}2m-1$. Suppose that $P_{N_m}$ is contractive. At matrix level $N_m$, the power duality mapping is
\[
\frak j_p(X)
=
X(X^*X)^{m-1},
\quad
X \in \L^p(\M_{N_m} \otvn \cal M).
\]
Moreover, we have $
P_{N_m}^\dagger
=
\Id_{S_{N_m}^{p^*}}\ot P^\dagger$ and $\Ran P_{N_m}^\dagger
=
\M_{N_m}(Y^\sharp)$. Consequently, Lemma~\ref{Lemma-smooth} gives
\begin{equation}
\label{matrix-duality-inclusion-even}
X(X^*X)^{m-1}
\in
\M_{N_m}(Y^\sharp),
\quad
X \in \M_{N_m}(Y).
\end{equation}
Fix $y_1,\ldots,y_r \in Y$. For $0 \leq k \leq r$, set $
v_k
\ov{\mathrm{def}}{=}
1+\left\lfloor\frac{k}{3}\right\rfloor$. Then $v_0=1$ and $v_r=N_m$. Let $(E_{ij})_{1\leq i,j\leq N_m}$ be the canonical matrix units of $\M_{N_m}$ and put
\[
F_k
\ov{\mathrm{def}}{=}
E_{v_{k-1},v_k},
\quad
A_k
\ov{\mathrm{def}}{=}
\begin{cases}
F_k\ot y_k,&k\text{ odd},\\
F_k^*\ot y_k,&k\text{ even},
\end{cases}
\quad 1\leq k\leq r.
\]
Thus $A_k \in \M_{N_m}(Y)$ for any integer $k$. For any $\zeta=(\zeta_1,\ldots,\zeta_r) \in \T^r$, define
\[
X(\zeta)
\ov{\mathrm{def}}{=}
\sum_{k=1}^r\zeta_kA_k.
\]
By \eqref{matrix-duality-inclusion-even}, we have $
\frak j_p(X(\zeta))
\in
\M_{N_m}(Y^\sharp)$ for any $\zeta \in \T^r$. Let $\mu$ be the normalized Haar measure on $\T^r$. Consider the Fourier coefficient
\begin{equation}
\label{Fourier-coefficient-even}
C
\ov{\mathrm{def}}{=}
\int_{\T^r}
\left(
\prod_{\substack{1\leq k\leq r\\ k\text{ odd}}}
\overline{\zeta_k}
\right)
\left(
\prod_{\substack{1\leq k\leq r\\ k\text{ even}}}
\zeta_k
\right)
\frak j_p(X(\zeta))
\d\mu(\zeta).
\end{equation}
Repeated applications of H\"older's inequality show that the map $\zeta \mapsto \frak j_p(X(\zeta))$ is continuous from $\T^r$ into the space $\L^{p^*}(\M_{N_m}\otvn\cal M)$. Hence the integral in \eqref{Fourier-coefficient-even} is a Bochner integral. Since $\M_{N_m}(Y^\sharp)$ is a closed linear subspace, we obtain
\begin{equation}
\label{Fourier-coefficient-in-dual-range}
C\in\M_{N_m}(Y^\sharp).
\end{equation}

We determine the $(1,N_m)$-entry of $C$. Expanding the duality mapping gives
\[
\frak j_p(X(\zeta))
=
\sum_{k_1,\ldots,k_r=1}^r
\zeta_{k_1}\overline{\zeta_{k_2}}
\zeta_{k_3}\cdots
\overline{\zeta_{k_{r-1}}}\zeta_{k_r}
A_{k_1}A_{k_2}^*A_{k_3}\cdots A_{k_{r-1}}^*A_{k_r}.
\]
For a fixed term, let $a_k$ be the number of occurrences of the label $k$ in an unstarred position and let $b_k$ be its number of occurrences in a starred position. Orthogonality of the characters of $\T^r$ shows that this term survives in \eqref{Fourier-coefficient-even} only if
\begin{equation}
\label{Fourier-exponent-conditions-even}
a_k-b_k
=
\begin{cases}
1,&k\text{ odd},\\
-1,&k\text{ even}.
\end{cases}
\end{equation}
There are $m$ unstarred positions and $m-1$ starred positions. Since there are $m$ odd labels, \eqref{Fourier-exponent-conditions-even} gives
\[
m
=
\sum_{k=1}^r a_k
\geq
\sum_{\substack{1\leq k\leq r\\k\text{ odd}}}a_k
\geq m.
\]
Consequently, equality holds throughout. We obtain
\[
a_k=1,\quad b_k=0
\quad(k\text{ odd}),
\quad
a_k=0,\quad b_k=1
\quad(k\text{ even}).
\]
Thus the surviving terms are indexed by permutations $\pi$ of $\{1,\ldots,r\}$ satisfying
\[
\pi(j)\equiv j\pmod 2,
\quad 1\leq j\leq r.
\]
For such a permutation, the corresponding product is
\[
F_{\pi(1)}F_{\pi(2)} \cdots F_{\pi(r)}
\ot
y_{\pi(1)}y_{\pi(2)}^* \cdots y_{\pi(r-1)}^*y_{\pi(r)}.
\]
We claim that the identity permutation is the only such permutation for which the matrix product has a nonzero $(1,N_m)$-entry. Indeed, the matrix unit $F_k$ is the directed edge
\[
v_{k-1} \to v_k.
\]
For any vertex $1 \leq s<N_m$, the effective edge list contains, in this order, two loops at $s$ and the forward edge from $s$ to $s+1$. The remaining zero, one or two edges are loops at $N_m$. A composable path from $1$ to $N_m$ which uses every edge must use the forward edges in their natural order, and every loop at a vertex must occur after arrival at that vertex and before departure from it. Hence the order is forced, except possibly for interchanging the two loops at a fixed vertex. The labels of these two loops are consecutive and therefore have opposite parity. Such an interchange is excluded by $\pi(j) \equiv j\pmod2$. The same argument applies to the two possible terminal loops. This proves the claim.

Consequently, we have
\[
C_{1,N_m}
=
y_1y_2^*y_3\cdots y_{r-1}^*y_r.
\]
It follows from \eqref{Fourier-coefficient-in-dual-range} that this alternating word belongs to $Y^\sharp$. This proves \eqref{alternating-word-dual-range-even}.

We next prove $2\Rightarrow3$. Let $n\geq1$ and let $X=[x_{ij}]\in\M_n(Y)$. For $1\leq i_0,i_{2m-1}\leq n$, we have
\begin{align}
\label{entries-even-duality-map}
\MoveEqLeft
\bigl[X(X^*X)^{m-1}\bigr]_{i_0,i_{2m-1}}\\
&=
\sum_{i_1,\ldots,i_{2m-2}=1}^n
\left[
\prod_{\ell=0}^{m-2}
\left(
x_{i_{2\ell},i_{2\ell+1}}
x_{i_{2\ell+2},i_{2\ell+1}}^*
\right)
\right]
x_{i_{2m-2},i_{2m-1}}. \nonumber
\end{align}
Every summand on the right-hand side is an alternating product of $2m-1$ elements of $Y$. Hence the second assertion gives
\[
\frak j_p(X)
=
X(X^*X)^{m-1}
\in
\M_n(Y^\sharp)
=
\Ran P_n^\dagger.
\]

Let $W\in\L^p(\M_n\otvn\cal M)$ and put $X\ov{\mathrm{def}}{=}P_n(W)$. Since $
W-X\in\ker P_n$ and $
\frak j_p(X)\in\Ran P_n^\dagger=(\ker P_n)^\perp$, 
we have
\[
(\Tr_n\ot\tr)\bigl(\frak j_p(X)^*(W-X)\bigr)=0.
\]
Consequently, we have
\begin{align*}
\MoveEqLeft
\norm{X}_p^p
=(\Tr_n\ot\tr)\bigl(\frak j_p(X)^*X\bigr)
=(\Tr_n\ot\tr)\bigl(\frak j_p(X)^*W\bigr)
\leq
\norm{\frak j_p(X)}_{p^*}\norm{W}_p
=
\norm{X}_p^{p-1}\norm{W}_p.
\end{align*}
If $X\neq0$, division by $\norm{X}_p^{p-1}$ gives
\[
\norm{P_n(W)}_p\leq\norm{W}_p.
\]
The same conclusion is immediate when $X=0$. Hence $P_n$ is contractive. Since $n\geq1$ was arbitrary, $P$ is completely contractive.

It remains to prove the exact range descriptions. Applying \eqref{power-duality-map-range-equality}, under the trace-duality identification, to each contractive amplification $P_n$, we obtain
\[
\Ran P_n^\dagger
=
\frak j_p(\Ran P_n)
=
\left\{
X(X^*X)^{m-1}:X\in\M_n(Y)
\right\}.
\]
Together with $\Ran P_n^\dagger=\M_n(Y^\sharp)$, this proves \eqref{matrix-duality-map-equality-even}. Taking $n=1$ gives
\[
Y^\sharp
=
\left\{
y(y^*y)^{m-1} : y \in \Ran P
\right\}.
\]
Each element on the right-hand side is an alternating product of $2m-1$ elements of $\Ran P$, whereas the second assertion gives the reverse inclusion for arbitrary alternating products. This proves \eqref{exact-alternating-word-dual-range-even}.
\end{proof}

The preceding theorem has an immediate dual counterpart.

\begin{cor}
\label{cor-conjugate-even-exponent-finite-level}
Let $m$ be an integer with $m \geq 2$. Set $p \ov{\mathrm{def}}{=}2m$, $q\ov{\mathrm{def}}{=}p^*
=\frac{2m}{2m-1}$, $N_m\ov{\mathrm{def}}{=} \left\lceil\frac{2m}{3}\right\rceil$. Let $Q\co\L^q(\cal M) \to \L^q(\cal M)$ be a bounded projection and put $Z\ov{\mathrm{def}}{=}\Ran Q$. The following assertions are equivalent.
\begin{enumerate}
\item The projection $Q$ is $N_m$-contractive.
\item The projection $Q$ is completely contractive.
\item The projection $Q$ is contractive and, for any $z_1,\ldots,z_{2m-1} \in Z$, we have
\begin{equation}
\label{dual-alternating-word-conjugate-even}
\frak j_q(z_1)
\frak j_q(z_2)^*
\frak j_q(z_3)\cdots
\frak j_q(z_{2m-2})^*
\frak j_q(z_{2m-1})
\in Z.
\end{equation}
\end{enumerate}
If these conditions hold, then
\begin{equation}
\label{exact-dual-alternating-word-conjugate-even}
Z
=\big\{
\frak j_q(z_1)
\frak j_q(z_2)^*\cdots
\frak j_q(z_{2m-2})^*
\frak j_q(z_{2m-1})
:z_1,\ldots,z_{2m-1}\in Z
\big\}.
\end{equation}
Moreover, for any integer $n \geq 1$, we have
\begin{equation}
\label{matrix-duality-map-equality-conjugate-even}
\Ran Q_n^\dagger
=
\frak j_{q,\L^q(\M_n\otvn\cal M)}\bigl(\M_n(Z)\bigr)
=
\M_n(\Ran Q^\dagger),
\end{equation}
where $Q_n\ov{\mathrm{def}}{=}\Id_{S_n^q}\ot Q$.
\end{cor}

\begin{proof}
Set $
P\ov{\mathrm{def}}{=}Q^\dagger
\co
\L^p(\cal M)\to\L^p(\cal M)$. By reflexivity, $P^\dagger=Q$. Moreover, under the matrix trace duality,
\[
(Q_n)^\dagger=P_n,
\quad n \geq 1.
\]
Consequently, we have
\begin{equation}
\label{adjoint-amplification-norm-equality-even}
\norm{Q_n}
=
\norm{(Q_n)^\dagger}
=
\norm{P_n},
\quad n \geq 1.
\end{equation}
It follows from Theorem~\ref{thm-even-exponent-finite-level} that $Q$ is $N_m$-contractive if and only if $P$ is $N_m$-contractive, if and only if $P$ is completely contractive, if and only if $Q$ is completely contractive. This proves the equivalence of the first two assertions.

Suppose that $Q$ is $N_m$-contractive. Then $Q$ is contractive. Equation \eqref{power-duality-map-range-equality}, under the trace-duality identification, gives
\[
\frak j_q(Z)
=
\Ran Q^\dagger
=
\Ran P,
\quad
\Ran P^\dagger
=
\Ran Q
=Z.
\]
Applying \eqref{alternating-word-dual-range-even} to $P$ gives \eqref{dual-alternating-word-conjugate-even}.

Conversely, suppose that $Q$ is contractive and that \eqref{dual-alternating-word-conjugate-even} holds. Again, we have
\[
\Ran P
=
\Ran Q^\dagger
=
\frak j_q(Z),
\quad
\Ran P^\dagger
=Z.
\]
Thus \eqref{dual-alternating-word-conjugate-even} is precisely the second assertion of Theorem~\ref{thm-even-exponent-finite-level} for $P$. Hence $P$ is completely contractive. Equation \eqref{adjoint-amplification-norm-equality-even} then shows that $Q$ is completely contractive.

Finally, \eqref{exact-dual-alternating-word-conjugate-even} follows from \eqref{exact-alternating-word-dual-range-even} applied to $P$. Since any amplification $Q_n$ is contractive, applying \eqref{power-duality-map-range-equality} to $Q_n$ gives
\[
\frak j_{q,\L^q(\M_n\otvn\cal M)}\bigl(\M_n(Z)\bigr)
=
\Ran Q_n^\dagger
=
\M_n(\Ran Q^\dagger),
\]
which proves \eqref{matrix-duality-map-equality-conjugate-even}.
\end{proof}

The preceding finite-level equivalences are not restricted to the semifinite tracial setting. The same argument extends to Haagerup noncommutative $\L^p$-spaces associated with arbitrary von Neumann algebras.

\begin{thm}
\label{thm-Haagerup-even-finite-level}
Let $\cal{M}$ be an arbitrary von Neumann algebra, and let $\L^p(\mathcal M)$ denote its Haagerup noncommutative $\L^p$-space. Let $m \in \mathbb N$ with $m\geq2$, and put $p=2m$ and $N_m=\lceil \frac{2m}{3}\rceil$. Let $P \co \L^p(\cal{M}) \to \L^p(\mathcal M)$ be a bounded projection, and write $Y=\Ran P$ and $Y^\sharp = \Ran P^\dagger$. The following conditions are equivalent: 

\begin{enumerate}
\item $P$ is $N_m$-contractive,
	
\item any alternating product $y_1y_2^*y_3\cdots y_{2m-2}^*y_{2m-1}$ with all $y_j \in Y$, belongs to $Y^\sharp$,

\item $P$ is completely contractive.
\end{enumerate}
The corresponding equivalence between $N_m$-contractivity and complete contractivity also holds at the conjugate exponent $p^*=\frac{2m}{2m-1}$.
\end{thm}

\begin{proof}
In this proof, $\Tr$ denotes the canonical trace functional on Haagerup $\L^1$-spaces, including their matrix amplifications, rather than the semifinite trace of the crossed product. We use the standard properties of Haagerup noncommutative $\L^p$-spaces. For $1<p<\infty$, these spaces are smooth and strictly convex, the H\"older product $
\L^{p_1}(\cal M)\times\cdots\times\L^{p_k}(\cal M)
\longrightarrow
\L^r(\cal M)$ where $
\frac1r=\sum_{j=1}^k\frac1{p_j}$, is continuous whenever the exponents are admissible, and the canonical trace satisfies the usual cyclicity relation for H\"older products. In particular, an alternating product of $2m-1$ elements of $\L^{2m}(\cal M)$ belongs to $\L^{\frac{2m}{2m-1}}(\cal M)$. If $x=u|x|$, the power duality mapping for the sesquilinear trace pairing is represented by
\[
\frak j_p(x)=u|x|^{p-1}.
\]
For $p=2m$, this becomes $
\frak j_{2m}(x)
=
x(x^*x)^{m-1}$. Finally, we have $
\L^p(\M_n\otvn\cal M)
=
S_n^p[\L^p(\cal M)]$ completely isometrically, and this identification is compatible with the H\"older products above and with trace-adjoint duality.

If $P_{N_m}$ is contractive, Lemma~\ref{Lemma-smooth} yields $\mathfrak j_p(X) \in \mathrm M_{N_m}(Y^\sharp)$ for any $X \in \M_{N_m}(Y)$. The torus Fourier coefficient and matrix-unit construction in the proof of Theorem~\ref{thm-even-exponent-finite-level} use only this inclusion, continuity of the polynomial product, and multiplication of matrix units. They therefore apply without change and show that any alternating word of length $2m-1$ in $Y$ belongs to $Y^\sharp$.

Conversely, suppose that all such words belong to $Y^\sharp$. Expanding each entry of $X(X^*X)^{m-1}$ gives $\mathfrak j_p(X) \in \M_n(Y^\sharp)=\Ran P_n^\dagger$ for any $X \in \M_n(Y)$. If $W \in \L^p(\mathrm M_n\overline\ot \mathcal M)$ and $X=P_n(W)$, trace-adjoint duality and H\"older's inequality give
\begin{equation*}
\|X\|_p^p
=\Tr(\mathfrak j_p(X)^*W)
\leq\|X\|_p^{p-1}\|W\|_p.
\end{equation*}
Thus every $P_n$ is contractive. The reverse implication from complete contractivity is immediate. Finally, $(P_n)^\dagger=(P^\dagger)_n$ and $\|P_n\|=\|(P^\dagger)_n\|$ give the assertion at the conjugate exponent.
\end{proof}

\begin{remark} \normalfont
For Schatten spaces, \cite[Theorem~1.1]{LRR09}, combined with uniqueness of norm-one projections in smooth Banach spaces, implies that $2$-contractivity of a projection is equivalent to complete contractivity for any $1<p<\infty$, $p\neq2$. Theorem~\ref{thm-Haagerup-even-finite-level} shows that, on the Haagerup $\L^{2m}$-space of an arbitrary von Neumann algebra, the finite level $
N_m=\left\lceil\frac{2m}{3}\right\rceil$ is sufficient for complete contractivity. We do not claim that this level is optimal. In particular, $N_m=2$ for $p=4$ and $p=6$.
\end{remark}

\section{Weighted tangent geometry for $p>2$}
\label{sec-weighted-tangent-general-p}

Throughout this section, let $\cal M$ be a semifinite von Neumann algebra equipped with a normal semifinite faithful trace $\Tr$. Fix $2<p<\infty$, and put $p^*=\frac{p}{p-1}$, $r=\frac{p}{p-2}$ and $\alpha = \frac{p-2}{2}$. Thus $\frac{1}{p^*}=\frac{1}{p}+\frac{1}{r}$ and $\frac{1}{2}=\frac{1}{p}+\frac{1}{2r}$. Let $P \co \L^p(\cal M) \to \L^p(\cal M)$ be a $2$-contractive projection. Its trace adjoint $P^\dagger\co\L^{p^*}(\cal M) \to \L^{p^*}(\cal M)$ is determined by
\begin{equation}
\label{gp-trace-adjoint}
\Tr\big((P^\dagger\eta)^*x\big)
=\Tr\big(\eta^*P(x)\big),
\quad x \in \L^p(\cal M), \eta \in \L^{p^*}(\cal M).
\end{equation}
We write $Y=\Ran P$ and $Y^\sharp=\Ran P^\dagger$. No positivity assumption on $P$ is imposed.

\subsection{Separate weighted tangent estimates}

\begin{prop}
\label{prop-gp-weighted-tangent-estimates}
For $u \in \Ran P$, put $a_u=(uu^*)^\alpha$ and $b_u=(u^*u)^\alpha$. Then $a_u,b_u \in \L^r(\cal M)_+$, and for any $z \in \L^p(\cal M)$, we have
\begin{equation}
\label{gp-separate-weighted-estimates}
\begin{split}
\big\|(uu^*)^{\frac{p-2}{4}}P(z)\big\|_2
&\leq\big\|(uu^*)^{\frac{p-2}{4}}z\big\|_2,\\
\big\|P(z)(u^*u)^{\frac{p-2}{4}}\big\|_2
&\leq\big\|z(u^*u)^{\frac{p-2}{4}}\big\|_2.
\end{split}
\end{equation}
Moreover, for any $x \in \L^p(\cal M)$, the following intertwining identities hold :
\begin{equation}
\label{gp-weighted-intertwining}
P^\dagger(a_ux)=a_uP(x),
\quad
P^\dagger(xb_u)=P(x)b_u.
\end{equation}
In particular, we have $a_uY\subset Y^\sharp$ and $Yb_u\subset Y^\sharp$.
\end{prop}

\begin{proof}
The identities $2\alpha r=p$ and $\|u^*\|_p=\|u\|_p$ give
\begin{equation*}
\|a_u\|_r=\|b_u\|_r=\|u\|_p^{p-2}.
\end{equation*}
H\"older's inequality therefore shows that $a_u^{1/2}z,zb_u^{1/2}\in\L^2(\cal M)$ and $a_uz,zb_u\in\L^{p^*}(\cal M)$ whenever $z\in\L^p(\cal M)$.

Put $q=p/2>1$. For $H,K\in\L^q(\cal M)_+$, the norm-derivative formula \eqref{derivation-duality-mapping-2}, together with the represented duality mapping $\frak j_q(H)=H^{q-1}$, yields
\begin{equation}
\label{gp-trace-power-derivative}
\lim_{s\downarrow0}\frac{\Tr((H+sK)^q)-\Tr(H^q)}{s}
=q\Tr(H^{q-1}K).
\end{equation}
When $H=0$, this follows directly from $q>1$. Thus \eqref{gp-trace-power-derivative} does not require $H$ to be invertible. Only first-order differentiability on $\L^q(\cal M)$ is used.

Fix $u\in Y$ and $z\in\L^p(\cal M)$. For $s\geq0$, consider
\begin{equation*}
R_s=\begin{bmatrix}u&s^{1/2}z\\0&0\end{bmatrix}
\in\L^p(\M_2\overline\otimes\cal M).
\end{equation*}
Since $P(u)=u$ and $P_2$ is contractive, the identity
\begin{equation*}
R_sR_s^*=\begin{bmatrix}uu^*+szz^*&0\\0&0\end{bmatrix}
\end{equation*}
gives
\begin{equation*}
\Tr\big((uu^*+sP(z)P(z)^*)^q\big)
\leq\Tr\big((uu^*+szz^*)^q\big).
\end{equation*}
Both sides agree at $s=0$. Subtracting their common value, dividing by $s>0$, and applying \eqref{gp-trace-power-derivative}, we obtain
\begin{equation*}
\Tr\big(a_uP(z)P(z)^*\big)\leq\Tr(a_uzz^*).
\end{equation*}
By H\"older's inequality and trace cyclicity, these terms are $\|a_u^{1/2}P(z)\|_2^2$ and $\|a_u^{1/2}z\|_2^2$, respectively. This proves the first estimate.

Using instead $C_s=\begin{bmatrix}u&0\\s^{1/2}z&0\end{bmatrix}$ and computing $C_s^*C_s$, we obtain
\begin{equation*}
\Tr\big((u^*u+sP(z)^*P(z))^q\big)
\leq\Tr\big((u^*u+sz^*z)^q\big).
\end{equation*}
The same differentiation gives $\Tr(b_uP(z)^*P(z))\leq\Tr(b_uz^*z)$, which is the second estimate in \eqref{gp-separate-weighted-estimates}.

Define continuous positive semidefinite Hermitian forms on $\L^p(\cal M)$ by
\begin{equation}
\label{gp-local-weighted-forms}
h_u^\ell(x,y)=\Tr(y^*a_ux),
\qquad
h_u^r(x,y)=\Tr(y^*xb_u).
\end{equation}
They are linear in the first variable. By \eqref{gp-separate-weighted-estimates}, $P$ is contractive for each associated seminorm. After quotienting by its null space and completing, the induced projection is therefore orthogonal. Consequently,
\begin{equation*}
h_u^\ell(Px,y)=h_u^\ell(x,Py),
\qquad
h_u^r(Px,y)=h_u^r(x,Py),
\quad x,y \in \L^p(\cal M).
\end{equation*}
Using \eqref{gp-trace-adjoint} and the nondegeneracy of trace duality gives \eqref{gp-weighted-intertwining}. Taking $x\in Y$ proves the two range inclusions.
\end{proof}

\subsection{Gram densities for finite-rank projections}

\begin{prop}
\label{prop-gp-Gram-densities}
Suppose that $P$ has finite rank and $Y\neq\{0\}$. Choose a basis $(x_1,\ldots,x_d)$ of $Y$, and put
\begin{equation}
\label{gp-Gram-densities}
A_p
=\sum_{j=1}^d(x_jx_j^*)^{\frac{p-2}{2}},
\qquad
B_p
=\sum_{j=1}^d(x_j^*x_j)^{\frac{p-2}{2}}.
\end{equation}
Let $e=s_\ell(Y)$ and $f=s_r(Y)$. Then $A_p,B_p\in\L^r(\cal M)_+$, $s(A_p)=e$, $s(B_p)=f$, and
\begin{equation}
\label{gp-exact-density-ranges}
A_pY
=Y^\sharp
=YB_p.
\end{equation}
The multiplication maps $L_{A_p}\co x\mapsto A_px$ and $R_{B_p}\co x\mapsto xB_p$ are bounded from $\L^p(\cal M)$ into $\L^{p^*}(\cal M)$ and satisfy
\begin{equation}
\label{gp-global-intertwining}
P^\dagger L_{A_p}=L_{A_p}P,
\qquad
P^\dagger R_{B_p}=R_{B_p}P.
\end{equation}
Their restrictions to $Y$ are linear bijections onto $Y^\sharp$.
\end{prop}

\begin{proof}
Membership in $\L^r(\cal M)$ and the boundedness of the multiplication maps follow from Proposition~\ref{prop-gp-weighted-tangent-estimates} and H\"older's inequality. Taking a strictly positive power preserves supports, and the support of a finite sum of positive measurable operators is the join of their supports. Hence
\begin{equation*}
s(A_p)=\bigvee_{j=1}^d s_\ell(x_j)
=e,
\qquad
s(B_p)=\bigvee_{j=1}^d s_r(x_j)
=f.
\end{equation*}
Summing \eqref{gp-weighted-intertwining} over the basis gives \eqref{gp-global-intertwining}, and therefore $A_pY\subset Y^\sharp$ and $YB_p\subset Y^\sharp$.

We verify injectivity without using bounded inverses of the densities. If $A_py=0$ for $y\in Y$, put $e_n=\mathbf1_{[1/n,n]}(A_p)$ and $g_n(t)=t^{-1}\mathbf1_{[1/n,n]}(t)$, with $g_n(0)=0$. Then $g_n(A_p)\in\cal M$ and
\begin{equation*}
e_ny=g_n(A_p)(A_py)=0.
\end{equation*}
Since $e_n\uparrow e$ strongly, $e_ny\to ey=y$ in $\L^p(\cal M)$, so $y=0$. Spectral truncation of $B_p$ gives the analogous conclusion for right multiplication. Finally, $\dim Y^\sharp=\dim Y=d$. The two injective maps from $Y$ into $Y^\sharp$ are therefore surjective, proving \eqref{gp-exact-density-ranges}.
\end{proof}

\subsection{Relative operators and modular realization}

\begin{thm}
\label{thm-gp-relative-modular-structure}
Retain the assumptions and notation of Proposition~\ref{prop-gp-Gram-densities}. The formulas
\begin{equation}
\label{gp-global-Hilbert-forms}
H_p^\ell(x,y)=\Tr(y^*A_px),
\qquad
H_p^r(x,y)=\Tr(y^*xB_p),
\quad x,y\in Y,
\end{equation}
define positive definite Hermitian forms on $Y$. Define
\begin{equation}
\label{gp-relative-operator}
\Delta_p=(L_{A_p}|_Y)^{-1}(R_{B_p}|_Y),
\end{equation}
where both restricted multiplication maps have codomain $Y^\sharp$. Then $\Delta_p$ is a positive invertible selfadjoint operator on $(Y,H_p^\ell)$ and
\begin{equation}
\label{gp-relative-identities}
A_p\Delta_p(y)=yB_p,
\qquad
H_p^r(x,y)=H_p^\ell(x,\Delta_py),
\quad x,y\in Y.
\end{equation}
Its imaginary powers satisfy
\begin{equation}
\label{gp-imaginary-powers}
\Delta_p^{\i t}(y)=A_p^{-\i t}yB_p^{\i t},
\quad y \in Y, t \in \R.
\end{equation}
Here imaginary powers of $A_p$ and $B_p$ are taken on their supports, and are unitaries of $e\cal Me$ and $f\cal Mf$, respectively. In particular, these maps form a one-parameter group of surjective complete isometries of $Y$ for its inherited $\L^p$ operator space structure. They are also unitary for both forms in \eqref{gp-global-Hilbert-forms}.

Consider the corner von Neumann algebra
\begin{equation}
\label{gp-linking-density}
\cal R=
\begin{bmatrix}e\cal Me&e\cal Mf\\f\cal Me&f\cal Mf\end{bmatrix},
\qquad
D_p
= \begin{bmatrix}
A_p^r&0\\
0&B_p^r
\end{bmatrix}.
\end{equation}
Equip $\cal R$ with the restriction $\Tr_{\cal R}$ of $\Tr_2\ot\Tr$. Then $D_p\in\L^1(\cal R)_+$ has support $1_{\cal R}$. Consequently,
\begin{equation}
\label{gp-linking-state}
c_p
=\Tr_{\cal R}(D_p),
\qquad
\psi_p(X)
=c_p^{-1}\Tr_{\cal R}(D_pX),
\quad X \in \cal R,
\end{equation}
define a finite positive constant $c_p$ and a normal faithful state $\psi_p$.

For $y\in e\L^p(\cal M)f$, write $\jmath(y)=\begin{bmatrix}0&y\\
0&0\end{bmatrix}$. Let $\sigma_{t,p}^{\psi_p}$ denote the canonical extension of the modular automorphism $\sigma_t^{\psi_p}$ to the tracial space $\L^p(\cal R)$. Then
\begin{equation}
\label{gp-relative-group-modular}
\jmath(\Delta_p^{\i t}y)
=\sigma_{-\frac{t}{r},p}^{\psi_p}(\jmath(y)),
\quad y \in Y, t\in \R.
\end{equation}
Thus the time rescaling is $-\frac{t}{r}=-(p-2)\frac{t}{p}$.

Finally, let $\Delta_{\psi_p}$ be the standard modular operator on $\L^2(\cal R,\Tr_{\cal R})$ associated with the vector $c_p^{-\frac{1}{2}}D_p^{\frac{1}{2}}$. The map
\begin{equation}
\label{gp-Hilbert-embedding}
W_p \co (Y,H_p^\ell) \to \L^2(\cal R,\Tr_{\cal R}),
\qquad
W_p(y)
=\jmath(A_p^{1/2}y),
\end{equation}
is a linear isometry. Its finite-dimensional range $\cal K_p=W_p(Y)$ reduces $\Delta_{\psi_p}$, and
\begin{equation}
\label{gp-standard-modular-identification}
W_p\Delta_pW_p^{-1}
=\big(\Delta_{\psi_p}|_{\cal K_p}\big)^{-\frac{1}{r}}.
\end{equation}
No finite-dimensionality assumption on the ambient algebra $\cal M$ is needed.
\end{thm}

\begin{proof}
H\"older's inequality makes both forms in \eqref{gp-global-Hilbert-forms} well-defined, and
\begin{equation*}
H_p^\ell(y,y)=\|A_p^{\frac{1}{2}}y\|_2^2,
\qquad
H_p^r(y,y)=\|yB_p^{\frac{1}{2}}\|_2^2.
\end{equation*}
Since $Y\subset e\L^p(\cal M)f$, $s(A_p)=e$ and $s(B_p)=f$, spectral truncation as in the preceding proof shows that these quantities vanish only when $y=0$. Thus both forms are positive definite.

By its definition, $\Delta_p$ is a linear bijection and satisfies the first identity in \eqref{gp-relative-identities}. For $x,y \in Y$, we obtain
\begin{equation*}
H_p^\ell(x,\Delta_py)
=\Tr\big((A_p\Delta_py)^*x\big)
=\Tr\big((yB_p)^*x\big)
=\Tr(y^*xB_p)
=H_p^r(x,y).
\end{equation*}
The Hermitian property and positive definiteness of $H_p^r$ imply that $\Delta_p$ is selfadjoint and positive for $H_p^\ell$. Since $Y$ is finite-dimensional and $\Delta_p$ is invertible, it has a basis of eigenvectors with strictly positive eigenvalues.

We next prove \eqref{gp-imaginary-powers}, taking particular care with possibly unbounded densities. Let $y\in Y$ be an eigenvector, with $\Delta_py=\lambda y$ and $\lambda>0$. Then
\begin{equation}
\label{gp-eigenvector-intertwining}
\lambda A_py=yB_p.
\end{equation}
All these products belong to $\L^{p^*}(\cal M)$. Put $R=(\lambda A_p+e)^{-1}\in e\cal Me$ and $S=(B_p+f)^{-1}\in f\cal Mf$. Multiplying \eqref{gp-eigenvector-intertwining} after adding $y$ to both sides by these bounded resolvents gives $Ry=yS$. Polynomial approximation therefore yields
\begin{equation}
\label{gp-continuous-functional-intertwining}
g(R)y=yg(S),
\quad g \in \C([0,1]).
\end{equation}
Functional calculi here are taken in the reduced algebras. Neither $R$ nor $S$ has a spectral atom at $0$ or $1$: the resolvents have trivial kernels on their supports, and $A_p$ and $B_p$ also have trivial kernels there.

Fix $t\in\R$. Choose continuous functions $\chi_n\co[0,1]\to[0,1]$ with compact support in $(0,1)$ and $\chi_n(s)\to1$ for any $0<s<1$. Define $g_n(s)=\chi_n(s)((1-s)/s)^{\i t}$ on $(0,1)$ and $g_n(0)=g_n(1)=0$. The spectral theorem gives bounded strong* convergence
\begin{equation*}
g_n(R)\longrightarrow(\lambda A_p)^{\i t},
\qquad
g_n(S)\longrightarrow B_p^{\i t}.
\end{equation*}
Bounded strong* convergence of multipliers implies convergence in $\L^p$ after multiplication by a fixed $\L^p$-vector. For completeness, this follows by approximating that vector in $\L^p$ by bounded operators $z$ with finite trace support and using
\begin{equation*}
\|T_nz\|_p^p\leq\|T_nz\|_\infty^{p-2}\|T_nz\|_2^2
\end{equation*}
when $T_n\to0$ strongly and is uniformly bounded; the right multiplication statement follows by taking adjoints. Passing to the limit in \eqref{gp-continuous-functional-intertwining} gives $(\lambda A_p)^{\i t}y=yB_p^{\i t}$, and hence
\begin{equation*}
A_p^{-\i t}yB_p^{\i t}=\lambda^{\i t}y=\Delta_p^{\i t}y.
\end{equation*}
Applying this to an eigenbasis proves \eqref{gp-imaginary-powers} on all of $Y$.

The maps in \eqref{gp-imaginary-powers} and their inverses are restrictions of left and right unitary multiplications on the rectangular corner. They are therefore complete isometries for the inherited operator space structure. Unitarity for $H_p^\ell$ follows from the spectral theorem. The relation $H_p^r(x,y)=H_p^\ell(x,\Delta_py)$ and commutation of $\Delta_p$ with its imaginary powers imply unitarity for $H_p^r$.

The polar partial isometries of elements of $Y$ belong to $e\cal Mf$, and the joins of their left and right supports are $e$ and $f$. Thus $e\cal Mf$ has full left and right supports in the reduced algebras, so $\cal R$ is its linking von Neumann algebra. Since $A_p,B_p\in\L^r(\cal M)$, we have $D_p\in\L^1(\cal R)_+$. Its support is $\operatorname{diag}(e,f)=1_{\cal R}$, so \eqref{gp-linking-state} defines a normal faithful state. By the tracial modular formula, see \cite{Tak03}, its modular automorphism group is
\begin{equation}
\label{gp-modular-action}
\sigma_t^{\psi_p}(X)=D_p^{\i t}XD_p^{-\i t},
\quad X\in\cal R, t\in\R.
\end{equation}
The scalar normalization $c_p$ cancels from this formula. The same unitary conjugation defines its isometric extension to any tracial $\L^q(\cal R)$, $1\leq q<\infty$. On the upper-right corner, it gives
\begin{equation*}
\sigma_{t,p}^{\psi_p}(\jmath(y))
=\jmath\big(A_p^{\i rt}yB_p^{-\i rt}\big),
\quad y \in e\L^p(\cal M)f.
\end{equation*}
Combining this with \eqref{gp-imaginary-powers} proves \eqref{gp-relative-group-modular}.

The map $W_p$ is well-defined by $1/2=1/p+1/(2r)$ and H\"older's inequality. Moreover,
\begin{equation*}
\langle W_px,W_py\rangle_{\L^2(\cal R)}
=\Tr(y^*A_px)=H_p^\ell(x,y),
\quad x,y \in Y,
\end{equation*}
so it is an isometry. The standard modular group on the tracial Hilbert space is $\Delta_{\psi_p}^{\i t}\xi=D_p^{\i t}\xi D_p^{-\i t}$. Consequently,
\begin{equation}
\label{gp-Hilbert-modular-intertwining}
\Delta_{\psi_p}^{\i t}W_py
=\jmath\big(A_p^{1/2}A_p^{\i rt}yB_p^{-\i rt}\big)
=W_p\Delta_p^{-\i rt}y.
\end{equation}
Thus $\cal K_p$ is invariant under the modular unitaries and their inverses, and hence reduces their selfadjoint generator and $\Delta_{\psi_p}$. Restricting \eqref{gp-Hilbert-modular-intertwining} to this finite-dimensional subspace and applying the spectral theorem gives
\begin{equation*}
\Delta_{\psi_p}|_{\cal K_p}=W_p\Delta_p^{-r}W_p^{-1},
\end{equation*}
which is equivalent to \eqref{gp-standard-modular-identification}.
\end{proof}

\begin{cor}
\label{cor-gp-modular-covariance}
Under the assumptions of Theorem~\ref{thm-gp-relative-modular-structure}, put $U_t=A_p^{-\i t}+(1-e)$ and $V_t=B_p^{\i t}+(1-f)$. Then
\begin{equation}
\label{gp-modular-covariance}
P(U_txV_t)
=U_tP(x)V_t,
\quad x \in \L^p(\cal M), t \in\R.
\end{equation}
\end{cor}

\begin{proof}
The map $\cal U_t\co x\mapsto U_txV_t$ is a surjective complete isometry of $\L^p(\cal M)$, and \eqref{gp-imaginary-powers} gives $\cal U_t(Y)=Y$. Thus $\cal U_t^{-1}P\cal U_t$ and $P$ are contractive projections onto the same nonzero subspace. Since $\L^p(\cal M)$ is smooth, Proposition~\ref{prop-Sullivan} gives $\cal U_t^{-1}P\cal U_t=P$, proving the assertion.
\end{proof}

\begin{remark} \normalfont
If $\cal M$ is finite-dimensional, then $A_p$ and $B_p$ are invertible in their reduced algebras, and \eqref{gp-relative-operator} becomes $\Delta_p(y)=A_p^{-1}yB_p$. For a general semifinite algebra, the definition through the restricted multiplication maps avoids any assumption that these inverses are bounded or even $\Tr$-measurable. Likewise, the embedding $W_p$ is used because $Y$ need not be contained in the original tracial $\L^2$-space.

For $p=4$, we have $r=2$, $A_4=\sum_jx_jx_j^*$ and $B_4=\sum_jx_j^*x_j$, so the modular time rescaling is $-\frac{t}{2}$. The preceding results establish weighted tangent and modular structures for any $p > 2$. The next subsection shows that, in the finite-rank case, these weighted identities lead to a ternary description of the range and to complete contractivity.
\end{remark}

\subsection{Ternary structure and complete contractivity in the finite-rank case}
\label{sec-fr-arbitrary-exponents}

We now complete the weighted construction for finite-rank projections. The appropriate ternary structure is obtained after a change of density: for $p\neq4$, one cannot in general expect the unweighted inclusion $YY^*Y\subset Y^\sharp$.

We shall use Propositions~\ref{prop-gp-weighted-tangent-estimates} and \ref{prop-gp-Gram-densities}, together with Theorem~\ref{thm-gp-relative-modular-structure}, for finite spanning families rather than just bases. Their proofs remain valid with this modification: the support identities only require that the family span $Y$, and the dimension argument uses $\dim Y^\sharp =\dim Y$, not the length of the family.

\begin{lemma}
\label{lem-fr-bounded-normalization}
Let $1<r<\infty$, let $A\in\L^r(\cal M)_+$ and put $e=s(A)$. If $a\in\L^r(\cal M)_+$ satisfies $a\leq A$, there exists a unique $C\in e\cal Me$ with $0\leq C\leq e$ such that $a= A^{\frac{1}{2}}CA^{\frac{1}{2}}$.
\end{lemma}

\begin{proof}
Put $e_n=\mathbf1_{[\frac{1}{n},n]}(A)$ and let $R_n=A^{-\frac{1}{2}}e_n$ be defined by bounded Borel functional calculus. The measurable operator $C_n=R_naR_n$ satisfies $0\leq C_n\leq e_n$, so $C_n$ belongs to $e\cal Me$. If $m\leq n$, then $e_mC_ne_m=C_m$. A weak* cluster point $C$ of the bounded sequence $(C_n)$ therefore satisfies $0 \leq C\leq e$ and $e_nCe_n=C_n$ for any $n$. It follows that
\begin{equation*}
e_nA^{\frac{1}{2}}CA^{\frac{1}{2}}e_n=e_nae_n.
\end{equation*}
Both $A^{\frac{1}{2}}CA^{\frac{1}{2}}$ and $a$ belong to $\L^r(\cal M)$, and compression by $e_n\uparrow e$ converges to the identity in $\L^r$ on operators supported by $e$. Hence $a=A^{\frac{1}{2}}CA^{\frac{1}{2}}$. The same compressions recover $e_nCe_n$ from $a$, proving uniqueness.
\end{proof}

\begin{lemma}
\label{lem-fr-invariant-algebra-factorization}
Let $\cal A\subset e\cal Me$ be a finite-dimensional von Neumann algebra with unit $e$. Suppose that $A\in\L^r(\cal M)_+$ has support $e$ and satisfies $A^{\i t}\cal A A^{-\i t}=\cal A$ for any $t\in\R$. Then there exist a positive invertible element $d\in\cal A$ and a positive operator $c\in\L^r(\cal M)$ with support $e$, affiliated with $\cal A'\cap e\cal Me$, such that $A=cd=dc$.
\end{lemma}

\begin{proof}
The restrictions of $\operatorname{Ad}(A^{\i t})$ form a norm-continuous one-parameter automorphism group of $\cal A$. They fix its center pointwise. On each full matrix summand, the generator is an inner $*$-derivation. Consequently, there exists $k=k^*\in\cal A$ such that
\begin{equation*}
A^{\i t}xA^{-\i t}
=\e^{\i tk}x\e^{-\i tk},
\quad x \in \cal A, t \in \R.
\end{equation*}
Set $d=\e^k$. Since the displayed automorphisms fix $d$, this element commutes with the spectral projections of $A$. Thus $c=Ad^{-1}$ is positive, belongs to $\L^r(\cal M)$, and has support $e$. Moreover, $c^{\i t}=A^{\i t}d^{-\i t}$ commutes with $\cal A$ for any $t$. The spectral theorem shows that $c$ is affiliated with $\cal A'\cap e\cal Me$.
\end{proof}

\begin{thm}
\label{thm-fr-general-p-TRO}
Let $2<p<\infty$ and let $P\co\L^p(\cal M) \to \L^p(\cal M)$ be a finite-rank $2$-contractive projection with nonzero range $Y$. Put $Y^\sharp=\Ran P^\dagger$, $e=s_\ell(Y)$ and $f=s_r(Y)$. There exist a finite-dimensional $\W^*$-$\TRO$ $\cal V\subset e\cal Mf$ with left and right supports $e,f$, and an operator $h\in\L^p(\cal M)_+$ with $s(h)=e$, such that
\begin{equation}
\label{fr-weighted-range-structure}
Y
=h\cal V,
\qquad
Y^\sharp
=h^{p-1}\cal V.
\end{equation}
The operator $h$ commutes with the left von Neumann algebra $\mathrm M(\cal V)=\operatorname{span}(\cal V\cal V^*)$. Moreover, $P$ is completely contractive.
\end{thm}

\begin{proof}
Put $\alpha=(p-2)/2$ and $r=p/(p-2)$. Choose a basis $u_1,\ldots,u_m$ of $Y$, and let $\cal F$ be the finite family consisting of these vectors and the vectors $u_j+\i^k u_l$, where $j<l$ and $k\in\{0,1,2,3\}$. Define
\begin{equation}
\label{fr-enlarged-Gram-densities}
A=\sum_{u\in\cal F}(uu^*)^\alpha,
\qquad
B=\sum_{u\in\cal F}(u^*u)^\alpha.
\end{equation}
The finite-spanning-family versions of Proposition~\ref{prop-gp-Gram-densities} and Theorem~\ref{thm-gp-relative-modular-structure} give
\begin{equation}
\label{fr-weighted-invariance}
A,B\in\L^r(\cal M)_+,
\quad s(A)=e,
\quad s(B)=f,
\quad A^{-\i t}YB^{\i t}=Y
\quad(t\in\R).
\end{equation}

\smallskip
\noindent\textit{The weighted Hilbert projection.}
Let $\cal H=e\L^2(\cal M)f$ and define $J(x)=A^{1/2}x$ on $e\L^p(\cal M)f$. H\"older's inequality makes this map well-defined. It is injective and has dense range in $\cal H$. To check density, let $z\in e\cal Mf\cap\L^2(\cal M)$ and put $e_n=\mathbf1_{[1/n,n]}(A)$. Since $p>2$, the bounded $\L^2$-vector $z$ belongs to $\L^p$. Thus $x_n=A^{-1/2}e_nz$ belongs to $e\L^p(\cal M)f$, and $J(x_n)=e_nz\to z$ in $\L^2$. Such vectors $z$ are dense in $\cal H$.

Summing the left inequalities in \eqref{gp-separate-weighted-estimates} over $\cal F$ gives $\|J(Px)\|_2\leq\|Jx\|_2$. Consequently,
\begin{equation}
\label{fr-Hilbert-projection}
QJx=JPx,
\quad x\in e\L^p(\cal M)f,
\end{equation}
defines a contractive projection on a dense subspace of $\cal H$. Its extension is an orthogonal projection with finite-dimensional range $\cal K=J(Y)$.

The left support of $\cal K$ is $e$. Indeed, suppose that $b\in e\cal Me$ satisfies $bA^{\frac{1}{2}}u_j=0$ for any $j$. Put $T=bA^{\frac{1}{2}} \in \L^{2r}(\cal M)$. Spectral truncation of $|u_j|$ in the equality $Tu_j=0$ gives $T s_\ell(u_j)=0$. Since $\bigvee_j s_\ell(u_j)=e$ and $T=Te$, it follows that $T=0$. Spectral truncation of $A$ then gives $b=0$.

\smallskip
\noindent\textit{A finite-dimensional algebra of normalized weights.}
For $u\in\cal F$, put $a_u=(uu^*)^\alpha$. Since $0\leq a_u\leq A$, Lemma~\ref{lem-fr-bounded-normalization} gives a positive contraction $C_u\in e\cal Me$ such that
\begin{equation}
\label{fr-normalized-weights}
a_u=A^{1/2}C_uA^{1/2}.
\end{equation}
For $t\in\R$, define $u_t=A^{-\i t}uB^{\i t}\in Y$ and $C_u(t)=A^{-\i t}C_uA^{\i t}$. Unitary covariance of functional calculus gives
\begin{equation*}
(u_tu_t^*)^\alpha
=A^{-\i t}a_uA^{\i t}
=A^{\frac{1}{2}}C_u(t)A^{\frac{1}{2}}.
\end{equation*}
The left weighted form associated with $u_t$ is therefore
\begin{equation*}
h_{u_t}^\ell(x,z)
=\langle C_u(t)Jx,Jz\rangle_{\cal H}.
\end{equation*}
By \eqref{gp-weighted-intertwining}, $P$ is selfadjoint for this form. Combining this with \eqref{fr-Hilbert-projection} and the selfadjointness of $Q$ shows that
\begin{equation}
\label{fr-normalized-commutation}
QL_{C_u(t)}
=L_{C_u(t)}Q
\quad\text{on }\cal H.
\end{equation}
Let $\cal A$ be the von Neumann subalgebra of $e\cal Me$ generated by $e$ and all the operators $C_u(t)$, for $u\in\cal F$ and $t\in\R$. Equation~\eqref{fr-normalized-commutation} shows that left multiplication by $\cal A$ leaves $\cal K$ reducing. Its restriction to $\cal K$ is a unital $*$-representation. This representation is faithful because $\cal K$ has left support $e$. Hence $\cal A$ embeds into $\B(\cal K)$ and is finite-dimensional. Its definition also gives $A^{\i t}\cal A A^{-\i t}=\cal A$ for any $t \in \R$.

Apply Lemma~\ref{lem-fr-invariant-algebra-factorization} to write $A=cd$, where $d \in \cal A$ is positive and invertible and $c\in\L^r(\cal M)_+$ has support $e$ and commutes with $\cal A$. Put
\begin{equation}
\label{fr-commuting-density}
h
=c^{\frac{1}{p-2}} \in \L^p(\cal M)_+.
\end{equation}
For $u \in \cal F$, let $k_u=d^{\frac{1}{2}}C_ud^{\frac{1}{2}} \in \cal A_+$. Then \eqref{fr-normalized-weights} becomes $a_u=ck_u$. Since $c$ commutes with $k_u$, joint functional calculus gives
\begin{equation}
\label{fr-recovering-bounded-vectors}
|u^*|
=(a_u){\frac{1}{p-2}}
=h k_u^{\frac{1}{p-2}}.
\end{equation}
Write the left polar decomposition of $u$ as $u=|u^*|w_u$, and put $v_u=k_u^{\frac{1}{p-2}}w_u\in e\cal Mf$. Then
\begin{equation}
\label{fr-bounded-vector-products}
u=hv_u,
\qquad
v_uv_u^*
=k_u^{\frac{2}{p-2}} \in \cal A.
\end{equation}
Here $s(|u^*|) = s(k_u)$, since $s(h)=e$ and $h$ commutes with $k_u$.

\smallskip
\noindent\textit{The deweighted range is a TRO.}
Let $\cal V=\operatorname{span}\{v_{u_1},\ldots,v_{u_m}\}$. Multiplication by $h$ is injective on $e\cal Mf$, so $Y=h\cal V$ and $\dim\cal V=m$. The same injectivity shows that
\begin{equation*}
v_{u_j+\i^k u_l}=v_{u_j}+\i^k v_{u_l}.
\end{equation*}
In particular, every $v_u$ occurring in \eqref{fr-bounded-vector-products} belongs to $\cal V$. Since $A=cd$ and $\frac{r}{2}=\frac{1}{2}+\frac{1}{p-2}$, we have
\begin{equation*}
\cal K=A^{\frac{1}{2}}Y
=c^{r/2}d^{\frac{1}{2}}\cal V.
\end{equation*}
For $b\in\cal A$ and $v\in\cal V$, invariance of $\cal K$ gives some $v'\in\cal V$ with
\begin{equation*}
bc^{\frac{r}{2}}d^{\frac{1}{2}}v
=c^{\frac{r}{2}}d^{\frac{1}{2}}v'.
\end{equation*}
The operator $c^{\frac{r}{2}}$ commutes with $\cal A$ and has support $e$, so cancellation gives $v'=d^{-\frac{1}{2}}bd^{\frac{1}{2}}v$. Since $d^{-1/2}\cal A d^{1/2}=\cal A$, this proves $\cal A\cal V\subset\cal V$.

The diagonal products $v_{u_j}v_{u_j}^*$ belong to $\cal A$ by \eqref{fr-bounded-vector-products}. For $j<l$, polarization gives
\begin{equation*}
v_{u_j}v_{u_l}^*
=\frac14\sum_{k=0}^3\i^k
\big(v_{u_j}+\i^k v_{u_l}\big)
\big(v_{u_j}+\i^k v_{u_l}\big)^*
\in\cal A.
\end{equation*}
Taking adjoints gives the remaining products. Thus $\cal V\cal V^*\subset\cal A$, and therefore
\begin{equation*}
\cal V\cal V^*\cal V
\subset\cal A\cal V
\subset\cal V.
\end{equation*}
The finite-dimensional space $\cal V$ is weak* closed, so it is a $\W^*$-$\TRO$. The polar decompositions in \eqref{fr-bounded-vector-products} give the same left and right supports for $u_j$ and $v_{u_j}$, hence $s_\ell(\cal V)=e$ and $s_r(\cal V)=f$. Finally, $h$ commutes with $\cal A$, and therefore with $\mathrm M(\cal V)\subset\cal A$.

\smallskip
\noindent\textit{Power duality at every matrix level.}
For a bounded operator $v=w|v|$ and $1<s<\infty$, write $F_s(v)=w|v|^{s-1}$. This notation does not require $v$ to belong to the ambient tracial $\L^s$-space. For $v\in\cal V$, the commutation just obtained gives
\begin{equation}
\label{fr-power-duality-factorization}
\frak j_p(hv)
=\big((hv)(hv)^*\big)^{\frac{p-2}{2}}hv
=h^{p-1}(vv^*)^{\frac{p-2}{2}}v
=h^{p-1}F_p(v).
\end{equation}
Functional calculus in the finite-dimensional linking algebra of $\cal V$ shows that $F_s(\cal V)\subset\cal V$ for any $1 < s < \infty$. The identities $F_{p^*}F_p=\Id$ and $F_pF_{p^*}=\Id$ consequently give $F_p(\cal V)=\cal V$. Since $P$ is contractive, \eqref{power-duality-map-range-equality}, interpreted under trace duality, and \eqref{fr-power-duality-factorization} yield $Y^\sharp=\frak j_p(Y)=h^{p-1}\cal V$.

Let $n\geq1$, put $h_n=\Id_{\M_n}\ot h$, and write $X=h_nV$ for $X\in\M_n(Y)$ and $V\in\M_n(\cal V)$. The same computation, now in the matrix amplification, gives
\begin{equation}
\label{fr-all-matrix-duality}
\frak j_p(X)=h_n^{p-1}F_p(V)
\in\M_n(h^{p-1}\cal V)
=\M_n(Y^\sharp)
=\Ran P_n^\dagger.
\end{equation}
For arbitrary $Z\in\L^p(\M_n\overline\otimes\cal M)$, set $X=P_nZ$. Since $Z-X \in \ker P_n$, equation~\eqref{fr-all-matrix-duality} and H\"older's inequality imply
\begin{equation*}
\|X\|_p^p
=\Tr\big(\frak j_p(X)^*Z\big)
\leq\|X\|_p^{p-1}\|Z\|_p.
\end{equation*}
The trace here is the amplified trace. Hence $\|P_nZ\|_p\leq\|Z\|_p$, including the trivial case $X=0$. Since $n$ is arbitrary, $P$ is completely contractive.
\end{proof}

\begin{cor}
\label{cor-fr-symmetric-Gram-deweighting}
Under the assumptions of Theorem~\ref{thm-fr-general-p-TRO}, choose any basis $x_1,\ldots,x_m$ of $Y$ and define $A_p,B_p$ by \eqref{gp-Gram-densities}. Put $\gamma=1/(2(p-2))$. Then the $\TRO$ constructed above satisfies
\begin{equation}
\label{fr-symmetric-Gram-deweighting}
Y=A_p^\gamma\cal V B_p^\gamma.
\end{equation}
More precisely, any $y \in Y$ has a unique representation $y=A_p^\gamma vB_p^\gamma$ with $v \in \cal V$. Thus $A_p^{-\gamma}YB_p^{-\gamma}$ can be interpreted as a space of bounded solutions, without assuming that the negative powers are measurable operators. This space is independent of the chosen basis.
\end{cor}

\begin{proof}
Let $Y=h\cal V$ be as in Theorem~\ref{thm-fr-general-p-TRO}. The right algebra $\mathrm N(\cal V)=\operatorname{span}(\cal V^*\cal V)$ is finite-dimensional with unit $f$. Choose a basis $v_1,\ldots,v_m$ of $\cal V$, put $S=\sum_jv_j^*v_j$, and define $w_j=v_jS^{-1/2}\in\cal V$, taking the inverse in $\mathrm N(\cal V)$. Then $\sum_jw_j^*w_j=f$. Set
\begin{equation*}
k=\sum_jw_j^*h w_j\in\L^p(\cal M)_+.
\end{equation*}
For $v\in\cal V$, commutation of $h$ with $\mathrm M(\cal V)$ gives
\begin{equation*}
vk=\sum_jvw_j^*h w_j
=h\sum_jvw_j^*w_j=hv.
\end{equation*}
The positivity of $h$ and $s(h)=e$, together with $\sum_jw_j^*w_j=f$, imply $s(k)=f$. Taking adjoints in $hv=vk$ and multiplying by elements of $\cal V$ shows that $k$ commutes with $\mathrm N(\cal V)$. Resolvent intertwining, followed by spectral functional calculus, also gives $h^{1/2}v=vk^{1/2}$.

Write the chosen basis of $Y$ as $x_j=hz_j=z_jk$, with $z_1,\ldots,z_m$ a basis of $\cal V$. Then
\begin{equation*}
A_p
=h^{p-2}S_\ell,
\qquad
B_p
=k^{p-2}S_r,
\end{equation*}
where $S_\ell=\sum_j(z_jz_j^*)^{(p-2)/2}\in\mathrm M(\cal V)$ and $S_r=\sum_j(z_j^*z_j)^{(p-2)/2}\in\mathrm N(\cal V)$. These positive elements have supports $e,f$ and are invertible in their respective finite-dimensional algebras. Therefore
\begin{equation*}
A_p^\gamma\cal V B_p^\gamma
=h^{\frac{1}{2}}S_\ell^\gamma\cal V S_r^\gamma k^{\frac{1}{2}}
=h^{\frac{1}{2}}\cal V k^{\frac{1}{2}}
=h\cal V=Y.
\end{equation*}
Here $\mathrm M(\cal V)\cal V\subset\cal V$ and $\cal V\mathrm N(\cal V)\subset\cal V$ imply $S_\ell^\gamma\cal V S_r^\gamma=\cal V$. Since $s(A_p)=e$ and $s(B_p)=f$, the map $v \mapsto A_p^\gamma vB_p^\gamma$ is injective on $e\cal Mf$, by spectral truncation. This proves uniqueness and the asserted interpretation. The same space $\cal V$ works for every basis, proving the final assertion.
\end{proof}

\begin{cor}
\label{cor-fr-all-exponents-complete-contractivity}
Let $1 < p < \infty$ with $p \neq 2$. Any finite-rank $2$-contractive projection on the tracial noncommutative space $\L^p(\cal M)$ is completely contractive.
\end{cor}

\begin{proof}
The zero projection is immediate, and Theorem~\ref{thm-fr-general-p-TRO} handles $p > 2$. If $1<p<2$, then $p^* > 2$ and $P^\dagger$ is a finite-rank $2$-contractive projection on $\L^{p^*}(\cal M)$. The theorem makes $P^\dagger$ completely contractive. Since $(P_n)^\dagger=(P^\dagger)_n$ and $\|P_n\|=\|(P^\dagger)_n\|$ for any integer $n$, the same follows for $P$.
\end{proof}

\begin{remark} \normalfont
\label{rem-fr-scope-and-unweighted-products}
The finite-rank hypothesis is used when the algebra of normalized weights is represented faithfully on the finite-dimensional space $\cal K$. It makes that algebra finite-dimensional, so its modular action can be implemented internally. The argument does not prove complete contractivity of an arbitrary infinite-rank $2$-contractive projection.

The unweighted inclusion $YY^*Y \subset Y^\sharp$ is false for $p\neq4$, even in rank one. Indeed, on $\L^p(\mathbb C^2)=\ell_p^2$, the projection
\begin{equation*}
P(x_1,x_2)
=\frac{x_1+2^{p-1}x_2}{1+2^p}(1,2)
\end{equation*}
is completely contractive. At matrix level $n$, this follows from the triangle inequality in $S_n^p$ and the scalar H\"older inequality applied to the two matrix norms. Its range is $Y=\mathbb C(1,2)$, whereas $Y^\sharp=\mathbb C(1,2^{p-1})$ and $YY^*Y=\mathbb C(1,8)$. These last two lines coincide only when $p=4$. Formula~\eqref{fr-weighted-range-structure}, rather than an unweighted cubic inclusion, is the appropriate replacement for general $p$.
\end{remark}


\paragraph{Competing interests} The author declares that he has no competing interests.

\paragraph{Data availability} No data sets were generated during this study.

{\footnotesize

\vspace{0.2cm}

\noindent C\'edric Arhancet\\ 
\noindent 6 rue Didier Daurat, 81000 Albi, France\\
URL: \href{http://sites.google.com/site/cedricarhancet}{https://sites.google.com/site/cedricarhancet}\\
cedric.arhancet@protonmail.com\\
ORCID: 0000-0002-5179-6972 

}

\end{document}